\documentclass{amsart}
\usepackage{mathrsfs}

\usepackage{soul}
\usepackage{ulem}

\usepackage{lineno}

\usepackage{hyperref}

\usepackage{dsfont}
\usepackage{here}
\usepackage{amsmath}
\usepackage{verbatim}
\usepackage{tikz-cd} 

\usepackage{graphicx, psfrag, epstopdf}

\usepackage{mathrsfs}

\usepackage{amssymb}
\usepackage{graphicx}

\usepackage[all,cmtip]{xy}

\usepackage{amsfonts}

\usepackage{soul}

\usepackage{color}
\usepackage{yfonts}

\usepackage{paralist}

\usepackage{stmaryrd}

\usepackage{amsxtra}

\usepackage{enumitem}
\usepackage{mathtools}

\def\cU{\mathcal U}
\def\cV{\mathcal V}

\def\cA{\mathcal A}

\def\cI{\mathcal I}

\newcommand{\norm}[1]{\| #1\|}

\def\cJ{          \mathcal J}

\def\cA{          \mathcal A}
\def\cB{          \mathcal B}
\def\cC{          \mathcal C}
\def\cD{          \mathcal D}
\def\cF{          \mathcal F}
\def\cG{          \mathcal G}

\def\cN{          \mathcal N}

\def\Eu{           E^{u}}
\def\Es{           E^{s}}
\def\Ec{           E^{c}}
\def\Ecs{           E^{cs}}
\def\Ecu{           E^{cu}}
\def\cW{          \mathcal W}
\def\cWu{          \mathcal W^{u}}
\def\cWs{          \mathcal W^{s}}
\def\cWc{          \mathcal W^{c}}
\def\cWcu{          \mathcal W^{cu}}
\def\cWcs{          \mathcal W^{cs}}
\def\cPH{          \mathcal{PH}}

\def\lc{  \lceil}
\def\rc{  \rceil}

\let\cal\mathcal

\def \R{{\mathbb R}}

\def \Z{{\mathbb Z}}
\def \N{{\mathbb N}}

\newcommand{\ary}{\begin{eqnarray}}
\newcommand{\eary}{\end{eqnarray}}

\newcommand{\aryst}{\begin{eqnarray*}}
\newcommand{\earyst}{\end{eqnarray*}}

\newcommand{\enmt}{\begin{enumerate}}
\newcommand{\eenmt}{\end{enumerate}}

\newcommand{\T}{{\mathbb T}}
\newcommand{\prf}{{\begin{proof}}}
\newcommand{\epf}{{\end{proof}}}

\newcommand{\cP}{{\mathcal P}}

\DeclareMathOperator{\diff}{Diff}

\newtheorem{theo}{\sc Theorem}[section]
\newtheorem{prop}[theo]{\sc Proposition}

\newtheorem{lemma}[theo]{\sc lemma}

\newtheorem{cor}[theo]{\sc Corollary}
\newtheorem{claim}[theo]{\sc Claim}

\newtheorem{conj}[theo]{\sc Conjecture}

\theoremstyle{definition}

\def\bee{\begin{equation}}
\def\eee{\end{equation}}

\newtheorem{defi}[theo]{\sc Definition}
\newtheorem{construction}[theo]{\sc Construction} 
\newtheorem{defi-prop}[theo]{\sc Definition-Proposition}

\newtheorem*{convention}{\sc Convention}
\newtheorem{nota}[theo]{\sc Notation}

\theoremstyle{rema}

\newtheorem{rema}[theo]{\sc Remark}

\newtheorem*{questionA'}{\sc Question A'}

\newtheorem*{questionB'}{\sc Question B'}

\newtheorem{thm}{\sc Theorem}

\numberwithin{equation}{section}

\newcommand{\Diff}{\text{Diff}}

\newcommand{\pdvr}[2]
{\dfrac{\partial^{#2} #1}{\partial \theta^{#2_1} \partial r^{#2_2}}}
\newcommand{\pdvrs}[2]
{\partial^{#2} #1 /\partial \theta^{#2_1} \partial r^{#2_2}}
\begin{document}

\title[$C^{\lowercase{r}}$-prevalence of stable ergodicity]{$C^{r}$-prevalence of stable ergodicity for a class of partially hyperbolic systems}

\author[Martin Leguil]{Martin Leguil$^*$}
\thanks{$^*$M.L. was supported by the ERC project of Sylvain Crovisier, the NSERC Discovery Grant,  reference number 502617-2017, of Jacopo De Simoi,  l'allocation doctorale AMX, and by the Brazilian-French Network in Mathematics}
\address{Martin Leguil, Department of Mathematics, University of Toronto, 40 St George St., Toronto, ON, Canada M5S 2E4 / 
	CNRS-Laboratoire de Math\'ematiques d’Orsay, UMR 8628, Universit\'e Paris-Sud 11, Orsay Cedex 91405, France}
\email{martin.leguil@utoronto.ca}
\email{martin.leguil@math.u-psud.fr}

\author[Zhiyuan Zhang]{Zhiyuan Zhang$^\dagger$}
\address{Zhiyuan Zhang, CNRS, Institut Galil\'ee
	Université Paris 13
	99, avenue Jean-Baptiste Cl\'ement
	93430 - Villetaneuse}
\thanks{$^\dagger$Z.Z. was supported by the NSF under Grant No. DMS-1638352  and l'allocation doctorale de l'ENS}
\email{zhiyuan.zhang@math.univ-paris13.fr}


\maketitle

\begin{abstract}
We prove that for $r \in \N_{\geq 2} \cup \{\infty\}$, for any dynamically coherent, center bunched and strongly pinched volume preserving $C^r$ partially hyperbolic  diffeomorphism $f \colon X \to X$,  if either (1) its center foliation is uniformly compact, or (2) its center-stable and center-unstable foliations are of class $C^1$,  then there exists a $C^1$-open neighbourhood of $f$ in $\diff^r(X,\mathrm{Vol})$, in which stable ergodicity is $C^r$-prevalent in Kolmogorov's sense. In particular, we verify Pugh-Shub's stable ergodicity conjecture in this region. This also provides the first result that verifies the prevalence of stable ergodicity in the measure-theoretical sense. Our theorem applies to a large class of algebraic systems. As applications, we give affirmative answers in the strongly pinched region to: 1. an open question of Pugh-Shub in \cite{PS}; 2. a generic version of an open question of  Hirsch-Pugh-Shub in \cite{HPS}; 
and 3. a generic version of an open question of Pugh-Shub in \cite{HPS}. 
\end{abstract}

\addtocontents{toc}{\protect\setcounter{tocdepth}{1}   }
 
\tableofcontents

\section{Introduction}

Smooth ergodic theory, that is, the study of statistical and geometric properties of measures invariant under a smooth transformation or flow, is a much studied subject in the modern dynamical systems. It has its root in Boltzmann's Ergodic Hypothesis in the study of gas particles back in the 19$^{th}$ century. Ever since Birkhoff's proof of his ergodic theorem, there has been a constant interest in understanding the genericity of ergodic systems. The pioneering work of A. Kolmogorov in the 1950's provided the first obstruction to the genericity of ergodicity for Hamiltonian systems. His idea was later developed into what is now known as the Kolmogorov-Arnold-Moser (KAM) theory. On the other hand, the work of E. Hopf, D. Anosov and Y. Sinai provided open sets of ergodic systems, known as Anosov systems, or uniformly hyperbolic systems.

\begin{defi}[Anosov diffeomorphisms]\label{defanosov}
Given a compact Riemannian manifold $X$, a diffeomorphism $f\in \mathrm{Diff}^1(X)$ is called \textit{uniformly hyperbolic} or \textit{Anosov} if there exists a continuous splitting $TX = \Es_f \oplus \Eu_f$ of the tangent bundle into $Df$-invariant subbundles
and constants $\bar{\chi}^{u}, \bar{\chi}^{s} > 0$ such that for any $x \in X$,  we have
\begin{alignat}{3}
\norm{Df(v_1)} &< e^{-\bar{\chi}^{s}}\|v_1\|, &\quad &\forall\, v_1 \in \Es_f(x)\backslash\{0\}, \label{anosovexponent1} \\
\norm{Df(v_2)} &> e^{\bar{\chi}^{u}}\|v_2\|, &\quad &\forall\, v_2 \in \Eu_f(x)\backslash\{0\}. \label{anosovexponent2}
\end{alignat}
\end{defi}
For the next nearly thirty years after Anosov-Sinai's work, uniformly hyperbolic systems remained the only examples where ergodicity was known to appear robustly, a property which is also called ``stable ergodicity": we say that a $C^2$ volume preserving diffeomorphism $f$ is stably ergodic if any volume preserving $g$ sufficiently close to $f$ in the $C^2$ topology is also ergodic. A breakthrough came when M. Grayson, C. Pugh and M. Shub \cite{GPS} gave the first non-hyperbolic example of a stably ergodic system, i.e., the time-one map of the geodesic flow on the unit tangent bundle of a surface of constant negative curvature. Such systems are special cases of partially hyperbolic systems, which are defined as follows.

\begin{defi}[Partially hyperbolic diffeomorphisms]\label{def partially hyperbolic}
Given a smooth Riemannian manifold $X$, a $C^1$ diffeomorphism $f\colon X \to X$ is called \textit{partially hyperbolic} if its $C^1$ norm is uniformly bounded and there exist a nontrivial continuous splitting of the tangent bundle into $Df$-invariant subbundles, $TX = \Es_f \oplus \Ec_f \oplus \Eu_f$, and continuous functions $\bar{\chi}^{u},  \bar{\chi}^{s}\colon X \to \R_{>0}$, $\bar{\chi}^{c}, \hat{\chi}^{c}\colon X \to \R$,  such that 
\begin{equation} \label{ph ineq 1}
 -\bar{\chi}^{s} <\bar{\chi}^{c }\leq \hat{\chi}^{c} < \bar{\chi}^{u},
\end{equation}
and for any $x\in X$, any $v_1 \in \Es_f(x) \backslash \{0\}$, $v_2 \in \Ec_f(x) \backslash \{0\}$, $v_3 \in \Eu_f(x) \backslash \{0\}$, we have
\begin{align} 
&\norm{D_x f(v_1)} < e^{-\bar{\chi}^{s}(x)}\norm{v_1}, \label{ph ineq 2}\\
e^{\bar{\chi}^c(x)}\norm{v_2} <\ &\norm{D_x f(v_2)}< e^{\hat{\chi}^{c}(x)}\norm{v_2}, \label{ph ineq 3}\\
e^{\bar{\chi}^{u}(x)}\norm{v_3} <\ &\norm{D_x f(v_3)}. \label{ph ineq 4}
\end{align}
We set $E^{cs}_f:=E^c_f \oplus E^s_f$ and $E^{cu}_f:=E^c_f \oplus E^u_f$. 
\end{defi}

Partially hyperbolic systems have served as the principal source for finding stably ergodic systems (for other examples, see \cite{BGP,BPP,T}). They also appear in the study of SRB measures, statistical mechanics, rigidity theory and homogeneous dynamics. Based on \cite{GPS} and other results, Pugh and Shub formulated the following fundamental conjecture: 
%

\begin{conj}[Pugh-Shub's Stable Ergodicity Conjecture, \cite{PS}]\label{stableergodicityconjecture}
Stable ergodicity is $C^r$-dense among the set of $C^r$ volume preserving partially hyperbolic diffeomorphisms on a compact connected manifold, for any integer $r\geq 2$.
\end{conj}

Since its introduction, this conjecture and related questions on stable ergodicity have been extensively studied, for instance in the following series of works \cite{ PS2, BW2, DW, RH, RHRHU, RHRHTU, BW, ABW, ACW}. We will later elaborate on the connections between them in Subsection \ref{subs sta erg}. We mention that Conjecture \ref{stableergodicityconjecture} is far from being solved: results directly related to Conjecture \ref{stableergodicityconjecture} are only known for $\dim E^c_f=1$.\footnote{The $C^1$ version of Conjecture \ref{stableergodicityconjecture} is proved in \cite{ACW}.} \\ 

Our main result (Theorem \ref{main thm 2}), will be given  in Section \ref{section main result}; as we will see,  we actually obtained 
 prevalence in Kolmogorov's sense, a notion which is much stronger  than density. 
In the next section, we will state Theorem \ref{main thm 2.1}, a corollary of Theorem \ref{main thm 2},  to help the reader understand the main features of our result.  

\subsection{Stable ergodicity and accessibility}\label{subs sta erg}

In \cite{PS2}, the authors proposed a route to prove the Stable Ergodicity Conjecture. They divided the conjecture into two parts, using a geometric notion originating in an argument due to E. Hopf \cite{Hopf}. 

Let $f$ be $C^r$ partially hyperbolic diffeomorphism of a smooth compact Riemannian manifold $X$, $r \in \N_{\geq 1}\cup \{\infty\}$. It is well-known (see \cite{HPS}) that $E^s$ and $E^u$ uniquely integrate to continuous foliations $\cW_f^s$ and $\cW_f^u$ respectively, called the \textit{stable} and \textit{unstable} foliations. For any $x \in X$ and $* = s,u$,  the leaf of $\cW_f^*$ through $x$, denoted by  $\cW_f^*(x)$,   is an immersed $C^{r}$-manifold, and  $f(\cW_f^*(x))=\cW_f^*(f(x))$. If $f \in \cPH^2(X)$, the transverse regularity of $\cW_f^s$ and $\cW_f^u$ 
is H\"older (see \cite{PSW}).

\begin{defi}[Accessibility]\label{def acc}
An \textit{$su$-path} of $f$ is a path obtained by concatenating finitely many subpaths, 
each of which lies entirely in a single leaf of $\cWs_f$ or $\cWu_f$. The map $f$ is said to be \textit{accessible} if any two points in $X$ can be connected by some $su$-path.  We say that $f$ is ($C^1$-)\textit{stably accessible} if there exists $\cU$, a $C^1$-open neighbourhood of $f$, such that any $g \in \cU$ is accessible.
\end{defi}

\begin{conj}[Accessibility implies ergodicity]\label{Accessibility implies ergodicity}
Essential accessibility implies ergodicity among $C^2$ volume preserving partially hyperbolic diffeomorphisms.
\end{conj}
Recall that \textit{essential accessibility} is a weakening of the notion of accessibility: it means that for 
any two measurable sets $A,B$ of positive volume, there exist $a\in A$ and $b\in B$ which can be connected by some $su$-path.
\begin{conj}[Density of accessibility]\label{Density of accessibility}
For any integer $r \in \mathbb{N}_{\geq 2}\cup \{\infty\}$, stable accessibility is open and dense among $C^r$ partially hyperbolic diffeomorphisms, volume preserving or not.
\end{conj}

The state of the art on Conjecture \ref{Accessibility implies ergodicity} is the result of K. Burns and A. Wilkinson \cite{BW}. They proved Conjecture \ref{Accessibility implies ergodicity} under one mild technical condition called \textit{center bunching}, which asserts, in rather loose terms, that the hyperbolic part dominates  nonconformality of the center.
Their result improved earlier work of Pugh-Shub in \cite{PS2}, which required two technical conditions: 
\textit{dynamical coherence} and a stronger form of center bunching. Dynamical coherence is a very commonly used notion, which asserts certain joint integrability of the invariant subspaces $E^{c},E^{s}, E^{u}$. We will give the formal definitions  in  Definitions \ref{defini center bunch}-\ref{plaque expansive}.

In comparison, there is a paucity of progress towards Conjecture \ref{Density of accessibility}. When the center dimension is one, Conjecture \ref{Density of accessibility}  was proved by F. Rodriguez-Hertz, M.A. Rodriguez-Hertz and R. Ures in \cite{RHRHU}. It is still open for any $\dim E^c >1$. To describe the current state of Conjecture \ref{Density of accessibility}, we mention several related results, which were obtained among certain classes of systems. 

 \enmt
\item[$\bullet$] K. Burns and A. Wilkinson \cite{BW2} proved a version of Conjecture \ref{Density of accessibility} for compact group extensions over Anosov systems. 
 
\item[$\bullet$] In a recent paper \cite{HS}, V. Horita and M. Sambarino obtained some $C^r$-density result for a class of partially hyperbolic systems  with $\dim  E^c=2$ and \textit{uniformly compact} center foliations (see Definition \ref{definitionuniformlycompact}).
 
 \item[$\bullet$] Another $C^r$-density 
 result for partially hyperbolic systems with $\dim E^{c}=2$ was obtained recently by A. Avila  and  M. Viana \cite{AV} using a very different method. 
 
 \item[$\bullet$] Z. Zhang \cite{Z} recently proved $C^r$-density of $C^2$-stable ergodicity for a class of skew products over Anosov maps, satisfying pinching, bunching conditions with certain type of dominated splitting in the center subspace.  
  
\eenmt
 
 The difficulty of Conjecture \ref{Density of accessibility} is mainly due to the $C^2$-smallness of the perturbation. In fact, the $C^1$-density of stable accessibility was already proved by D. Dolgopyat and A. Wilkinson \cite{DW} in 2003. There was a line of research focused on the $C^1$ version of Conjecture \ref{stableergodicityconjecture}. In the case where $\dim  E^c=1,2$, this was proved in \cite{BMVW} and \cite{RHRHTU}. Recently, the $C^1$-version of Conjecture \ref{stableergodicityconjecture} was completely solved  by A.  Avila, S. Crovisier and A. Wilkinson \cite{ACW}.\footnote{In \cite{ACW} the authors showed that stable ergodicity is $C^1$-dense among  $C^1$ volume preserving partially hyperbolic diffeomorphisms. If Conjecture \ref{stableergodicityconjecture} is true, then such statement would be an immediate  corollary by \cite{Avila}.}

As the main result of this paper, we will verify $C^r$-density of stable ergodicity in $C^1$-neighbourhoods of two classes of partially hyperbolic systems, defined by some technical conditions. Let us first recall some  notions needed to state our result. 

The following definition is commonly-used in the study of partially hyperbolic systems (see \cite{PS,PS2,PSW,PSW3}). 
\begin{defi}[Pinching]\label{def pin}
	An Anosov diffeomorphism $f$ 
	with constants $\bar{\chi}^{s}, \bar{\chi}^{u}$ as in \eqref{anosovexponent1}, \eqref{anosovexponent2} is  called 
	\textit{$\theta$-pinched} for some $\theta \in (0,1)$ if
	there exist 
	$\hat{\chi}^{u}, \hat{\chi}^{s} > 0$ s.t.  
	\begin{gather*}
	e^{-\hat{\chi}^{s}} < \norm{Df^{-1}}^{-1}\leq \norm{Df} < e^{\hat{\chi}^{u}},\\  
	-\bar{\chi}^{s} + \theta \hat{\chi}^{u} < 0, \quad \bar{\chi}^{u} - \theta \hat{\chi}^{s} > 0.
	\end{gather*}
	
	A partially hyperbolic system $f$ 
	with functions  $\bar{\chi}^{s}, \bar{\chi}^{u}, \bar{\chi}^c, \hat{\chi}^c$ as in \eqref{ph ineq 1}--\eqref{ph ineq 4}   is   called 
	$\theta$-\textit{pinched}  for some $\theta \in (0,1)$ 
	if there exist $\hat{\chi}^{u}, \hat{\chi}^{s} > 0$ such that 
	\begin{gather}\label{ph ineq 5}
	e^{-\hat{\chi}^{s}} < \norm{Df^{-1}}^{-1}\leq \norm{Df} < e^{\hat{\chi}^{u}},\\  
	\theta \hat{\chi}^{u} <  \bar{\chi}^{s}+\bar{\chi}^c, \quad
	\theta \hat{\chi}^{s}<\bar{\chi}^{u} -  \hat{\chi}^c. \nonumber
	\end{gather}
\end{defi}

By definition, given any Anosov or partially hyperbolic diffeomorphism $f$, there exists $\theta \in (0,1)$ such that $f$ is $\theta$-pinched.

A related notion is the following. 
\begin{defi}[Center bunching]\label{defini center bunch}
	A partially hyperbolic system $f$ with functions $\bar{\chi}^{s}, \bar{\chi}^{u}, \bar{\chi}^{c}, \hat{\chi}^{c}$ as in \eqref{ph ineq 1}--\eqref{ph ineq 4} is called \textit{center bunched} (see \cite{BW}) if 
	\begin{equation}\label{def cent bun} 
	\hat{\chi}^{c}<\bar{\chi}^{s} + \bar{\chi}^{c},  \qquad
	-\bar{\chi}^{c}<\bar{\chi}^{u}- \hat{\chi}^{c}. 
	\end{equation}
\end{defi} 

\begin{defi}[Dynamical coherence, plaque expansiveness]\label{plaque expansive}
	We say that a partially hyperbolic  system $f$ is: 
	\begin{itemize}
		\item \textit{dynamically coherent} (see \cite{HPS}) if  $\Ecs_f$, resp. $\Ecu_f$, integrates to a $f$-invariant foliation $\cWcs_f$,  resp.  $\cWcu_f$, called the \textit{center-stable foliation}, resp. the \textit{center-unstable foliation}. In this case, for any $x \in X$, we let $\cWc_f(x):= \cWcs_f(x) \cap \cWcu_f(x)$; the collection of all such leaves forms a foliation  $\cWc_f$, called the  \textit{center foliation}, which integrates $\Ec_f$, and subfoliates both $\cWcs_f$ and $\cWcu_f$ (see \cite{BW3});
		\item \textit{plaque expansive} (see \cite[Section 7]{HPS}) if $f$ is dynamically coherent and there exists  $\varepsilon>0$ with the following property: if $(p_n)_{n \geq 0}$ and $(q_n)_{n \geq 0}$ are $\varepsilon$-pseudo orbits which respect $\cW_f^c$ and if $d(p_n,q_n)\leq \varepsilon$ for all $n \geq 0$, then $q_n\in \cW_f^c(p_n)$. It is known that plaque expansiveness is a $C^1$-open condition (see Theorem 7.4 in \cite{HPS}).
	\end{itemize} 
\end{defi}

\begin{defi}[Uniformly compact foliation]\label{definitionuniformlycompact}
	A foliation is   \textit{uniformly  compact} if all the leaves are compact, and the leaf volume of the leaves is uniformly bounded.
\end{defi}

We can now state our main result. 

\begin{thm} \label{main thm 2.1}
Let $X$ be a compact smooth Riemannian manifold. Let $r \in \N_{\geq 2} \cup \{\infty\}$, and assume that $f\in \mathrm{Diff}^r(X)$ 
is a dynamically coherent, center bunched partially hyperbolic diffeomorphism. Let $c:= \dim  E^c_f$. If  either $c=1$\footnote{Any partially hyperbolic diffeomorphism with center dimension $1$ is automatically center bunched.} and $f$ is plaque expansive, or $c >1$ and $f$ satisfies at least one of the following assertions:
\begin{itemize}
\item $f$ is $(\frac{c-1}{c})^{\frac{1}{7}}$-pinched and has uniformly compact center foliation;
\item $f$ is $(\frac{c-1}{c})^{\frac{1}{9}}$-pinched and the maps $x \mapsto E^{cs}_f(x), E^{cu}_f(x)$ are of class $C^1$,
\end{itemize}
then there exists a $C^1$-open neighbourhood of $f$ in $\diff^r(X)$, denoted by $\cU$, such that $C^1$-stable accessibility is prevalent in the $C^{r}-J-$Kolmogorov sense in $\cU$, for any $J \geq J_0$ (see Subsection \ref{sub prev} below for a precise definition). Here $J_0$ is an integer depending only on $\dim X$. 

Moreover, let $\mathrm{Vol}$ be a smooth volume form  on $X$, and assume that $f \in \diff^{r}(X, \mathrm{Vol})$ satisfies one of the previous conditions. Then the above conclusion is true for $\cU_{0}$, a $C^1$-open neighbourhood of $f$ in $\diff^{r}(X, \mathrm{Vol})$, in place of $\cU$.
In particular, $C^1$-stable ergodicity is $C^{r}$-dense in $\cU_{0}$.
\end{thm}

Theorem \ref{main thm 2.1} generalizes all the  results on stable ergodicity from \cite{BW2,HS,AV,Z} to arbitrary center dimension in the strongly pinched region. 
 Compared to the previous works, our result has two significant novelties:
\enmt
\item this is the first time that $C^r$-density of stable ergodicity is proved for fully nonlinear systems with arbitrary center dimension, for $r \geq 2$\footnote{Among the set of   compact group extensions, $C^r$-density of stable ergodicity was proved in \cite{BW2}. These systems also have higher center dimension, but they are simpler as the action on the fiber is by group translation, hence is characterised by finitely many parameters.};
\item this is the first result that shows that stable ergodicity and accessibility  are prevalent in the measure-theoretical sense. 
\eenmt
To provide motivations for (2), let us mention that  there are two ways to approach the question of genericity: topological and metric. These notions are sometimes conflicting\footnote{For instance, among circle diffeomorphisms, a topological generic  map  has rational rotation number, while those with irrational rotation number occupy positive measures in many one-parameter families.}. The study of prevalence properties goes back to Kolmogorov \cite{Kol}. We say that a property is prevalent if it holds for a \textit{typical dynamical system in Kolmogorov's sense} (see Definition \ref{parameter family, prevalence}) as in \cite{B,Kol,PS3} (see \cite{HunKal,OY} for different notions). 
 Even when $\dim E_f^c=1$, we have strengthened the result of \cite{RHRHU} as we show that stable ergodicity is not only $C^r$-dense but also prevalent among center bunched partially hyperbolic diffeomorphisms with one-dimensional center, assuming plaque expansiveness. In \cite[Conjecture 3]{PS3}, the authors conjectured that: for the generic finite dimensional submanifold $V$ in $\mathrm{Diff}^r(M)$ and almost every $f \in V$, the equivalence classes of points in the chain recurrent set of $f$ are open in the chain recurrent set. They also mentioned that the validity of such conjecture would give a ``finite spectral decomposition for $f$ where each piece of the decomposition has something akin to the accessibility property".  P. Berger \cite{B} constructed a counter-example to this conjecture in the Newhouse domain. Our result strongly suggests that the accessibility property could be prevalent among partially hyperbolic diffeomorphisms. 

\subsection{Further illustrations of our result}

Conjecture \ref{stableergodicityconjecture} has its origin in several concrete models. 
For instance, given an integer $n\geq 1$, the linear automorphism of $\T^n:= \R^n / \Z^n$ associated to a matrix $A \in \mathrm{SL}(n, \Z)$ is defined as 
the unique diffeomorphism $f_{A} \colon \T^n \to \T^n$ such that the following diagram commutes, where $\pi\colon \R^n \to \T^n$ denotes the natural projection:
\begin{center}
	\begin{tikzcd}
	\R^n \arrow{r}{A} \arrow[swap]{d}{\pi} & \R^n \arrow{d}{\pi} \\%
	\T^n \arrow{r}{f_{A}}& \T^n
	\end{tikzcd}
\end{center}
Back in the 1970's, Hirsch-Pugh-Shub \cite{HPS} already asked 
whether any ergodic linear automorphism of $\T^n$ stably ergodic, for $n \geq 2$.

Positive answer to this question is known when the map is Anosov by \cite{AnoSinai}. This question 
was solved by F. Rodriguez-Hertz in \cite{RH} for any $n \leq 5$. This in particular answered a special case of the question, asked in \cite{GPS} for an explicit $4 \times 4$ matrix. 
 More generally, in  \cite{RH}, the author investigated the case of pseudo-Anosov maps with $2$ eigenvalues of modulus 1. 
In \cite{PS}, the authors mentioned that the validity of Conjecture \ref{stableergodicityconjecture} would give a positive answer to the following weaker version of the previous question\footnote{See the remark below   \cite[Conjecture 1]{PS}.}, namely, given two integers $n, r \geq 2$, 
whether the $C^r$-generic volume preserving perturbation of any ergodic automorphism of $\T^n$ ergodic.  
To the best of our knowledge, the question remains open for any $\dim  E^c  >1$. 

Let us now consider $M=\mathrm{SL}(n,\R)/\Gamma$ for some  uniform discrete subgroup $\Gamma$  of $\mathrm{SL}(n,\R)$. Let   $A \in \mathrm{SL}(n,\R)$ with at least one eigenvalue of modulus different from $1$, and  let $L_A\colon M \to M$ be  the left translation by $A$. In \cite{PS}, the authors ask whether $L_A$ is stably ergodic among $C^2$ volume preserving diffeomorphisms of $M$. 

Unlike the case of $\T^n$, the topological complexity of the homogeneous manifold $M$ has so far prevented a generalization of \cite{RH}. One can also consider the generic version of the previous question, namely, whether  the $C^\infty$ generic volume preserving perturbation of $L_A$ is stably ergodic. This is true when the map $L_A$ is Anosov,   by \cite{AnoSinai}, or when the center dimension $\dim   E^c$ is equal to one, but the question remains open for any $\dim E^c>1$.

As corollaries of the main result of this paper, we  answer  the previous questions in any dimension in the strongly pinched region, namely for maps with pinching exponents close to $1$.

\begin{thm} \label{main thm 1.1}
	Let $n \geq 2$ be some integer. For any $r \in \N_{\geq 2} \cup \{\infty\}$, any linear partially hyperbolic automorphism $f_{A} \colon \T^n \to \T^n$, ergodic or not, that is $(\frac{c-1}{c})^{\frac{1}{5}}$-pinched, where $c$ is the number of eigenvalues of $A \in \mathrm{SL}(n, \Z)$ of modulus $1$, there exists $\cal U$, a $C^1$-open neighbourhood of $f_A$ in $\cPH^r(\T^n, \mathrm{Vol})$, such that for some $C^r$-dense subset $\cU'$ of $\cU$, any map in $\cU'$ is a stably ergodic diffeomorphism.
\end{thm}
%

\begin{thm}\label{main thm 1.2}
	Let $\Gamma$ be a uniform discrete subgroup   of $\mathrm{SL}(n,\R)$, let $M:=\mathrm{SL}(n,\R)/\Gamma$ and let $L_A\colon M \to M$ be  the left translation by $A \in \mathrm{SL}(n,\R)$, assuming that $A$ has an eigenvalue with modulus different from $1$. Then, the $C^\infty$ generic volume preserving perturbation of $L_A$ is stably ergodic for any $\theta$-pinched $L_A$, where $\theta \in (0,1)$ depends  only on the integer $n$. 
\end{thm}

The partially hyperbolic splittings for Theorems \ref{main thm 1.1} and \ref{main thm 1.2} are the canonical ones: the center spaces are the neutral subspaces of the affine actions. Note that even for linear automorphisms with two-dimensional center, Theorem  \ref{main thm 1.1} partially improves and generalizes the main result in \cite{RH}, in the following sense: (1) we removed the pseudo-Anosov condition; (2) our result also applies to non-ergodic maps; (3) we weakened the regularity condition (for the perturbations) from $C^5$ to $C^1$.  

Another open question in \cite{PS} is the following.\footnote{Actually, the version we give here is stronger than that initially stated in \cite{PS}.}
Given two compact Riemannian manifolds  $M,N$, 
where $M$ supports a $C^r$ volume preserving Anosov diffeomorphism $g \colon M \to M$, $r \in \mathbb{N}_{\geq  2} \cup \{\infty\}$, 
is the $C^r$-generic volume preserving perturbation of $g \times \mathrm{Id}\colon M\times N \to M \times N$ ergodic? 
Again, this question has a positive answer when the map in question satisfies $\dim  N =1$.  A recent result of A. Avila and M. Viana \cite{AV} also gives an affirmative answer to this question when $\dim  N=2$ (see also \cite{HS} for a related result). This question remains open for any $N$ of dimension at least $3$. 
As a corollary of our main result, we can answer this question in any dimension in the strongly pinched region.

\begin{thm}\label{main thm 1.3}
	Let  $M,N$ be two compact Riemannian manifolds,  
	where $M$ supports a $C^r$ volume preserving Anosov diffeomorphism $g \colon M \to M$, $r \in \mathbb{N}_{\geq  2} \cup \{\infty\}$, and 
	let $h \colon M \to \mathrm{Isom}(M)$ be a $C^r$ map. Assume that $g$ is $\big(\frac{n-1}{n}\big)^{\frac 17}$-pinched where $n:=\dim N$. Then a $C^r$-generic volume preserving perturbation of the map $(x,y)\mapsto (g(x),h(x,y))$ is stably ergodic. 
\end{thm}

Theorems \ref{main thm 1.1}, \ref{main thm 1.2} and \ref{main thm 1.3} are immediate consequences of a more general result, Theorem \ref{main thm 2}, stated in Section \ref{section main result}.


\subsection{Idea of the proof}

We follow closely the method in \cite{Z}. In \cite{Z}, the author
studied a class of skew products over Anosov systems, and divided the problem into:
I. showing that the property of having a stably open accessibility class is $C^r$-generic;
II. showing that the property of having an open accessibility class with intermediate
volume is $C^r$-meager. In Step I, we prove the existence of open accessibility classes
using a quantitative version of a theorem of M. Bonk and B. Kleiner in \cite{BK} (the details about this quantitative statement  can be found in Section \ref{section Bonk-Kleiner}). In our
setting, this boils down to destroying common intersections between many different
holonomy loops. This is done by a parameter exclusion within a family of random
perturbations. 

The main new observation in this paper (\textit{\`a la} Avila) is that: by letting the number of loops be sufficiently large compared to the dimension of the manifold, we are left with enough room to create open accessibility class at every point, due to the fact that the random perturbations we use are not very sensitive to the map and the point. The details about this argument can be found in Section \ref{section generic open acc}. This allows us to bypass Step II in \cite{Z} which only works under restrictive assumptions (for a different application of the method in Step II, see \cite{Z2}). Our construction is also suitable to study the measure-theoretical prevalence. Our method suggests that under a strong pinching condition, the failure of accessibility should be a phenomenon of {\it infinite codimension}.

On the technical level, in Section \ref{section spanning}, we use \cite{DW} to construct families of center disks to connect
different regions of the space. Small complications arise in the study of prevalence,
since we need to organize several families for different maps. For each disk we
consider a parametrized family of random perturbations generated by vector fields
with disjoint supports, and parametrize a part of the accessibility class of any given
point $x$ in a slightly smaller disk by $[0,1]^c$, where $c$ is the center dimension. We then apply Bonk-Kleiner's criterion to show the openness of the accessibility class of
$x$ for most parameters in the family. The regularity results in \cite{PSW,PSW3} are used to
reduce the problem to a finite set of points and loops.

\begin{convention} 
Given a compact smooth Riemannian manifold $X$ with a smooth volume form $\mathrm{Vol}$, for any $r \in \N_{\geq 1} \cup \{\infty\}$,	we denote by $\cPH^r(X)$ (resp. $\cPH^r(X,\mathrm{Vol})$) the set of all $C^r$ (resp. $C^r$ volume preserving) partially hyperbolic diffeomorphisms on $X$ with bounded $C^r$ norms. 
	
In the course of the paper, we will often use constants depending on a diffeomorphism $f$ (and that may or may not depend on other things). We say that a constant $C$ depending on a $C^r$ diffeomorphism $f$ is \textit{$C^r$-uniform} if it works for all diffeomorphisms in a $C^r$-open neighbourhood of $f$.
We introduce several constants related to a diffeomorphism in Notations \ref{notation 1}, \ref{notation 2} and Construction \ref{given c disk get a chart}.

Given $l \geq 0$ and diffeomorphisms $f_1,f_2,\dots, f_l$, we use the notation $\prod_{i=1}^{l}f_i$ to denote $f_l \circ \cdots \circ f_1$, where by convention $\prod_{i=j+1}^{j}f_i:= \mathrm{Id}$ for any $j=1,\dots,l-1$.
\end{convention}

\section*{Acknowledgments}
The authors wish to thank 
Artur Avila, Pierre Berger, Sylvain Crovisier, Julie Déserti, Jacopo De Simoi, Bassam Fayad, Pascal Hubert, Rapha\"el Krikorian, Carlos Matheus, Davi Obata,   Rafael Potrie, Enrique Pujals, Federico Rodriguez-Hertz, Bruno Santiago, Frank Trujillo, Amie Wilkinson  for their support and  useful conversations. We also wish to thank the hospitality of the Universit\'e Paris Diderot in Paris and of IMPA in Rio de Janeiro -- where part of this work was carried over. We thank the anonymous referee for many useful suggestions.

\section{Main result}\label{section main result}

In the following, we fix a smooth Riemannian manifold $X$  with a smooth volume form $\mathrm{Vol}$.
	Let $f \colon X \to X$ be a partially hyperbolic diffeomorphism  with functions $\bar{\chi}^{s}, \bar{\chi}^{u}, \bar{\chi}^{c}, \hat{\chi}^{c}$ as in \eqref{ph ineq 1}--\eqref{ph ineq 4}.
For any real number $\varrho \in \R$, we set 
\begin{equation}\label{definition betaa} 
\beta(f,\varrho):= \min\left(\frac{\bar{\chi}^{s} + \bar{\chi}^{c}}{\hat{\chi}^{u}}, \frac{\bar{\chi}^{u} - \hat{\chi}^{c}}{\hat{\chi}^{s}}\right)\min\left(\frac{\bar{\chi}^s}{\hat{\chi}^s}, \frac{\bar{\chi}^u}{\hat{\chi}^u}\right)^{\varrho}.  
\end{equation}

We will focus on the case where $f$ is dynamically coherent (recall Definition \ref{plaque expansive}) and satisfies   one of the following properties:
\begin{enumerate}[label=$(\alph*)$]
	\item\label{condition a}
	the center foliation $\cWc_f$ is  uniformly compact (see Definition \ref{definitionuniformlycompact});  
	\item\label{condition b} the maps $x \mapsto E^{cs}_f(x), E^{cu}_f(x)$ are of class $C^1$.
\end{enumerate}

Let us state the most general version of our result, which contains Theorem \ref{main thm 2.1}.
\begin{thm}\label{main thm 2}
Let $r \in \N_{\geq 2} \cup \{\infty\}$, let $f  \in \cPH^r(X)$ be dynamically coherent and center bunched, and let $c:= \dim E^{c}_f$.  If either $c=1$ and $f$ is plaque expansive, or $c>1$ and $f$  satisfies one of the following conditions:
\begin{enumerate}[label=(\arabic*)]
	\item\label{condition 1 thm E} condition \ref{condition a} holds for $f$, and $\beta(f,3)>\frac{c-1}{c}$;
	\item\label{condition 2 thm E} condition \ref{condition b} holds for $f$, and $\beta(f,4)>\frac{c-1}{c}$,
\end{enumerate}
then there exist $\cU$, a $C^1$-open neighbourhood of $f$ in $\diff^r(X)$, and  an integer $J_0$ only depending on $\dim X$, such that for any $J \geq J_0$, $C^1$-stable accessibility is prevalent in $\cU$ in the $C^{r}-J-$Kolmogorov sense. 

Moreover, let $f \in \diff^{r}(X, \mathrm{Vol})$ satisfy the above condition. Then the above conclusion is true for $\cU_{0}$, a $C^1$-open neighbourhood of $f$ in $\diff^{r}(X, \mathrm{Vol})$, in place of $\cU$.
In particular, $C^1$-stable ergodicity is $C^{r}$-dense in $\cU_{0}$.

\end{thm}

\begin{rema}
	If $f$ is $\theta$-pinched for some $\theta\in(0,1)$ and satisfies  \ref{condition a}, resp. \ref{condition b}, then we can see that $\beta(f,3)>\theta^{7}$, resp. $\beta(f,4)>\theta^{9}$. Therefore,  Theorem \ref{main thm 2.1} is a consequence of Theorem \ref{main thm 2}.  Similarly, Theorem \ref{main thm 1.1}, \ref{main thm 1.2} and \ref{main thm 1.3} follow by making appropriate choice of functions $\bar{\chi}^{s}, \bar{\chi}^{u}, \bar{\chi}^{c}, \hat{\chi}^{c}$. We omit these straightforward computations. 
\end{rema}

\section{Preliminaries}

In the following section, we recall some general notions about parameter families, prevalence and partially hyperbolic diffeomorphisms. 

\subsection{Prevalence}\label{sub prev}
Let $X$ be a smooth Riemannian manifold with a smooth volume form $\mathrm{Vol}$. We refer the reader to \cite[Chapter II, \S 3]{GG} and \cite{Hirsch} for more details about the notions recalled in the following. 


\begin{defi}[$C^r$ topology]
	Let $m,n \geq 1$ be integers. Given $k \in \N$ and  $f \in C^k(\R^m, \mathbb{R}^n)$, we set $\|f\|_{C^k}:=\sup\limits_{0 \leq i\leq k,\ x \in \R^m} \|\partial_x^i f\|$. Here, $\partial_x^i f$ is a multi-linear map from $(T_x \R^m)^i$ to $T_{f(x)}\R^n$, and $\|\partial_x^i f\|$ denotes the norm of this linear map. We let $d_{C^k}$ be the distance induced by   $\|\cdot\|_{C^k}$. Given $f,g \in C^\infty(\R^m,\R^n)$, we set $d_{C^\infty}(f,g):=\sum_{k=0}^{\infty} 2^{-k} \frac{\|f-g\|_{C^k}}{\|f-g\|_{C^k}+1}$. 
	For $r \in \N\cup \{\infty\}$, the $C^r$  topology on 
	$C^r(\R^m, \mathbb{R}^n)$ is the  topology induced by   $d_{C^r}$.
	Given smooth Riemannian manifolds $M,N$, we define accordingly the $C^r$ topology for maps between $M$ and $N$,  
	and we denote by  $C_b^r(M,N)$   the subset of maps in $C^r(M,N)$ with bounded distance to the constant maps.
\end{defi}

\begin{defi}[Parameter family]
Given integers $r,m,n,J\geq 1$, we define the space $C^r([0,1]^J, C_b^r(\R^m,\R^n))$ as the set of families $\{f_\omega\}_{\omega\in [0,1]^J}$ of maps $f_\omega\in C^r(\R^m,\R^n)$ such that for every $0 \leq i,j \leq r$ the derivative  $\partial_\omega^i \partial_x^j f_\omega(x)$ is a multi-linear map which depends continuously on  $(\omega,x) \in [0,1]^J \times \R^m$.
The $C^r$ topology on the space $C^r([0,1]^J,C_b^r(\R^m,\R^n))$ is the topology induced by the norm $\|\cdot \|_{C^r}$:
$$
\|\{f_\omega\}_\omega\|_{C^r}:=\sup\limits_{\substack{0\leq i,j \leq r,\\ (\omega,x)\in [0,1]^J \times \R^m}} \|\partial_\omega^i \partial_x^j f_\omega(x)\|.
$$ 
The $C^\infty$ topology on $C^\infty([0,1]^J,C_b^\infty(\R^m,\R^n))$ is defined by analogy.

Let $r \in \N\cup \{\infty\}$. Given smooth  Riemannian manifolds $M,N$ and $\cU \subset C_b^r(M,N)$, a $C^r-J-$\textit{family in} 
$\cU$ is an element $\{f_\omega\}_\omega \in C^r([0,1]^J,\cU)$. We define the $C^r$ topology on $C^r([0,1]^J,\cU)$ analogously and denote by $d_{C^r}$ the associated metric.
\end{defi}
\begin{defi}[Prevalence] \label{parameter family, prevalence}
	Let $\mathcal{U}$ be an open subset in $\mathrm{Diff}^r(X)$ or $\mathrm{Diff}^r(X,\mathrm{Vol})$, $r \in \N\cup \{\infty\}$.   A property $\mathscr{P}$ for maps in $\mathrm{Diff}^{r}(X)$ or $\mathrm{Diff}^r(X,\mathrm{Vol})$  is said to be \textit{prevalent in the $C^r-J-$Kolmogorov sense} in $\cU$, if for a $C^r$-generic $C^r-J-$family $\{f_\omega\}_\omega$ in $\cU$ and for almost every $\omega \in [0,1]^J$,  $f_\omega$ satisfies $\mathscr{P}$.
\end{defi}


We introduce the following notion for technical reasons.  
\begin{defi}[Good family]\label{defregularfamilyofdiff}
	Let $r \in \N \cup \{\infty\}$. Given an integer $J \geq 1$, a $C^r-J-$family  $\{f_\omega\}_{\omega} \in C^{r}([0,1]^{J}, \diff^{r}(X))$ is said to be \textit{good} if the fixed points of $f_\omega^{k}$ are isolated for all integer $k \geq 1$ and almost every $\omega \in [0,1]^{J}$.
\end{defi}

\begin{prop}\label{propapproximationbyregularfamily} 
There exists $J_0 = J_0(\dim X)>0$, which we fix in the rest of this paper, with the following property. For any  $r \in \N_{\geq 2} \cup \{\infty\}$ and any integer $J \geq J_0$,  the set of good $C^{r}-J-$families is dense in the set of $C^{r}-J-$families with respect to the $C^r$ topology.
\end{prop}  
\begin{proof}
This is essentially contained in the proof of \cite[Theorem 2.2]{Kal}. In \cite{Kal}, the author showed the prevalence of Kupka-Smale diffeomorphisms in $\diff^{r}(X)$.  In contrast to the Kupka-Smale property, our notion of good family is only a transversality condition on the level of $0$-jets.  Next lemma suffices for our purpose. 

\begin{lemma}\label{modification1} For any integers $p, q \geq 3$, any $r \in \N_{\geq 2} \cup \{\infty\}$, for any $C^r$ map $f \colon (-1,1)^{p} \to \R^{q}$, there exist an integer $L \geq 1$ and $C^{\infty}$ divergence-free vector fields $V_1,\dots, V_L$ on $\R^{q}$, supported in $(-1,1)^{q}$, such that the following is true. Let us denote by $\mathscr{F}^{b_i}_{V_i} \colon \R^q \to \R^q$ the time-$b_i$ map of the flow generated by $V_i$, and let $F \colon (-1,1)^{p} \times (-1,1)^L \to \R^{q}$ be defined by $F\colon (x,b) \mapsto \mathscr{F}_{V_{L}}^{b_L}\circ \cdots\circ \mathscr{F}_{V_{1}}^{b_1}(f(x))$, for all $b=(b_1,\dots,b_L)$. Then the map 
\begin{equation*}
\mathscr{G} \colon
(-1,1)^{p+L} \to \R^{p} \times \R^q,\quad 
(x,b) \mapsto (x, F(x,b))
\end{equation*}
is a submersion for any $(x,b)$ such that $F(x,b) \in (-\frac 12,\frac 12)^{q}$.
\end{lemma}
\noindent Lemma \ref{modification1}  is proved via 
a direct construction. Using Lemma \ref{modification1} in place of  \cite[Lemma 1.5]{Kal}, the proof of Proposition \ref{propapproximationbyregularfamily} proceeds as  that of \cite[Theorem 2.2]{Kal}. 
\end{proof}

\subsection{Partially hyperbolic diffeomorphisms} \label{ph dif}

Fix an integer $d \geq 1$. We let $X$ be a smooth $d$-dimensional Riemannian manifold with a smooth volume form $\mathrm{Vol}$.

Let $f \colon X \to X$ be a  partially hyperbolic diffeomorphism or an Anosov map.

In the following, we will call a leaf of $\cW_f^c,\cWcu_f$, etc. a center leaf, center-unstable leaf, etc. 

 \begin{defi}[Holonomies] 

Let $f\in \cPH^1(X)$ be dynamically coherent. 

$\bullet$ Let $x_1\in X$ and $x_2 \in \cW_f^s(x_1)$.  By transversality, for $i =1,2$, there exists a neighbourhood  $\cC_i$ of $x_i$ within $\cW_{f,\mathrm{loc}}^{cu}(x_i)$ such that for any $x \in \cC_1$, the local $s$-leaf through $x$ intersects $\cC_2$ at a unique point, denoted by $H_{f,x_1,x_2}^{s}(x)=H_{f,\cC_1,\cC_2}^{s}(x)$. We thus   get a well-defined local homeomorphic embedding  
$H_{f,\cC_1,\cC_2}^{s}\colon
 \cC_1 \to \cC_2$,
called the (local) \textit{stable holonomy map} between $\cC_1$ and $\cC_2$.  For $i=1,2$, set $\tilde{\cC}_i:=\cC_i\cap \cW_{f,\mathrm{loc}}^{c}(x_i)$; by restriction,  $H_{f,\cC_1,\cC_2}^{s}$ induces a local homeomorphism $H_{f,\tilde \cC_1,\tilde \cC_2}^{s}\colon
\tilde \cC_1 \to \tilde \cC_2$. Unstable holonomies are defined accordingly.

$\bullet$ Let $x_1\in X$ and $x_2 \in \cW_f^c(x_1)$ be two sufficiently close points in the same center leaf. 
Let $*\in \{u,s\}$. 
Then, 
the (local) \textit{center holonomy map} 
$
H_{f,x_1,x_2}^c
$ 
along local leaves in $\cW_f^c$ is a well-defined local homeomorphism from a neighbourhood of $x_1$ in $\mathcal{W}_f^*$ to $\mathcal{W}_f^*(x_2)$.
\end{defi}



The following result  \cite[Theorem A]{PSW} relates the pinching condition in Definition \ref{def pin} with the regularity of $u,s$-holonomy maps.
\begin{prop} \label{thm theta pinching theta holder}
If $f\in \cPH^1(X)$ is $\theta$-pinched for some $\theta \in (0,1)$, then the local unstable and stable holonomy maps are uniformly $\theta$-H\"older.
\end{prop}


The following result is contained in the proof of  \cite[Theorem B]{PSW}. It relates the center bunching condition in Definition \ref{defini center bunch} to the regularity of $u,s$-holonomies.
\begin{prop}\label{cor smooth holonomy maps}
If $f \in \cPH^{2}(X)$ is dynamically coherent and center bunched, then local stable/unstable 
holonomy maps between center leaves are $C^1$ 
 when restricted to some center-stable/center-unstable leaf  
and have uniformly continuous derivatives. 
\end{prop}

In Proposition \ref{cor smooth holonomy maps}, uniformity of the continuity is a simple consequence of the invariant section theorem and the uniform $C^2$ bound. 

\subsection{On leaf conjugacy} 

Later on, we will focus on dynamically coherent systems satisfying one of the conditions \ref{condition a} or \ref{condition b} in Section \ref{section main result}. 
The following result is due to Hirsch-Pugh-Shub. 
\begin{prop}[Theorem 7.1, \cite{HPS}, see also Theorem 1 in \cite{PSW3}]\label{thmplaqueexpansivetostablydc}
	Let $f$ be a dynamically coherent partially hyperbolic diffeomorphism  satisfying \ref{condition a} or \ref{condition b}.
	 Then any $g \in \cPH^1(X)$ which is sufficiently $C^1$-close to $f$ is also dynamically coherent. Moreover, there exists a homeomorphism $\mathfrak{h}=\mathfrak{h}_g \colon X\to X$,
	called a \textit{leaf conjugacy}, such that: (1) $\mathfrak{h}$ maps a $f$-center leaf to a $g$-center leaf; (2) both $h$ and $h^{-1}$ tend to ${\rm Id}$ in the uniform norm as $d_{C^1}(f,g)$ tends to $0$.
\end{prop}
\begin{proof} 
It suffices to see that any $f$ in the proposition is plaque expansive (recall Definition \ref{plaque expansive}). 
The plaque expansiveness is proved in \cite{Ca} (see also \cite[Proposition 13]{PSW3}) under  \ref{condition a}, respectively in \cite{HPS} under \ref{condition b}. 
\end{proof}

 The following result, due to Pugh-Shub-Wilkinson \cite{PSW3},  ensures that the leaf conjugacy $\mathfrak{h}$ in Theorem \ref{thmplaqueexpansivetostablydc} has H\"older regularity under \ref{condition a} or \ref{condition b}.

\begin{prop}[Theorems A-B,  \cite{PSW3}]\label{thmregularityofleafconjandcenterholon}
Let $f \in \cPH^1(X)$ be dynamically coherent, satisfying  \ref{condition a} (resp. \ref{condition b}), let $\bar \chi^s,\bar \chi^u,\hat \chi^s,\hat \chi^u$ be as in \eqref{ph ineq 1}--\eqref{ph ineq 5}, and let $\theta$ be a constant such that
\begin{equation} \label{theta''condition}
0 < \theta < \min\left(\frac{\bar{\chi}^s}{\hat{\chi}^s}, \frac{\bar{\chi}^u}{\hat{\chi}^u}\right) \leq 1.
\end{equation}
Then for some $C^1$-open neighbourhood $\cU_0 = \cU_0(f, \theta)$ of $f$, for any $g \in \cU_0$, the leaf conjugacy $\mathfrak{h}_g$ in Proposition \ref{thmplaqueexpansivetostablydc} exists and can be made uniformly $\theta$-H\"older, and local center holonomies between  sufficiently close   leaves 
are uniformly $\theta$-H\"older (resp. $\theta^2$-H\"older).
\end{prop}

We will later use Propositions \ref{thm theta pinching theta holder}, \ref{cor smooth holonomy maps} and \ref{thmregularityofleafconjandcenterholon} while keeping track of the uniformity of various quantities. We summarise these statements as follows.

\begin{nota}  \label{notation 1}

Let $X$ be a $d$-dimensional compact Riemannian manifold ($ d \geq 3$) with metric $d(\cdot, \cdot)$, and $f \in \cPH^1(X)$ be dynamically coherent and plaque expansive. 
 We denote by $d_{\cWs_f}, d_{\cWu_f}, d_{\cWc_f}, d_{\cWcs_f}, d_{\cWcu_f}$ the associated leafwise distances. For any $x\in X$, $\sigma>0$, and $*=s,u, c,cs,cu$, we set $\cW_f^*(x,\sigma):= \{y \in \cW_f^*(x) \ \vert\ d_{\cW_f^*}(x,y) < \sigma \}$.
Then there exist $C^1$-uniform constants $h_f > 0$, $\sigma_f \in (0,   h_f)$, $C_f > 1$, $\overline{\Lambda}_f > 1$, $\varepsilon_f > 0$, $\theta'_f, \theta''_f \in (0,1)$, and in case $f$ is $C^2$, also
a $C^2$-uniform constant $\Lambda_f>0$ satisfying 
\enmt
\item\label{item 1 definition 1} For $* = c,s,u,cs,cu$, for any $x \in X$, $y \in \cW^{*}_f(x, \sigma_f)$, 
we have 
\begin{center}
$d(x,y) \leq d_{\cW^{*}_f}(x,y) \leq C_f d(x,y)$.
\end{center}
\item\label{item 2 definition 1}  
For any $x \in X$, any $y \in B(x, \sigma_f)$, $\cW^{s}_f(x, h_f)$ transversally intersects $\cW^{cu}_f(y, h_f)$ at a unique point $z$, and $d_{\cW^{s}_f}(x,z), d_{\cW^{cu}_f}(z,y) < C_fd(x,y)$. If in addition, $y \in \cW^{cs}(x, \sigma_f)$, then $d_{\cW^s_f}(x,z),d_{\cWc_f}(z,y) < C_fd_{\cW^{cs}_{f}}(x,y)$.  
\item\label{item 3 definition 1} (Center bunching) If $f \in \cPH^2(X)$ is center bunched, then for any $x \in X$ and $y \in \cW^{cs}_f(x,\frac{\sigma_f}{2})$, the holonomy map $H^{s}_{f,x,y} \colon \cWc_f(x,\frac{\sigma_f}{2}) \to \cWc_{f}(y,  C_f\sigma_f)$ 
is well-defined. Moreover, $DH^{s}_{f,x,y}$ is uniformly continuous  and has norm bounded by $\Lambda_f$. 
\item\label{item 4 definition 1} (Pinching) $f$ is $\theta'_f$-pinched. Besides,  if $f \in \cPH^2(X)$, then  for any $x \in X$, $y \in \cW^{s}_f(x, \frac{\sigma_f}{2})$, 
the holonomy map $H_{f,x,y}^{s}\colon \cW^{cu}_f(x,\frac{\sigma_f}{2})\to \cW^{cu}_f(y,1) $ 
is well-defined 
and has $\theta'_f$-H\"older norm bounded by $\Lambda_f$.


\item\label{item 5 definition 1} (H\"olderness of $\mathfrak{h}$ and $H^c$) 
If $f$ satisfies  \ref{condition a}, resp. \ref{condition b}, then $\theta''_f$ satisfies \eqref{theta''condition} in place of $\theta$, and the following is true:
\begin{enumerate}
	\item[(5.1)]\label{subitem 6.1}  any $g \in \cPH^1(X)$ with $d_{C^1}(f,g) < \varepsilon_f$ is plaque expansive and  for any $x \in X$, any $y \in \cWc_g(x, \frac{\sigma_f}{2})$, the holonomy map $H_{g,x,y}^c$ is defined  on $\cW^{*}_g(x, \frac{\sigma_f}{2})$, $*=u,s$, and its $\theta''_f$- (resp. its $(\theta''_f)^2$-) H\"older norm is bounded by $\overline{\Lambda}_f$. 
	\item[(5.2)]\label{subitem 6.2} For any $g_1,g_2 \in \cPH^1(X)$ such that $d_{C^1}(f,g_i)< \varepsilon_f$, $i=1,2$, the leaf conjugacy $\mathfrak{h}_{g_1,g_2}=\mathfrak{h}_{g_2}\circ \mathfrak{h}_{g_1}^{-1}$ has $(\theta_f'')^2$-H\"older norm bounded by $\overline{\Lambda}_f$. Here $\mathfrak{h}_{g_i}$ is given by Proposition \ref{thmregularityofleafconjandcenterholon} for $g_i$ in place of $g$. 
\end{enumerate}
\eenmt

Moreover, we assume that Properties \ref{item 2 definition 1}, \ref{item 3 definition 1}, \ref{item 4 definition 1} above are also satisfied when we exchange the roles of $u$ and $s$.  
\end{nota}

\begin{defi}
	Let $f \in \cPH^2(X)$ be dynamically coherent and center bunched. 
	We say that $f$ satisfies  
	\begin{equation}\label{condition ae}
	\mbox{if \  \ref{condition a} holds and} \ 
	\theta'_f (\theta''_f)^{3} > \frac{c-1}{c}.
	\tag*{$(ae)$}
	\end{equation} 
	Similarly, we say that  $f$ satisfies   
	\begin{equation}\label{condition be}
	\mbox{if \  \ref{condition b} holds and } \ \theta'_f (\theta''_f)^{4} > \frac{c-1}{c}. 
	\tag*{$(be)$}
	\end{equation} 
\end{defi}

\begin{rema} 
If  Theorem \ref{main thm 2}\ref{condition 1 thm E}  holds for $f$, 
then  by Definition \ref{def pin}, Propositions \ref{thm theta pinching theta holder}, \ref{thmregularityofleafconjandcenterholon}, we can choose $\theta'_f, \theta''_f$ such that \ref{condition ae} holds. Similarly, if  Theorem \ref{main thm 2}\ref{condition 2 thm E}  holds for $f$,
then   we can choose $\theta'_f, \theta''_f$ such that \ref{condition be} holds. 
\end{rema}

\section*{Standing hypotheses for the rest of the paper}

We denote by $X$ a $d$-dimensional compact smooth Riemannian manifold with a smooth volume form $\mathrm{Vol}$; $r$ belongs to $\N_{\geq 1} \cup \{\infty\}$; and $f \in \diff^r(X)$.  
Whenever $f$ is declaimed to be partially hyperbolic, we denote  $c:=\dim E_f^c$, $d_s:=\dim  E_f^s$ and $d_u:=\dim E_f^u$.  Moreover:

\begin{enumerate}
	\item[(H1)]\label{(H1)} $f \in \cPH^r(X)$ is dynamically coherent and plaque expansive   in Sections \ref{subsection c disks}-\ref{subs holo map}, \ref{from dim gap}, \ref{holo maps fami}, \ref{sec construct cha};
	\item[(H2)]\label{(H2)} $f\in \cPH^r(X)$ is center bunched and $r \geq 2$, in Sections \ref{subs holo map},  \ref{from dim gap}.
\end{enumerate}

\section{Random perturbations} \label{sec random perturbations}

In this section, we will establish some estimates for certain perturbations of the holonomy maps of a dynamically coherent plaque expansive  partially hyperbolic
diffeomorphism.

\subsection{Basic notions and constructions}
We start with the following more general
situation. The following suspension construction will be used repeatedly.

\begin{defi}[$C^r$ deformation] \label{smooth deform}
Let $a \in \R^I$ for some integer $I >0$, and let  $U$ be an open neighbourhood of $a$ in $\R^{I}$.
A $C^{r}$ map $\hat{f} \colon U \times X \to X$ satisfying $\hat{f}(a,\cdot) = f$ and $\hat f(b,\cdot)\in \diff^r(X)$, $\forall\, b \in U$ is called a \textit{$C^r$ deformation at $(a,f)$ with $I$-parameters}.
We associate with such $\hat{f}$ the suspension map $T(\hat{f})$ defined by 
\begin{equation}\label{def T}
T= T(\hat{f}) \colon
U \times X \to U \times X,\quad 
(b,x) \mapsto (b,\hat{f}(b,x)). 
\end{equation}
If in addition $\hat{f}(b,\cdot) \in \Diff^{r}(X,\mathrm{Vol})$ for all $ b \in U$,  then we say that $\hat{f}$ is  \textit{volume preserving}. 
%
\end{defi}
%
%

\begin{defi}[Infinitesimal $C^r$ deformation]\label{def infini deform}
	
	Given an integer $I>0$, a $C^r$ map $V \colon \R^{I} \times X \to TX$ is called an \textit{infinitesimal $C^{r}$ deformation with $I$-parameters} if 
	
	\begin{enumerate}
		\item for each $B \in \R^{I}$, $V(B,\cdot)$ is a $C^{r}$ vector field on $X$;
		\item for each $x \in X$, $B \mapsto V(B,x)$ is a linear map from $\R^{I}$ to $T_{x}X$.
	\end{enumerate}
\end{defi}
	
\begin{construction}
	Given $I >0$,  $a \in \mathbb{R}^I$, and $V$, an infinitesimal $C^{r}$ deformation with $I$-parameters, then for any sufficiently small $\epsilon > 0$, we associate with $V$ a $C^{r}$ deformation at $(a,f)$ with $I$-parameters, denoted by $\hat{f}$,  which is defined by 
	\begin{equation*}
	\hat{f}(b, x):= \mathscr{F}_{V(b-a, \cdot)}(1,  f(x)), \quad \forall\, (b,x) \in U \times X,
	\end{equation*} 
	where $U = B(a, \epsilon) \subset \R^{I}$ and for any $B \in \mathbb{R}^I$, $\mathscr{F}_{V(B,\cdot)} \colon \R \times  X  \to X$ denotes the $C^r$ flow generated by the vector field $V(B,\cdot)$.
	In this case, we say that $\hat{f}$ is \textit{generated by} $V$. 
	If in addition  for each $B \in \R^{I}$, $V(B,\cdot)$ is divergence-free, then $\hat f$ is  volume preserving  as in Definition \ref{smooth deform}, and we say that $V$ is  \textit{volume preserving}. 
\end{construction}

For any $V$ as in Definition \ref{def infini deform}, we use $\|\cdot\|_X$ to denote the uniform norm of the derivatives of $V$ restricted to $\{0\}\times X$.
The following lemma gives bounds on the norms of deformations induced by infinitesimal deformations.

\begin{lemma}\label{lem compare T V}
Assume that  $r \geq 2$. Let $I \in \N_{\geq 1}$ 
and let  $\hat{f} \colon U \times X \to X$ be a $C^r$ deformation at $(0,f)$ generated by some infinitesimal $C^r$ deformation  with $I$-parameters $V$. Take $T = T(\hat{f})$ as in Definition \ref{smooth deform}. Then there exists a  $C^2$-uniform constant $C_0=C_0(f)>0$, such that by possibly taking $U$ smaller, it holds: 
\enmt
\item  $\norm{DT} < C_0 (1 + \norm{\partial_bV}_X) $ and $\norm{D^2T} < C_0 (1 + \norm{\partial_b\partial_xV}_{X})(1 + \norm{\partial_bV}_{X})$;
\item $\pi_X DT((0,x),B) = V(B,f(x))$
for any $(x,B) \in X\times T_0 U$. Here for each $v \in T(U\times X)$, we denote by $\pi_X(v)$ the component of $v$ in $TX$. 
\eenmt
\end{lemma}
\begin{proof} 
We defer the  proof to  Appendix \ref{app b}. \end{proof}

Some of the estimates will depend on the support of a deformation or of an infinitesimal deformation, which we now define.
\begin{defi} \label{support of deformation} 
For an infinitesimal $C^{r}$ deformation with $I$-parameters $V \colon \R^{I}\times X \to TX$, we define
\begin{align*}
\mathrm{supp}_{X}(V):= \{ x \in X\ \vert\ \exists\, B \in \R^{I} \mbox{ such that } V(B,x) \neq 0 \}.
\end{align*} 
Given $a \in \R^I$, an open neighbourhood $U$ of $a$, and a $C^{r}$ deformation at $(a,f)$ with $I$-parameters $\hat{f} \colon U \times X \to X$, we define
\begin{align*}
\mathrm{supp}_{X}(\hat{f}):= \{x \in X\ \vert\ \exists\, b \in U \mbox{  such that } \hat{f}(b,x) \neq f(x) \}.
\end{align*}
\end{defi}
It is clear from Definitions \ref{def infini deform} and \ref{support of deformation} that for any infinitesimal  $C^{r}$ deformation $V$, if $\hat{f}$ is the $C^{r}$ deformation of $f$ generated by $V$, then we have 
\begin{align}\label{supp V supp hat f}
\mathrm{supp}_{X} (\hat{f}) \subset f^{-1}(\mathrm{supp}_{X}(V)).
\end{align}

\subsection{$c$-disk and $c$-family}\label{subsection c disks} 
 We first introduce the following notion. 
\begin{defi}[Accessibility class]\label{Accessibility class}
	Let $f\colon X \to X$ be a partially hyperbolic diffeomorphism. For any $x \in X$, any $\ell > 0$ and any integer $k \geq 1$, we let $Acc_{f}(x, \ell, k)$ be the set of all points $y \in X$  that can be attained from $x$ through a $k$-legged \textit{accessibility sequence} $x=z_0, z_1, \dots, z_k = y$, where for each $0 \leq i \leq k-1$, $z_{i+1} \in \cWs_{f}(z_i, \ell) \cup \cWu_f(z_i, \ell)$. We let the \textit{accessibility class} of $f$ at $x$ be $Acc_f(x):= \cup_{\ell > 0, k \geq 1} Acc_f(x, \ell, k)$. 
\end{defi}
For any $f\in \cPH^1(X)$, accessibility classes of $f$ form a partition of $X$. By Definition \ref{def acc}, $f$ is accessible if and only this partition consists of a single class.

In the rest of Section \ref{sec random perturbations}, we assume  $f$ satisfies \hyperref[(H1)]{(H1)}. 

\begin{defi}[$c$-disk] 
	For each $x \in X$ and $\sigma>0$, we call $\cC = \cWc_f(x,\sigma)$ the \textit{center disk} of $f$ (or $c$-disk of $f$ for short) centered at $x$ with radius $\sigma$, and we set $\varrho(\cC):=\sigma$. In addition, for any $\theta \in (0, 1 ]$, we also define $\theta \cC:= \cWc_f(x, \theta \sigma)$.
\end{defi}

\begin{defi}
A collection of disjoint center disks $\cal D = \{\cC_1, \dots, \cC_K \}$ is called a \textit{family of center disks for $f$} (or \textit{$c$-family for $f$} for short). In addition, we set 
$$
\underline{r}(\cD):=\inf_{\cC\in \cD}\{\varrho(\cC)\},\quad \overline{r}(\cD):=\sup_{\cC\in \cD}\{\varrho(\cC)\},\quad n(\cD):= K.
$$
Given $\theta \in (0,1)$  and 
$k \in \mathbb{N}_{\geq 1}$, we say that $\cal D$ is a \textit{$(\theta, k)$-spanning} $c$-family for $f$ if 
\begin{equation*}
X = \cup_{\cC \in \cD} \cup_{x \in \theta\cC} Acc_{f}(x, 1, k).
\end{equation*}
Given any subset $\mathcal{C} \subset X$, and $\sigma\geq 0$, we set $(\mathcal{C},\sigma):=\{x \in X\, \vert\, d(x,\mathcal{C})\leq \sigma\}$. 
Given a collection $\cD = \{\cC_1, \dots, \cC_K\}$ of subsets of $X$, we set 
\begin{equation*}
(\cD, \sigma):= \cup_{j=1}^K (\cC_j, \sigma),\qquad \forall\,\sigma \geq 0. 
\end{equation*}

A collection $\cD$  of subsets of $X$ is called \textit{$\sigma$-sparse} if for any two distinct $\cC, \cC' \in \cD$, $(\cC, \sigma), (\cC', \sigma)$ are disjoint. Any $c$-family for $f$ is $\sigma$-sparse for some $\sigma>0$. 

\end{defi}


The next lemma is a consequence of the continuity of the invariant foliations with respect to the dynamics. Roughly speaking, it says that any $c$-family can be slightly perturbed into a $c$-family for a given nearby map. 

\begin{lemma}\label{lemma stability spanning}
	For any  $\theta\in (0,1)$, $\theta' \in (\theta,1)$, $\theta'' \in (0,1]$, $\rho_M > 0$, $\rho_m \in (0,\rho_M)$, and any $\sigma>0$, there exists a $C^1$-open neighbourhood $\mathcal{U}$ of $f$ in $\cPH^1(X)$, and $C^1$-uniform constants $\epsilon>0$ and $\sigma'>0$, such that for any $c$-disk of $f$, denoted by $\mathcal{C}=\cW_f^c(x,\rho)$, with $\rho\in (\rho_m,\rho_M)$; for any  $g \in \cU$; for any $y \in B(x,\epsilon)$, the $c$-disk $\mathcal{C}_g=\cW_g^c(y,\rho)$  of $g$ satisfies $\mathcal{C}_g \subset (\mathcal{C},\sigma)$, $\theta''\mathcal{C}_g \subset (\theta''\mathcal{C},\sigma)$, and 
	\begin{equation}\label{eq lem c family}
	(\theta\mathcal{C},\sigma')\subset \cup_{y \in \theta'\mathcal{C}_g} Acc_g(y,1,2). 
	\end{equation}
\end{lemma}

\begin{proof}
	By letting $\epsilon$ be sufficiently small, we clearly have that for any $g$ sufficiently $C^1$-close to $f$, $\mathcal{C}_g=\cW_g^c(y,\rho)$ satisfies $\mathcal{C}_g \subset (\mathcal{C},\sigma)$ and $\theta''\mathcal{C}_g \subset (\theta'' \mathcal{C},\sigma)$. By Proposition \ref{thmplaqueexpansivetostablydc}, $\cW_g^u$, $\cW_g^{cs}$ (resp. $\cW_g^s$, $\cW_g^{cu}$) exist and are uniformly transverse for all $g$ sufficiently $C^1$-close to $f$. Then \eqref{eq lem c family} follows by letting $\sigma'$ and $d_{C^1}(f,g)$ be sufficiently small compared to $\rho_m$, $\rho_M$ and $\theta'-\theta$. 
\end{proof}

\begin{rema}\label{remarque trois}
	Given any $\theta,\theta',\theta'',\rho_m,\rho_M$ as in Lemma \ref{lemma stability spanning}, any $\sigma>0$, any integer $k \geq 1$, any $(\theta,k)$-spanning $c$-family for $f$, denoted by $\mathcal{D}$, such that $[\underline{r}(\mathcal{D}),\overline{r}(\mathcal{D})]\subset (\rho_m,\rho_M)$, the following is true: for any $g$ sufficiently $C^1$-close to $f$, there exists a $(\theta',k+2)$-spanning $c$-family for $g$, denoted by $\mathcal{D}'$, such that $[\underline{r}(\mathcal{D}'),\overline{r}(\mathcal{D}')]=[\underline{r}(\mathcal{D}),\overline{r}(\mathcal{D})]$, $n(\mathcal{D}')=n(\mathcal{D})$. Indeed, we can apply Lemma \ref{lemma stability spanning} successively for each $\mathcal{C}\in \mathcal{D}$. Moreover, we can also ensure that for each $\mathcal{C}'\in \mathcal{D}'$, we have $\mathcal{C}'\subset (\mathcal{C},\sigma)$ and $\theta''\mathcal{C}'\subset (\theta''\mathcal{C},\sigma)$ for some $c$-disk $\mathcal{C}$ in $\mathcal{D}$. 
\end{rema}

\subsection{Extended map and center subspaces}

Recall that in this subsection, \hyperref[(H1)]{(H1)} holds. Let $\bar{\chi}^{c},\hat{\chi}^{c}, \bar{\chi}^{s}, \bar{\chi}^{u}$ be as in Definition \ref{def partially hyperbolic} so that \eqref{ph ineq 1} to \eqref{ph ineq 4} are satisfied. Let $\xi > 0$ be a constant such that
\begin{equation}\label{definition constant xi}
 \min(\bar{\chi}^{s}+\bar{\chi}^c, \bar{\chi}^{u} - \hat{\chi}^{c},\bar{\chi}^s,\bar{\chi}^u)> \xi.
\end{equation}

 \begin{lemma}\label{lemma T deformation}
\label{ph for T} 
Let $I \in \mathbb{N}_{\geq 1}$,  $a \in \R^{I}$ and let $U \subset \R^I$ be an open neighbourhood  of $a$. Let $\hat{f} \colon U \times X \to X$ be a $C^{r}$ deformation at $(a,f)$ with $I$-parameters. If $U$ is chosen sufficiently small, then 
the map $T=T(\hat f)$ is a $C^r$ dynamically coherent partially hyperbolic system for some $T$-invariant splitting 
\begin{equation*}
T_{b}U \oplus T_{x}X = E^{s}_{T}(b,x) \oplus E^{c}_{T}(b,x) \oplus E^{u}_{T}(b,x),\quad \forall\, (b,x) \in U \times X.
\end{equation*}
Moreover, for any $(b,x) \in U \times X$,   we have
\begin{gather}\label{prop subspace def}
E^{*}_{T}(b,x) = \{0\} \oplus E^{*}_{\hat{f}(b,\cdot)}(x),\qquad  \cW_T^{*}(b,x)=\{b\} \times \cW_{\hat f(b,\cdot)}^{*}(x),\qquad \text{for } *=u,s, \nonumber\\ 
\text{and} \quad E^{c}_{T}(b,x) = Graph(\nu_b(x,\cdot)) \oplus E^{c}_{\hat{f}(b,\cdot)}(x),
\end{gather}
for a unique linear map $\nu_b(x,\cdot) \colon T_bU \to E^{su}_{\hat{f}(b,\cdot)}(x):=E^{s}_{\hat{f}(b,\cdot)}(x)\oplus E^{u}_{\hat{f}(b,\cdot)}(x)$.

If in addition \hyperref[(H2)]{(H2)} holds, then, after reducing the size of $U$, 
$u,s$-holonomy maps between center leaves of $T$ (within distance $1$) are $C^{1}$ 
when restricted to some center-unstable/center-stable leaf, with uniformly continuous, uniformly bounded derivatives. 
\end{lemma}

\begin{proof}  
For small enough $U$,  the map $T$ is a dynamically coherent partially hyperbolic diffeomorphism  (it is $C^1$-close to $(b,x)\mapsto (b,f(x))$). A detailed treatment for this statement can be found in \cite[Section 7]{PSW3}. 

In the following, let $*=u$ or $s$, and let $U$ be small. Then for all $b \in U$, $E_{\hat f(b,\cdot)}^*$ is close to $E_f^*$, and the expansion/contraction rate of $\hat f(b,\cdot)$ along $E_{\hat f(b,\cdot)}^*$ is close to that of $f$ along $E_f^*$. For any  $(b,x) \in U \times X$ and $B+v \in T_b U \oplus T_x X$, 
we have
\begin{equation*}\label{eq derivee dtvb}
DT((b,x),B+v)=B+ [ \partial_b \hat f((b,x),B) + \partial_x\hat f((b,x),v)].
\end{equation*}

Then $DT(b,x)$ maps $\{0\}\oplus E_{\hat f(b,\cdot)}^*(x)$ to $\{0\}\oplus E_{\hat f(b,\cdot)}^*(\hat f(b,x))$, which gives $E^{*}_{T}(b,x) = \{0\} \oplus E^{*}_{\hat{f}(b,\cdot)}(x)$. 
It is direct to check that $\{b\}\times \mathcal{W}_{\hat f(b,\cdot)}^*(x)$ integrates $E_T^*$ restricted to $\{b\}\times X$. Thus  $\cW_T^*(b,x)=\{b\}\times \mathcal{W}_{\hat f(b,\cdot)}^*(x)$. Moreover, it is clear that $\{0\}\oplus E_{\hat f(b,\cdot)}^c(x)\subset E_T^c(b,x)$, and $E_T^c(b,x)\cap (\{B\}\oplus T_x X) \neq \emptyset$ for any $B \in T_b U$. We define $\nu_b(x,B)$ as the unique vector in $E_{\hat f(b,\cdot)}^{su}(x)$ such that $B+\nu_b(x,B)\in E_T^c(b,x)$. It is direct to see that $\nu_b(x,\cdot)$ is a linear map. 

Now, if $r \geq 2$ and $f$ is center bunched, by $C^{1}$-openness of center bunching, for sufficiently small $U$, we can verify that $T^n$ is also center bunched for some $n \in \mathbb{N}$. The smoothness of $s,u$-holonomy maps of $T$ follows from Proposition \ref{cor smooth holonomy maps}.
\end{proof}

Let $U$, $\hat f$ and $T$ be as in Lemma \ref{lemma T deformation}. 
In the following, for any $(b,x) \in U \times X$, we will tacitly use the inclusions $E^{*}_{\hat{f}(b,\cdot)}(x) \hookrightarrow \{0\} \oplus E^{*}_{\hat{f}(b,\cdot)}(x) \subset T_{b}U\oplus T_{x}X$ for $* = s,u,c$, and the isomorphism  $\R^{I} \simeq T_bU \oplus \{0\} \subset T_bU \oplus T_xX$.

For any $(b,x) \in U\times X$ and $v=B+v' \in T_{b}U \oplus T_{x}X$, we denote by $\pi_{X}(v):=v'$ the component  in $T_{x}X$, and set $\pi_b(v):=B+\nu_b(x,B)$. We also denote by $\pi_*(v)$ the component of $v$ in $E_{\hat f(b,\cdot)}^*$ for $*=u,s,c$. By a slight abuse of notation, we let $\pi_X(b,x):=x$.

We introduce the following definitions, motivated by the need to control return times of a map to the support of a deformation. 

\begin{defi} \label{def rec funtion}
	For any subsets $A, B \subset X$, we define
	\begin{align*}
	R(f, A, B) &:= \inf \{ n \geq 0 \ \vert \ f^{n}(A) \cap B \neq \emptyset \mbox{ or } f^{-n}(A) \cap B \neq \emptyset\};\\
	R_{\geq 0}(f, A, B) &:= \inf \{n \geq 0 \ \vert\ f^{n}(A) \cap B \neq \emptyset  \}; \\
	 R_{\star}(f, A, B) &:= \inf \{ n \geq 1 \ \vert\ f^{\star n}(A) \cap B \neq \emptyset \},\quad \star=+,-.
	\end{align*}
	For any subset $A \subset X$, we use the abbreviation $R_{\star}(f,A):= R_{\star}(f,A,A)$, $\star=+,-$. 
	
	Similarly, for a $C^1$ deformation of $f$, $\hat{f} \colon U \times X \to X$, we set
	\begin{equation*}
	R_{\star}(\hat{f}, A, B):= \inf \{ n \geq 1 \ \vert\ \exists\, b \in U \mbox{ such that } \hat{f}(b,\cdot)^{\star n}(A) \cap B \neq \emptyset \}, \quad \star = +, -.
	\end{equation*}
	We define $R(\hat{f}, A,B)$, $R_{\geq 0}(\hat{f}, A, B)$, $R_{+}(\hat{f},A)$ and $R_{-}(\hat{f},A)$  in an analogous way.
	Moreover, it is clear that $R(\hat{f},A,B) = \min(R_{\geq 0}(\hat{f}, A, B), R_{-}(\hat{f}, A, B))$.
\end{defi}

In the following, for $*= s,u$, and for any $p \in M$, we set 
\ary \label{def of chikstar}
\bar{\chi}_k^*(p):= \begin{cases}
	\sum_{j=0}^{k-1} \bar{\chi}^*(f^j(p))>0, & \forall\, k \geq 0, \\
	\sum_{j=k}^{-1} \bar{\chi}^*(f^j(p))>0, & \forall\, k <0.
\end{cases} 
\eary
We  define $\bar{\chi}_k^c(p)$ and $\hat{\chi}_k^c(p)$ in a similar way.

The following lemma collects some basic properties of the center bundle $E_T^c$.



\begin{lemma} \label{lem property of nu}
	We have
\enmt
\item\label{lem trans v outside support}
For any $x \in X \backslash \mathrm{supp}_{X}(\hat{f})$, any $B \in T_0U$,  
\begin{equation*}
DT(B + \nu_0(x,B)) = B + \nu_0(f(x),B).
\end{equation*}
Equivalently, $Df(x, \nu_0(x,\cdot))=\nu_0(f(x), \cdot)$.  
\item\label{lem a priori bound for nu}
$\sup_{x \in X} \norm{\nu_0(x, \cdot)} \leq C_1 \norm{T}_{C^1}$
for a $C^1$-uniform 
constant $C_1 = C_1(f) > 0$.
\item \label{lem apriori-refined est for nu}
There is a $C^1$-uniform constant $C_2=C_2(f)>0$  s.t. $\forall\, (x,B) \in X \times T_0U$,
\begin{equation*}
\norm{\nu_0(x, B)} \leq C_2\|T\|_{C^1} e^{- \kappa(\hat f, x)}\norm{B},
\end{equation*}
where  for $w \in M$, we let
\begin{equation} \label{first line}
\kappa(\hat f, w):=\min (\bar\chi^u_{R_{\geq 0}(\hat{f}, w, \mathrm{supp}_X(\hat{f}))}(w), \bar\chi^s_{-R_{-}(\hat{f}, w, \mathrm{supp}_X(\hat{f}))}(w))>0. 
\end{equation}
\eenmt
\end{lemma}
\begin{proof}
Proof of \eqref{lem trans v outside support}:
For any $x \in X \backslash \mathrm{supp}_{X}(\hat{f})$, any $B \in T_0U$,  we have  $DT((0,x), B) = B$.
We have $DT(\nu_0(x,B)) \in E^{su}_T(0,f(x))$ and by \eqref{prop subspace def}, $DT(B + \nu_0(x,B)) \in  E^{c}_{T}(0,f(x))$.
Thus $$DT(B + \nu_0(x,B)) \in (B+E^{su}_T(0,f(x))) \cap E^{c}_{T}(0,f(x)),$$ 
while the right hand side contains only $B + \nu_0(f(x), B)$.

Proof of \eqref{lem a priori bound for nu}: 
For any $(x, B) \in X\times  T_0 U$, the unstable part of $\nu_0(x,B)$ satisfies
\begin{equation*}
\pi_{u}\nu_0(x,B) = -\sum_{n=1}^{+\infty}Df^{-n} (f^{n}(x),\pi_u\partial_b\hat{f}((0,f^{n-1}), B)).
\end{equation*}
Then $\|\pi_u \nu_0\|\leq C_1 \|T\|_{C^1}$ by $\|\partial_b \hat f\|\leq \|T\|_{C^1}$ and $\|Df^{-1}|_{E_f^u}\|<1$. 
Arguing similarly for the stable part, we conclude the proof.

Proof of \eqref{lem apriori-refined est for nu}:
By \eqref{lem trans v outside support}, for any $x \in X$ and $0 \leq n < R_{-}(\hat{f}, \{ x \}, \mathrm{supp}_X(\hat{f}))$, we have 
\begin{equation*}
\pi_s\nu_0(x,B) = Df^{n}(\pi_s\nu_0(f^{-n}(x),B)),\quad \forall\, B\in T_0 U. 
\end{equation*}
By Lemma \ref{lem property of nu}\eqref{lem a priori bound for nu} and \eqref{ph ineq 2}, for some $C^1$-uniform constant $C_1 = C_1(f) > 0$, we thus have  
\begin{equation*}
\norm{\pi_s\nu_0(x,\cdot)} \leq C_1\norm{T}_{C^1} e^{-\bar\chi^s_{-R_{-}(\hat{f}, x, \mathrm{supp}_X(\hat{f}))}(x)}. 
\end{equation*} 
Similarly, we have $\norm{\pi_u\nu_0(x,\cdot)} \leq C_1\norm{T}_{C^1} e^{-\bar\chi^u_{R_{\geq 0}(\hat{f}, x, \mathrm{supp}_X(\hat{f}))}(x)}$. 
\end{proof}

\subsection{Holonomy maps}\label{subs holo map}

Recall that in this section, \hyperref[(H1)]{(H1)} and \hyperref[(H2)]{(H2)} are satisfied. 
In the following, we fix  an integer $I >0$, and let $\hat{f} \colon U \times X \to X$ be a  $C^{2}$ deformation at $(0,f)$ with $I$-parameters.  We set $T = T(\hat{f})$.  In the following, we will always take $U$ conveniently small so that by Lemma \ref{lemma T deformation}, the stable and unstable holonomy maps for $T$ between close center leaves of $T$ are $C^1$.

We need bounds for the derivatives of  holonomy maps with respect to parameters. The following lemma is  proved by combining the construction in \cite[Proof of Theorem A]{PSW} and invariant section theorem for jets in \cite[Proof of Theorem 3.2	]{HPS}. We defer the technical proof to the appendix.

\begin{lemma} \label{prop holon map and nu}
Let $*=u,s$. Take $x \in X$, $y \in \cW^{c*}_{f}(x,\sigma_f/2)$ and set $z:= H^{*}_{f,x,y}(x)$. 
\begin{enumerate}
\item \label{lemma first coordinate}
For any $B \in T_0U$, we have 
\begin{equation*}
\pi_bDH^*_{T, (0,x), (0,y)}(B + \nu_0(x,B)) = B + \nu_0(z,B).
\end{equation*}
\item\label{lemma cent coordinate}
There exists a $C^2$-uniform constant $C_3 = C_3(f) > 0$ s.t. for any $B \in T_0U$, 
\begin{align*}
&\norm{\pi_cDH^*_{T,(0,x),(0,y)}(B + \nu_0(x,B))} \\ 
&\leq C_3 (\max(e^{- \kappa(\hat f,x)},e^{- \kappa(\hat f,z)}) \norm{DT}+ \norm{D^2T}d_{\cW_f^*}(x,z)) \norm{B}.
\end{align*}
\end{enumerate}
\end{lemma}
\begin{proof}
	In Appendix \ref{app a}.
\end{proof}

 The following proposition provides fine control of the derivatives of  holonomy maps with respect to parameters when we are given  certain recurrence condition. We give some illustration in Figure \ref{fig 1}.

\begin{prop} \label{refined est}
There exists a $C^2$-uniform constant $C_4=C_4(f) > 0$ such that the following is true. Fix any $R_0,C> 0$, $\sigma \in (0, \sigma_f)$. Assume that $\hat{f}$ is generated by $V$, an infinitesimal $C^{2}$ deformation such that
 $\sigma \norm{\partial_b \partial_xV}_{X} + \norm{\partial_bV}_{X}< C$. Then there exists a $C^1$-uniform constant $\xi' > 0$ such that we have the following:
\begin{enumerate}
\item Let $y \in \cWcu_f(x, \sigma)$ and $z := H^{u}_{f, x, y}(x)$. Assume that $x \notin \mathrm{supp}_{X}(V)$ and $R_{-}(\hat{f}, \{x, z\}, \mathrm{supp}_X(V)) > R_0$.
Then for any $B \in T_0U$,
\begin{equation*}
\norm{\pi_cDH^{u}_{T, (0, x), (0, y)}(B + \nu_0(x,B))- \pi_cV(B, z)} \leq C_4 C^3 e^{-R_0\xi '}\norm{B}.
\end{equation*}
\item Let $y \in \cWcs_f(x, \sigma)$ and $z := H^{s}_{f, x, y}(x)$. Assume that $R_{+}(\hat{f}, \{ x, z\}, \mathrm{supp}_X(V)) > R_0$. Then for any $B \in T_0U$, 
\begin{equation*}
\norm{\pi_cDH^{s}_{T, (0, x), (0, y)}(B + \nu_0(x,B)) } \leq C_4 C^3 e^{-R_0\xi '}\norm{B}.
\end{equation*}
\end{enumerate}
Note that the terms on the RHS of the above inequalities are independent of $\sigma$. 
\end{prop}

\begin{figure}[h] 
	\begin{center}
		\includegraphics [width=12cm]{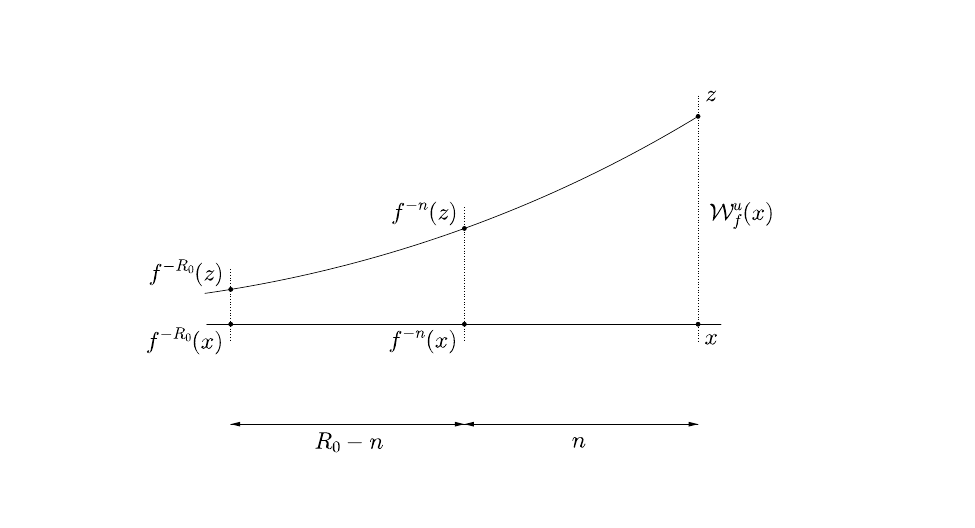}
	\end{center}
	\caption{The point $f^{-j}(x)$ for any integer $1 \leq j \leq R_0$, and the point $f^{-k}(z)$ for any integer $2 \leq k \leq R_0$ lie outside of ${\rm supp}(\hat f)$. We apply Lemma \ref{prop holon map and nu} to $f^{-n}(x)$ and $f^{-n}(z)$ where $n$ is a small fraction of $R_0$.}\label{fig 1}
\end{figure}

\begin{proof}
We first prove (1). Without loss of generality, we assume that 
$R_0$ is sufficiently large.

Let $1 \leq n\leq R_0-1$. Successive application of Lemma \ref{lem property of nu}\eqref{lem trans v outside support} gives
\begin{equation*} \label{D T - 1}
DT^{-n}(B + \nu_0(x,B)) = B + \nu_0(f^{-n}(x),B).
\end{equation*}
Then the invariance of the foliations under the dynamics yields
\begin{align}
&\pi_{c}DH^{u}_{T, (0, x), (0, y)}(B + \nu_0(x,B))  \nonumber  =\pi_{c}DH^{u}_{T, (0, x), (0, z)}(B + \nu_0(x,B))  \nonumber   \\
=\ &\pi_{c}DT^{n}((\pi_b + \pi_{c})DH^{u}_{T, T^{-n}(0, x), T^{-n}(0, z)}DT^{-n}(B + \nu_0(x,B))) \nonumber \\
 =\  &\pi_{c}DT^{n}((\pi_b + \pi_{c})DH^{u}_{T, T^{-n}(0, x), T^{-n}(0, z)}(B + \nu_0(f^{-n}(x), B))),\label{expan1}
\end{align}
where we have used that $DH^{u}_{T, T^{-n}(0, x), T^{-n}(0, z)}(B +  \nu_0(f^{-n}(x),B))\in E_T^c(T^{-n}(0,z))$. 

\begin{claim}
There exists a  $C^2$-uniform constant $c_1=c_1(f) > 0$ such that
\begin{align*}
\norm{\pi_{c}DH^{u}_{T, T^{-n}(0, x), T^{-n}(0, z)}(B + \nu_0(f^{-n}(x), B))} 
\leq c_1 C^2e^{-\min(\bar{\chi}^{u}_n(f^{-n}(x)),\bar{\chi}_{-R_0 + n}^{s}(f^{-n}(x)))} \norm{B}. 
\end{align*}
\end{claim}
\begin{proof}
By $y \in \cWcu_f(x,\sigma)$ and by distortion estimates,  for some $C^1$-uniform constant $c_{2}=c_{2}(f) > 0$, we obtain 
\begin{equation*}
d_{\cWu_f}(T^{-n}(0,x), T^{-n}(0,z)) <  e^{-\bar{\chi}^{u}_{-n}(x)}d_{\cWu_f}(x,z) < c_{2} e^{-\bar{\chi}^{u}_{-n}(x)}\sigma.
\end{equation*}
By Lemma \ref{lem compare T V}, there  exist $C^2$-uniform constants $c_{i}=c_i(f) > 0$, $i=3,4,5$,  so that 
\begin{equation*} 
\norm{DT} < c_3 C, \quad
\norm{D^2T} < c_{4} (1 + \norm{\partial_b \partial_x V}_{X})(1 + \norm{\partial_bV}_{X})  < c_{5} C^2\sigma^{-1}.
\end{equation*}
Recall that $R_-(\hat f, \{x,z\},\mathrm{supp}_X(V)) > R_0>n$  by hypothesis,  thus by \eqref{supp V supp hat f},  
\begin{align*}
R_{\geq 0}(\hat{f}, \{f^{-n}(x), f^{-n}(z)\}, \mathrm{supp}_X(\hat{f})) &\geq n-1, \\
R_{-}(\hat{f}, \{f^{-n}(x), f^{-n}(z) \}, \mathrm{supp}_X(\hat{f})) &\geq R_0 - n. 
\end{align*}
Thus
\begin{equation}  \label{eq longterm1}
\kappa(\hat f,f^{-n}(x)) \geq \min(\bar{\chi}^{u}_{n-1}(f^{-n}(x)),\bar{\chi}_{-R_0 + n}^{s}(f^{-n}(x))).
\end{equation}
Similarly, we have
\begin{equation} \label{eq longterm2}
\kappa(\hat f,f^{-n}(z)) \geq \min(\bar{\chi}^{u}_{n-1}(f^{-n}(z)),\bar{\chi}_{-R_0 + n}^{s}(f^{-n}(z))).
\end{equation}
On the other hand, by distortion estimates, we have
\begin{equation*}
|\mbox{RHS of } \eqref{eq longterm1} - \mbox{RHS of } \eqref{eq longterm2}| < C'
\end{equation*}
for some $C^2$-uniform constant $C' > 0$.
Thus the claim follows from Lemma \ref{prop holon map and nu}\eqref{lemma cent coordinate}.
\end{proof} 
There is a $C^1$-uniform constant $\eta_0 \in (0,1)$  such that for any
integer  $n \in (\eta_0 R_0,  2\eta_0 R_0)$, we have
\begin{equation*}
 \bar{\chi}^{u}_n(f^{-n}(x)) \leq \bar{\chi}_{-R_0 + n}^{s}(f^{-n}(x)).
\end{equation*}
By the above claim, we get 
 \begin{equation} \label{pi c d h}
\norm{\pi_{c}DH^{u}_{T, T^{-n}(0, x), T^{-n}(0, y)}(B + \nu_0(f^{-n}(x), B))} \leq 2c_{1} C^2 e^{-\bar{\chi}^{u}_{-n}(x)}\norm{B}.
\end{equation}

By Lemma \ref{prop holon map and nu}\eqref{lemma first coordinate}, we have
\begin{equation} \label{pi b D H u T}
\pi_{b}DH^{u}_{T, T^{-n}(0, x), T^{-n}(0, z)}(B + \nu_0(f^{-n}(x), B)) =B + \nu_0(f^{-n}(z),B).
\end{equation}

By Lemma \ref{lem property of nu}\eqref{lem trans v outside support}, we also have 
\begin{equation}\label{eq 3.17}
DT^n(B+\nu_0(f^{-n}(z),B))=DT(B+\nu_0(f^{-1}(z),B)). 
\end{equation}

By \eqref{expan1}-\eqref{eq 3.17}, and since for some $C^1$-uniform constant $c_{6}=c_{6}(f) > 0$, 
\begin{equation*}
\norm{DT^{n}(0,\cdot)|_{E^{c}_{T}(w)}}< c_{6} C e^{\hat{\chi}^{c}_n(w)}, \quad \forall\, w \in M,
\end{equation*}
we deduce that for some $C^2$-uniform constant $c_{7}=c_{7}(f) > 0$, it holds 
\begin{align*}
&\norm{\pi_cDH^{u}_{T, (0,x), (0,y)}(B + \nu_0(x,B))  - \pi_cDT(B + \nu_{0}(f^{-1}(z), B)) } \\
\leq \ & c_{7} C^3 e^{\hat{\chi}^{c}_{-n}(x) - \bar{\chi}^{u}_{-n}(x)}\norm{B} \leq c_{7} C^3 e^{-n\xi}\norm{B} \leq c_{7} C^3 e^{-R_0\eta_0\xi}\norm{B}.
\end{align*}
We conclude (1) by setting $\xi' := \eta_0 \xi$ and by using Lemma \ref{lem compare T V}(2), that is,
$$
\pi_c DT(B+\nu_0(f^{-1}(z),B))=\pi_c DT((0,f^{-1}(z)),B)=\pi_c V(B,z). 
$$

Under condition (2), a similar argument (by choosing some $n$ comparable to $R_0$) shows that for some $C^2$-uniform constant $c_8=c_8(f)>0$, we have
\begin{equation*}
\norm{\pi_cDH^{s}_{T, (0,x), (0,y)}(B + \nu_0(x,B)) } \leq c_{8} C^3 e^{-(\bar{\chi}^{s}_n(x) + \bar{\chi}^{c}_n(x) )}\norm{B} \leq c_8 C^3 e^{- R_0 \xi'}\norm{B}.
\end{equation*}
\end{proof}


\section{Submersion from parameter space to phase space}\label{from dim gap}

In this section,
 we will estimate the measure of parameters in a $C^{r}$ deformation corresponding to certain ``unlikely coincidences''. 
First,  we need to estimate the derivatives (with respect to parameters) of holonomy maps along certain $su$-paths.

Throughout this section, we assume that \hyperref[(H1)]{(H1)} and \hyperref[(H2)]{(H2)} hold. 
\begin{defi}\label{def 4 legged}
Given $x \in X$,  a triplet $\gamma = (x_1,x_2,x_3) \in X^3$  is called a \textit{$f$-loop} at $x$  if the following holds:
$$x_1 \in \cWu_{f}(x), \quad x_2 \in \cWs_{f}(x_1), \quad x_3 \in \cWu_{f}(x_2),\quad x \in \cWcs_{f}(x_3).$$
The \textit{length} of $\gamma$ is defined as $$\ell(\gamma):=d_{\cWu_f}(x,x_1) + d_{\cWs_f}(x_1,x_2) + d_{\cWu_f}(x_2,x_3) + d_{\cWcs_f}(x_3,x).$$

 By points \eqref{item 1 definition 1}-\eqref{item 3 definition 1} in Notation \ref{notation 1},
for each $f$-loop $\gamma$ such that $\ell(\gamma) =: \sigma < \frac{\sigma_f}{2}$, we have a well-defined map $H_{f, \gamma} \colon \cWc_f(x, \Lambda_f^{-4}\sigma) \to \cWc_f(x, (1+C_f)\sigma)$:
$$H_{f,\gamma}:= H^s_{f,x_3,x}H^u_{f,x_2,x_3}H^s_{f,x_1,x_2}H^u_{f,x,x_1}.$$

 Let $\hat{f} \colon U \times X \to X$ be a $C^1$ deformation at $(0,f)$, and let $T = T(\hat{f})$.  For any  $f$-loop $\gamma = (x_1, x_2, x_3)$ at $x$,  we define the \textit{lift} of $\gamma$ for $T$ as  $\hat{\gamma}:= ((0,x_1), (0,x_2),(0,x_3))$.

\end{defi}

\begin{nota}\label{notation 2}
 Recall that $c= \dim E^c_f$, and let 
\begin{equation}\label{term 2011}
K_f:=cC_f\Lambda_f^{4c}.
\end{equation}
We fix a $C^2$-uniform constant $\overline{\sigma}_f \in \big(0, \frac{\sigma_f}{100C_fK_f}\big)$ such that
for any $x \in X$, any collection $\{\gamma_{j} = (x_{j,1}, x_{j,2}, x_{j,3})\}_{j=1,\dots,c}$ of $f$-loops at $x$ such that $\ell(\gamma_j) < \overline{\sigma}_f, \forall\, 1 \leq j \leq c$,  the map $\prod_{j=1}^c H_{f, \gamma_j}$ is defined on $\cWc_f(x, \overline{\sigma}_f)$.
In this case, for any $1 \leq k \leq c+1$ we set $y_{k}:= \prod_{l=1}^{k-1} H_{f, \gamma_{l}}(x)$, and for $1 \leq j \leq c$ we define
\begin{equation*}
y_{j,1}:= H^{u}_{f,x,x_{j,1}}(y_j), \quad y_{j,2}:= H^{s}_{f, x_{j,1}, x_{j,2}}(y_{j,1}), \quad y_{j,3}:= H^{u}_{f,x_{j,2}, x_{j,3}}(y_{j,2}).
\end{equation*}
\end{nota}

The next lemma follows from Notation \ref{notation 1}, \ref{notation 2} by straightforward computations. We thus omit its proof. 
\begin{lemma} \label{lem41} 
Let $f,x, \gamma_j, y_j, y_{j,k}$ be  as in Notation \ref{notation 2}. Assume that for some $\sigma \in (0,\overline{\sigma}_f)$, we have $\ell(\gamma_j) < \sigma$, for all $j=1,\dots, c$. Then for any $j=1,\dots,c$, $d_{\cWc_f}(x,y_{j+1})\leq C_f \sigma + \Lambda_f^4 d_{\cWc_f}(x,y_{j})$. Let $y_{j,0}:=y_j$ and $y_{j,4}:=y_{j+1}$. We have
\begin{equation*}
\begin{array}{rcl} 
d_{\cWc_f}(x, y_j), d_{\cWc_f}(x_{j,k}, y_{j,k}) < K_f  \sigma < \frac{\sigma_f}{10}, &\qquad& \forall\, k=1,2,3,\\
d_{\cWu_f}(y_{j,k-1}, y_{j,k}),d_{\cWs_f}(y_{j,k}, y_{j,k+1})  < 3C_fK_f\sigma < \frac{\sigma_f}{10}, &\qquad& \forall\, k=1,3.
\end{array}
\end{equation*}
\end{lemma}

In the following, we let $I>0$ be some integer,  let $V$ be an infinitesimal $C^2$ deformation with $I$-parameters, and  let  $\hat{f}\colon U \times X \to X$ be the $C^{2}$ deformation at $(0,f)$ generated by $V$. Set $T = T(\hat{f})$. Let $\xi'$ be defined as in Proposition \ref{refined est}.

\begin{lemma}  \label{prop linear approx} 
There exists a $C^2$-uniform constant $C_5=C_5(f) > 0$ such that the following is true. Let $\sigma\in (0,\overline{\sigma}_f)$, $x \in X$ and let $\gamma=(x_1,x_2,x_3)$ be a $f$-loop at $x$ with $\ell(\gamma) < \sigma$. Let $x_4:= H_{f, \gamma}(x)$,  and let $C, R_0 > 0$ satisfy that:
\enmt
\item $\sigma\norm{\partial_b\partial_xV}_{X} + \norm{\partial_bV}_{X} < C$;
\item $R(f, \{x, x_{2}, x_{3},x_4\}, \mathrm{supp}_X(V)) > R_0$;
\item $R_{\pm}(f, \{x_1\}, \mathrm{supp}_X(V)) > R_0$.
\eenmt 

Let $\hat{\gamma}$ be the lift of $\gamma$ for $T$.
Then, the holonomy map $ H_{T, \hat{\gamma}}$ is $C^1$ in an open neighbourhood of $(0,x)$ in $\cWc_{f}(x)$, and for any $B \in T_0 U$, we have
\begin{equation*}
\norm{\pi_cDH_{T,\hat{\gamma}}(B + \nu_0(x,B)) - D(H^{s}_{f, x_3, x} H^{u}_{f, x_2, x_3} H^{s}_{f, x_1, x_2})(\pi_cV(B, x_1))} \leq  C_5 C^{3}e^{-R_0\xi'} \norm{B}.
\end{equation*}
\end{lemma}
\begin{proof}  
Let $x:=x_0 \in X$. By definition, we have
\begin{equation*} 
H_{T,\hat{\gamma}}  
= H^s_{T, (0,x_3), (0, x)}H^u_{T,(0,x_2), (0, x_3)}H^s_{T, (0, x_1),(0, x_2)}H^u_{T, (0,x), (0, x_1)}.
\end{equation*}
Since $f$ is center bunched, Lemma \ref{ph for T} implies that $H^{u}_{T, (0,x), (0,x_1)}$, $H^{s}_{T,(0, x_1),(0, x_2)}$, $H^{u}_{T,(0, x_2),(0, x_3)}$ and $H^{s}_{T,(0, x_3),(0, x)}$ are $C^{1}$ if $U$ is small enough. 
Given any $ B \in T_{0}U$, let us calculate 
$DH_{T,\hat{\gamma}}(B + \nu_0(x,B))$. 
Set $I_0(B):= B + \nu_0(x,B) \in E^c_{T}(0, x)$. For $i = 1,\dots, 4$, we define
\begin{equation}\label{eq Ii DH}
I_i(B):= DH^{*_i}_{T, (0, x_{i-1}), (0, x_i)}(I_{i-1}(B))  \in E^{c}_T(0, x_i),
\end{equation}
where $*_i= u$ if $i=1,3$ and $*_i= s$ if $i=2,4$.
In particular, we have $DH_{T,\hat{\gamma}}(B + \nu_0(x,B))  = I_4(B)$.
Then by Lemma \ref{prop holon map and nu}\eqref{lemma first coordinate} and simple induction we deduce that
\begin{equation}\label{recurrence IicB}
I_{i}(B) = I_{i}^{c}(B) + (B + \nu_0(x_i, B)), \quad \forall\, i=1,\dots,4,
\end{equation}
with $I_{i}^{c}(B):= \pi_{c}(I_{i}(B))$. By \eqref{eq Ii DH}, we thus obtain
\begin{align}\label{I i+1 I i}
I_{i}^{c}(B) = DH^{*_i}_{f, x_{i-1}, x_{i}}(I_{i-1}^{c}(B)) + \pi_{c}DH^{*_i}_{T, (0,x_{i-1}), (0,x_{i})}(B + \nu_0(x_{i-1}, B)).
\end{align} 
By the hypothesis we made 
on $V$, 
$x_0, x_2, x_3, x_4 \notin \mathrm{supp}_X(V)$ and $R_{\pm}(f, x_i, \mathrm{supp}_X(V)) > R_0$ for $0 \leq i \leq 4$. We can apply Proposition \ref{refined est} to obtain 
\begin{align*}
\norm{\pi_{c}DH^{u}_{T, (0,x), (0,x_{1})}(B + \nu_0(x, B)) - \pi_cV(x_1,B)}  &\leq C_4 C^3 e^{-R_0\xi'}\norm{B},\\
\norm{\pi_{c}DH^{s}_{T, (0,x_{i-1}), (0,x_{i})}(B + \nu_0(x_{i-1}, B))} &\leq C_4 C^3 e^{-R_0\xi'}\norm{B}, \quad i=2,4,\\
\norm{\pi_{c}DH^{u}_{T, (0,x_2), (0,x_{3})}(B + \nu_0(x_{2}, B))} &\leq C_4 C^3 e^{-R_0\xi'}\norm{B}.
\end{align*}

Combining this with \eqref{recurrence IicB} and \eqref{I i+1 I i}, we see that there exists a $C^2$-uniform constant $C_5=C_5(f) > 0$ such that
\begin{equation*}
\norm{I_4^{c}(B) - \pi_c D(H^{s}_{f, x_3, x} H^{u}_{f, x_2, x_3} H^{s}_{f, x_1, x_2})(\pi_cV(B, x_1)) } \leq C_5 C^{3}e^{-R_0\xi'}\norm{B}.
\end{equation*}
\end{proof}

The following definition is motivated by Lemma \ref{lem41} and Lemma \ref{prop linear approx}.  The next proposition roughly says that if we have enough control on the magnitude of the deformation and on the return times to the support of the perturbation, then we can obtain a lower bound on the determinant of  the differential of a certain  map  from parameter space to phase space (see Figure \ref{fig 2} for an illustration).
This will be important  in the parameter exclusion  which appears in Section \ref{section generic open acc}. 

\begin{defi}\label{defiadaptedtosth}
Given any $ \sigma \in (0, \overline{\sigma}_f)$, $C, R_0 > 0$, let $\gamma= (x_{1},x_{2},x_{3})$ be a $f$-loop at a point $x \in X$ with $\ell(\gamma)<\sigma$.  We say that $V$ is \textit{adapted to $(\gamma, \sigma, C, R_0)$} if:
\enmt
\item  $\sigma\norm{\partial_b\partial_xV}_{X} + \norm{\partial_bV}_{X} < C$;
\item $R(f, \cWc_f(z, K_f\sigma), \mathrm{supp}_X(V)) > R_0$ for $z= x, x_{2}, x_{3}$;
\item $R_{\pm}(f, \cWc_f(x_{1}, K_f\sigma), \mathrm{supp}_X(V)) > R_0$.
\eenmt
\end{defi}

\begin{prop}\label{determinant for smooth deformations} 
For any integer $L  > 0$, real numbers $C, \kappa > 0$,  there exist $C^2$-uniform constants $R_0  = R_0(f, L,c, C, \kappa)> 0$ and $\kappa_0 = \kappa_0(f, L, c, C, \kappa) > 0$ such that the following is true. 

Let $x \in X$, $\sigma \in (0, (12C_fK_f)^{-1}\overline{\sigma}_f)$. For each $1 \leq i \leq L$, $1 \leq j \leq c$, let $\gamma_{i,j} = (x_{i,j,1}, x_{i,j,2}, x_{i,j,3})$ be a $f$-loop  at $x$ of length at most $\sigma$ such that  $V$ is adapted to $(\gamma_{i,j}, \sigma,C, R_0)$. Denote by  $B = (B_{\alpha})_{1 \leq \alpha \leq I}$ an element of $T_{0}U = \R^{I}$. Assume that for some integer $1 \leq j_0 \leq c$, and   indices $\{\alpha_{i,j}\}_{1\leq i \leq L, 1\leq j\leq c } \subset \{1,\dots, I\}$, we have for any $1 \leq i, k \leq L$ and $1 \leq j \leq c$ that: if $i\neq k$ or $j \neq j_0$, then for all $z \in \cWc_f(x_{i,j,1}, K_f 
\sigma)$, we have
\begin{equation} \label{determinant 1}
D_{B_{\alpha_{k,1}}, \dots, B_{\alpha_{k,c}} }(\pi_cV(B, z)) = 0, 
\end{equation}
while for any $z \in \cWc_f(x_{i,j_0,1}, K_f 
\sigma)$, we have
\begin{equation} \label{determinant 2}
\big|\det(B \mapsto D_{B_{\alpha_{i,1}}, \dots, B_{\alpha_{i,c}} }(\pi_cV(B, z)) )\big| > 2\kappa.
\end{equation}

Let $\hat{\gamma}_{i,j}$ be the lift of $\gamma_{i,j}$ for $T$, and set $z_i:= \prod_{j=1}^{ c}H_{f, \gamma_{i,j}}(x)$. Then there exists a linear subspace $H \subset T_{0}U = \R^{I}$ of dimension $Lc$ such that 
\begin{equation*}
\det(\Xi|_{H}) \geq \kappa_0,
\end{equation*}
where we set $\Xi \colon\left\{\begin{array}{rcl}
T_0 U &\to&  \prod_{i=1}^{L} E^c_{f}(z_i),\\
B &\mapsto& \left(\pi_cD(\prod_{j=1}^c H_{T, \hat{\gamma}_{i,j}})(B+\nu_0(x,B))\right)_{i=1,\dots,L}.
\end{array} 
\right.$
\end{prop}

\begin{figure}
	\begin{center}
		\includegraphics [width=10cm]{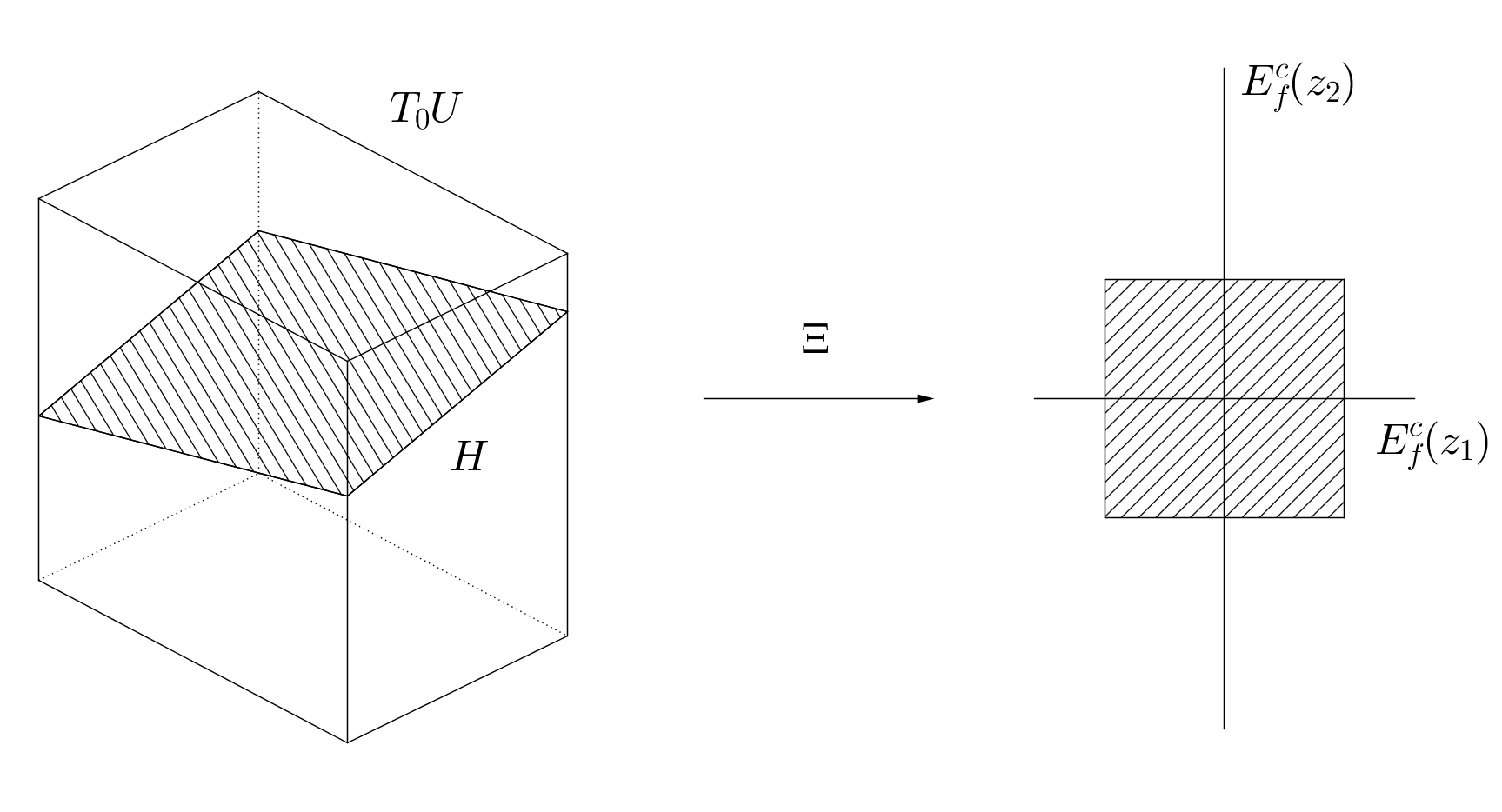}
	\end{center}
	\caption{The map  $\Xi$ is a submersion. Here, the picture is for $L=2$, $H$ is a $2D$ subspace of $T_0 U$, and $\Xi|_{H}$ is invertible. }\label{fig 2}
\end{figure}


\begin{proof} 
For each $1 \leq i \leq L$, and $1 \leq j \leq c$, 
we define
\begin{align*}
y_{i,j}&:= \prod_{l=1}^{j-1} H_{f, \gamma_{i,l}}(x), 
& y_{i,j,1} &:= H^{u}_{f, x, x_{i,j,1}}(y_{i,j}),\\
y_{i,j,2}&:=  H_{f,x_{i,j,1},x_{i,j,2}}^s(y_{i,j,1}),
& y_{i,j,3} &:= H^{u}_{f, x_{i,j,2}, x_{i,j,3}}(y_{i,j,2}).
\end{align*}
By the choice of $\overline{\sigma}_f$ in Notation \ref{notation 2}, Lemma \ref{lem41} yields 
\begin{equation} \label{term 3000}
y_{i,j} \in \cWc_f(x, K_f 
\sigma), \quad y_{i,j,k} \in \cWc_f(x_{i,j,k}, K_f 
\sigma),\quad \forall\, k=1,2,3.
\end{equation}
Denote by $\gamma'_{i,j} = (y_{i,j, 1}, y_{i,j, 2}, y_{i,j, 3})$ the associated $f$-loop at $y_{i,j}$. Note that we also have $z_i \in \cW_f^c(x,K_f\sigma)$, for all $1 \leq i \leq L$. 

By assumption, for any $1 \leq i \leq L$, $1 \leq j \leq c$, we have $\ell(\gamma_{i,j})<\sigma$, thus by Lemma \ref{lem41}, $\ell(\gamma'_{i,j}) \leq 12 C_f K_f\sigma < \overline{\sigma}_f$. 
Since $V$ is $(\gamma_{i,j}, \sigma, C, R_0)$-adapted, we get
\begin{equation*}
 R(f, \{y_{i,j}, y_{i,j,2}, y_{i,j,3}\}, \mathrm{supp}_X(V))>R_0,\quad  R_{\pm}(f, \{y_{i,j,1}\}, \mathrm{supp}_{X}(V)) > R_0.
\end{equation*}

Let $H := \oplus_{i=1}^{L} \oplus_{j=1}^{c} \R \partial_{B_{\alpha_{i,j}}}$.
Define $\Xi_H \colon H \to \prod_{i=1}^{L}E^c_{f}(z_i)$ by
\begin{equation*}
\Xi_H(B) = \big(\pi_cD\big(\prod_{j=1}^{c}H_{T, \hat{\gamma}_{i,j}}\big)(B +\nu_0(x,B))\big)_{i=1,\dots,L}.
\end{equation*}
By Lemma \ref{prop holon map and nu}\eqref{lemma first coordinate}, we have for any $1 \leq i , k \leq L$:
\begin{equation*}
D_{B_{\alpha_{k,1}}, \dots, B_{\alpha_{k,c}}}\big(\pi_c D\big(\prod_{j=1}^{c}H_{T, \hat{\gamma}_{i,j}}\big)(B + \nu_0(x,B))\big) = \sum_{l=1}^{c} I_{i,k,l}(B),
\end{equation*}
where for each $1 \leq l \leq c$, we set
\begin{equation*}
I_{i,k,l}(B):=  D_{B_{\alpha_{k,1}}, \dots, B_{\alpha_{k,c}}}\big(\pi_cD\big(\prod_{j=l+1}^{c}H_{T, \hat{ \gamma}_{i,j}}\big)(\pi_c DH_{T, \hat{\gamma}_{i,l}}(B + \nu_0(y_{i, l}, B)))\big).
\end{equation*}
It is clear that for all $1 \leq i \leq L$, $1 \leq l \leq c$, 
\begin{equation*}
\pi_c DH_{T, \hat{\gamma}_{i,l}}(B + \nu_0(y_{i, l}, B)) = \pi_c DH_{T, \hat{\gamma}'_{i,l}}(B + \nu_0(y_{i, l}, B)),
\end{equation*}
where $\hat{\gamma}'_{i,l}$ denotes the lift of $\gamma'_{i,l}$ for $T$.
Since $\gamma'_{i,l}$ and $V$ satisfy the conditions of Lemma \ref{prop linear approx} with $(\sigma,C)$ replaced by $(12C_fK_f\sigma,12 C_f K_f C)$,  there exists a $C^2$-uniform constant $c_{1}=c_1(f) > 0$ (we absorb  the term $12C_fK_f$ in $c_1$) so that 
\begin{gather*}
 \big|I_{i,k,l}(B)  -  D_{B_{\alpha_{k,1}}, \dots,B_{\alpha_{k,c}} } \big(D\big(\prod_{j=l+1}^{c}H_{f, \gamma_{i,j}}\cdot H^{s}_{f, x_{i, l,3}, x}H^{u}_{f, x_{i, l,2}, x_{i,l,3}}H^{s}_{f, x_{i, l,1}, x_{i,l,2}} \big) \\ 
 \cdot \big(\pi_cV(B, y_{i,l,1})\big)\big) \big| \leq c_{1} C^{3}e^{-R_0 \xi'} \|B\|.
\end{gather*} 
Let $1 \leq i,k \leq L$ and $1 \leq j \leq c$. If $  i \neq k $ or $ j \neq j_0 $, then by \eqref{term 3000},  $z= y_{i,j,1}$ satisfies \eqref{determinant 1}, hence $
|I_{i,k,j}| < c_{2} C^3 e^{-R_0 \xi'}$. Again  by \eqref{term 3000}, $z= y_{i,j_0,1}$ satisfies \eqref{determinant 2}, hence
\begin{equation*}
|\det (I_{i,i,j_0})| > c_{3}\kappa - c_{4} C^3 e^{-R_0 \xi'}.
\end{equation*}
Here $c_{2}, c_{3}, c_{4}>0 $ are $C^2$-uniform constants depending only on $f, c$.

Thus for some  $C^2$-uniform constant $c_{5}>0$ depending only on $f, L,c, C$, for any  sufficiently large $R_0$ depending only on $f, L,c, C,\kappa$, we have $\det(\Xi_H) > c_{5} \kappa^{L}$.
 Moreover, it is easy to see that $R_0$ is $C^2$-uniform with respect to $f$.
Then $
\kappa_0:= c_{5} \kappa^{L}$ depends only on $ f, L,c, C,\kappa$ and is $C^2$-uniform in $f$.  This concludes the proof.
\end{proof}

\section{Finding suitable spanning $c$-families}\label{section spanning}

Since we will study a parametrised family of diffeomorphisms, we need a bit more work to find suitable $c$-families.
Let us start by recalling a result presented in \cite[Lemma 1.2]{DW}. We use here the notations introduced in Subsection \ref{subsection c disks}.

\begin{lemma}[Accessibility modulo central disks]\label{lem acc mod den dis}
Let $f \in \cPH^{1}(X)$ be dynamically coherent. Assume that the fixed points of $f^{k}$ are isolated, for all $k\geq 1$. Then for every integer $\overline{R} > 0$ there exists $\cD = \cD(f, \overline{R})$, a $c$-family for $f$ such that 
\begin{equation*}
(1)\ \overline{r}(\cD) < \overline{R}^{-1},\qquad 
(2)\ R_\pm(f, \cD) > \overline{R},\qquad
(3)\ \cD \text{ is } (\frac{1}{80},  2)\text{-spanning}.
\end{equation*}
\end{lemma}

For the convenience of the choices of some constants, we replaced the constant $\frac{1}{2}$ in \cite[Lemma 1.2]{DW} by $\frac{1}{80}$. This does not introduce any new difficulty into the proof. 
The following is a consequence of the above lemma.
\begin{cor}\label{cor acc mod den dis} 
Assume that $f \in \cPH^1(X)$ is dynamically coherent, and the fixed points of $f^{k}$ are isolated for all $k\geq 1$.  Then for every $\overline{R} > 0$, there exist $C^1$-uniform constants $N = N(f, \overline{R})>0$, $\rho = \rho(f, \overline{R})\in(0,\overline{R}^{-1})$ and $\sigma = \sigma(f, \overline{R})>0$ such that the following is true. For all $g$ sufficiently $C^1$-close to $f$, there exists $\cD_g$, a $(\frac{1}{40},  4)$-spanning $c$-family for $g$ such that   Lemma \ref{lem acc mod den dis}$(1)$ is satisfied  for $(\cD_g,g)$ in place of $(\cD, f)$. Moreover, we have
\begin{equation*}
(1)\ \underline{r}(\cD_g) > \rho,\quad
(2)\ n(\cD_g) < N,\quad
(3)\ \cD_g \text{ is } \sigma\text{-sparse},\quad
(4)\ R_\pm(g, (\cD_g, \sigma)) > \overline{R}.
\end{equation*}

\end{cor}
\begin{proof}
Let $\mathcal{D}=\cD(f, 2\overline{R})$ be a $(\frac{1}{80},  2)$-spanning $c$-family for $f$  given by Lemma \ref{lem acc mod den dis}. Set $N:=n(\cD)+1$. Take $\rho \in (0, (2\overline{R})^{-1})$ such that $\underline{r}(\cD)>\rho$, and choose $ \sigma>0$ so that (3), (4) above are true for $(\cD, f, 4\sigma)$ in place of $(\cD_g, g, \sigma)$. Then for any $g$ sufficiently $C^1$-close to $f$, for any $c$-family  $\mathcal{D}_g$ for $g$ with $(\mathcal{D}_g,0)\subset (\mathcal{D},\sigma)$, items (3),(4) above are satisfied for $(\mathcal{D}_g,g)$. 

By Lemma \ref{lemma stability spanning} applied to $f$ and $(\theta,\theta',\theta'',\rho_m,\rho_M)=(\frac{1}{80},\frac{1}{40},\frac{1}{40},\rho,\overline{R}^{-1})$, and by Remark \ref{remarque trois}, we see that there exists a $\big(\frac{1}{40},4\big)$-spanning $c$-family $\mathcal{D}_g$ for $g$, satisfying Lemma \ref{lem acc mod den dis}(1), Corollary   \ref{cor acc mod den dis}(1),(2), and $(\mathcal{D}_g,0)\subset (\mathcal{D},\sigma)$. Combined with the previous discussions, this concludes the proof. 
\end{proof}

While working with a family of diffeomorphisms, we will need to consider several $c$-families. 
For that purpose, we use  a superposition of a collection of perturbations which are localized   in parameter-phase space. The following proposition will allow us to arrange their support in such a way that the interferences between them are very weak; it will serve as a key step in the inductive construction of these localized perturbations in Proposition \ref{lower bound determinant 2}. 

\begin{prop} \label{prop slow recurrent}
Let $r\in \mathbb{N}_{\geq 2} \cup \{\infty\}$, $J\geq 1$, and let $\{f_{a}\}_{a \in [0,1]^{J}}$ be a good (see Definition \ref{defregularfamilyofdiff}) $C^{r}-J-$family   in the space of dynamically coherent,  $C^r$ partially hyperbolic diffeomorphisms. Then for any integers $K,R_0 \geq 1$, any real numbers $\vartheta > 0$, $h_0 > 0$, there exists  a  set $\Omega_1$ compactly contained in $[0,1]^{J}$ with $\mathrm{Leb}([0,1]^{J} \backslash \Omega_1) < \vartheta$, an integer $N_0 > 1$, and real numbers $\rho_0 \in (0, h_0)$, $\rho_1 \in (0, \rho_0)$, $\sigma_0,\lambda_0 > 0$ such that the following is true.

 Take any $a \in \Omega_1$, and any integer $0 \leq l \leq K-1$. For any collection of points $\{a_i\}_{i=1}^l \subset B(a,\lambda_0) \cap [0,1]^{J}$, any $1 \leq i \leq l$, let $\cD_i$ be a $c$-family for $f_{a_i}$ satisfying 
\begin{equation*}
(1)\ [ \underline{r}(\cD_i) , \overline{r}(\cD_i) ] \subset (\rho_1, \rho_0);\qquad
(2)\ n(\cD_i) < N_0.
\end{equation*}
Then there exists a $(\frac{1}{20}, 6)$-spanning $c$-family for $f_a$, denoted by $\cD_{l+1}$, such that $(1)$, $(2)$ above are satisfied for $i = l+1$, and moreover,
\enmt
\item $\cD_{l+1}$ is $\sigma_0$-sparse, and $(\cD_{l+1}, \sigma_0)$ is disjoint from $(\{\cD_{i}\}_{i=1}^{l},\sigma_0)$;
\item for any $a' \in B(a, \lambda_0) \cap [0,1]^{J}$, we have
\begin{align*}
&\bullet\ R(f_{a'}, (\cD_{l+1}, \sigma_0), (\{\cD_{i}\}_{i=1}^l, \sigma_0))>R_0;\\
&\bullet\ R_{\star}(f_{a'}, (\cD_{l+1}, \sigma_0))   > R_0,\qquad \star=+,-.
\end{align*}
\eenmt
\end{prop}
\begin{proof}
We choose $\Omega_1$ to be any compact set contained in $\mathrm{int}([0,1]^{J})$ such that $\mathrm{Leb}([0,1]^{J} \backslash \Omega_1) < \vartheta$, and for any $a \in \Omega_1$, the fixed points of $f_a^{k}$ are isolated for any integer $k \geq 1$. The existence of $\Omega_1$ is guaranteed by our hypothesis that $\{f_a\}_{a\in [0,1]^{J}}$ is a good family.

By the compactness of $[0,1]^{J}$, there exists $\rho_2 \in (0, h_0)$ such that for any $a \in [0,1]^{J}$, $x \in X$, the tangent space of $\cWc_{f_a}(x,4\rho_2)$ is sufficiently close  to $E^c_{f_a}(x)$ so that for any $y \in B(x, \rho_2)$, $\cWc_{f_a}(y, 4\rho_2)$ intersects $B(x, \rho_2)$ in a single local center manifold.

Let $\rho_0\in (0,  \frac{\rho_2}{2})$ be small enough so that for any $a \in [0,1]^{J}$, any $x \in X$, we have
\begin{align}
f_a^{p}(\cWc_{f_a}(x, \rho_0)) &\subset \cWc_{f_a}(f_a^{p}(x),\rho_2),\quad \forall\, -R_0 \leq p \leq R_0.\label{condd 2}
\end{align}
By Corollary \ref{cor acc mod den dis} applied to $\overline{R} > \max(\rho_0^{-1}, R_0)$, and by the compactness of $\Omega_1$, there exist $N_0>0$, $\rho_1\in (0,\rho_0)$, $\sigma_1\in (0,\rho_2)$ such that for all $a \in \Omega_1$ there exists a $(\frac{1}{40},  4)$-spanning $c$-family for $f_a$, denoted by $\widetilde{\cD}(a)$, such that $[\underline{r}(\widetilde{\cD}(a)) , \overline{r}(\widetilde{\cD}(a))] \subset (\rho_1, \rho_0)$, $n(\widetilde{\cD}(a)) < N_0$, and 
\begin{equation}\label{eq sparse}
\widetilde{\cD}(a)\text{ is } \sigma_1\text{-sparse},\quad \text{and}\quad R(f_a, (\widetilde{\cD}(a), \sigma_1)) > R_0.
\end{equation} 

Take $\sigma_3  > 0$ sufficiently small such that for any $a \in [0,1]^{J}$, $x \in X$, $y \in B(x,\sigma_3)$, and any $\rho \in (\rho_1, \rho_0)$, we have $\cWc_{f_a}(y,\rho) \subset (\cWc_{f_a}(x,\rho), \sigma_1/2)$ and\footnote{As in the case of \eqref{eq lem c family}, here \eqref{proj acc class} follows from the uniform transversality of $\cW_{f_a}^{cs}$ and $\cW_{f_a}^u$, respectively $\cW_{f_a}^{cu}$ and $\cW_{f_a}^s$.}
\begin{equation}\label{proj acc class}
 \cWc_{f_a}\big(x, \frac{1}{40}\rho\big) \subset \bigcup_{z \in \cWc_{f_a}(y, \frac{1}{20}\rho)} Acc_{f_a}(z, 1, 2).
\end{equation}

Take $\sigma_2 > 0 $  such that for any $a \in [0,1]^{J}$, any $x\in X$, any collection of $3R_0N_0K$ points  $\{x_{i}\}_{i=1}^{3R_0 N_0 K}\subset X$, there exists $y \in B(x, \sigma_3)$ such that  for all $1 \leq i \leq 3R_0N_0K$, we have $d(\cWc_{f_a}(y,\rho_2), \cWc_{f_{a}}(x_{i}, \rho_2)) > 3\sigma_2$.

Take a small constant $\lambda_0 > 0$ such that for any $a \in [0,1]^{J}$, $a' \in B(a,\lambda_0) \cap [0,1]^{J}$, $x \in X$, and any $-R_0 \leq p \leq R_0$, we have 
\begin{equation}\label{cond fpa disk} 
f^{p}_a(\cWc_{f_{a'}}(x, \rho_0)) \subset (f^{p}_a(\cWc_{f_{a}}(x, \rho_0)),  \sigma_2).
\end{equation}

Fix any $a \in \Omega_1$. We denote $\widetilde{\cD}:= \widetilde{\cD}(a)=\{\widetilde{\mathcal{C}}_1,\dots,\widetilde{\mathcal{C}}_{N_1}\}$ for some $N_1 < N_0$. We take $l, \{a_i\}_{i=1}^{l}$ and $\{\cD_i\}_{i=1}^{l}$ as in the proposition. 
We will modify $\widetilde{\cD}$ to obtain $\cD_{l+1}$ that satisfies the conclusion of the proposition. 

We will define $\cC_{l+1,1},\dots, \cC_{l+1, N_1}$ by induction so that $\cD_{l+1} = \{\cC_{l+1,1},\dots, \cC_{l+1, N_1}\}$ is a $(\frac{1}{20},  6)$-spanning $c$-family for $f_a$. Let $0 \leq j \leq N_1-1$ be an integer such that for all $1 \leq k \leq j$, $\cC_{l+1,k}$ is  defined and satisfies $\varrho(\cC_{l+1,k}) \in (\rho_1, \rho_0)$, 
\ary\label{term 2000}
&& \mathcal{C}_{l+1,k}\subset(\widetilde{C}_k,\sigma_1/2),\\
&& \frac{1}{40}\widetilde{\cC}_{k} \subset \bigcup_{x \in \frac{1}{20}\cC_{l+1, k}}  Acc_{f_a}(x, 1, 2), \label{term 200bis} \\
\mbox{ and } &&(\cC_{l+1,k}, \sigma_2) \cap (\cC', \sigma_2) = \emptyset, \quad \forall\, 
\cC'\in \cal M_k,\label{construction inductive dl1} 
\eary
where
\begin{equation*}
\cal M_k:= \bigcup_{\substack{1\leq i \leq l,\ \cC \in \cD_i, \\ -R_0 \leq p \leq R_0}} f_a^{p}(\cC)\  \cup\ \bigcup_{\substack{1 \leq m \leq k-1,\\ -R_0 \leq p \leq R_0}}f_a^{p}(\cC_{l+1, m}).
\end{equation*}

The above is true for $j=0$. 
By the choices of $\rho_0$ and $\lambda_0$, by \eqref{condd 2} and \eqref{cond fpa disk}, for any $1 \leq i \leq {l}$, $\cC \in \cD_i$, and $-R_0 \leq p \leq R_0$,
there exists $x \in X$ such that $f^{p}_a(\cC) \subset (f^{p}_a(\cWc_{f_a}(x,\rho_0)), \sigma_2) \subset (\cWc_{f_a}(f^{p}_{a}(x), \rho_2),\sigma_2)$. Similarly, for each $1\leq m \leq j, -R_0 \leq p \leq R_0$,  there exists $x \in X$ such that $f^{p}_a(\cC_{l+1,m}) \subset \cWc_{f_a}(x, \rho_2)$.
Thus $\mathcal{M}_{j+1}$ is contained in less than $3R_0 N_0 K$ many $\sigma_2$-neighbourhoods of $c$-disks for $f_a$ of radius $\rho_2$. 
Then by 
\eqref{condd 2}, \eqref{proj acc class},  and the choice of $\sigma_2$, there exists a center disk $\cC_{l+1, j+1}= \cWc_{f_a}(y, \rho')$ for some $\rho'\in (\rho_1,\rho_0)$, satisfying  \eqref{term 2000}, \eqref{term 200bis} and \eqref{construction inductive dl1} for $k=j+1$. 
We complete the construction of $\cD_{l+1}$ by induction.

Since $\widetilde{\cD}$ is $(\frac{1}{40},  4)$-spanning, by \eqref{term 200bis}, $\cD_{l+1}$  is $(\frac{1}{20},6)$-spanning.
By taking $\sigma_0 > 0$ sufficiently small, depending only on $\{f_a\}, R_0, \sigma_2$, we can ensure that for any $-R_0 \leq p \leq R_0$, any $\cC \in \bigcup_{1 \leq i \leq l+1}\cD_{i}$, we have $f_a^{p}((\cC, 2\sigma_0)) \subset (f_a^{p}(\cC), \sigma_2/4)$.

By further requiring that $\sigma_0$ be sufficiently small, depending only on $\{f_a\},R_0,\sigma_2,\sigma_1$, \eqref{eq sparse}, \eqref{term 2000}, \eqref{construction inductive dl1}  implies that $\cD_{l+1}$ is $2\sigma_0$-sparse, and 
\begin{equation*}
\bullet\ R(f_a, (\cD_{l+1}, 2\sigma_0), (\{\cD_i\}_{i=1}^l, 2\sigma_0)) >R_0,\quad\quad \bullet\ R_{\pm}(f_a, (\cD_{l+1}, 2\sigma_0)) > R_0.
\end{equation*}
By continuity, and after possibly taking  $\lambda_0$ to be even smaller, but depending only on $\{f_a\}$, $\rho_0$, $\rho_1$, $R_0$, $N_0$, $K$, $\sigma_0$, we can ensure that for all $a' \in B(a,\lambda_0) \cap [0,1]^{J}$, 
\begin{equation*}
\bullet\ R(f_{a'}, (\cD_{l+1}, \sigma_0), (\{\cD_{i}\}_{i=1}^l, \sigma_0)) >R_0,\quad\quad \bullet\ R_{\pm}(f_{a'}, (\cD_{l+1}, \sigma_0))   > R_0.
\end{equation*}
\end{proof}

\section{A stable criterion for stable values}\label{section Bonk-Kleiner}

\subsection{A criterion for stable values}

In this section we state a topological lemma that is at the core of our construction of open accessibility classes. First we borrow a few definitions from \cite{BK}.

\begin{defi}
If $f \colon X \to Y$ is a continuous map between metric spaces $X$ and $Y$, then $y \in Y$ is a \textit{stable value} of $f$ if there is $\epsilon > 0$ such that $y \in \mathrm{Im}(g)$ for every continuous map $g \colon X \to Y$ such that $d_{C^{0}}(f,g) < \epsilon$.
\end{defi}



\begin{defi} \label{light map 2}
Given a constant $\epsilon > 0$,
a continuous map $f \colon X \to Y$ between metric spaces $X$ and $Y$ is called \textit{$\epsilon$-light} if for every $y \in Y$, every connected component of $f^{-1}(y)$ has diameter strictly smaller than $\epsilon$.
\end{defi}

\begin{rema}
This definition is a quantitative  version of the notion of \textit{light map}  in \cite{BK} (a map $f \colon X \to Y$  is called \textit{light} if all point inverses are totally disconnected). 
\end{rema}

Now we state the main topological result  in this section.

\begin{thm} \label{ epsilon lm sv}
For any integer $n \geq 1$, there exists a constant $\epsilon=\epsilon(n) > 0$ such that any $\epsilon$-light continuous map $f\colon [0,1]^n \to \R^{n}$ has a stable value.
\end{thm}
\begin{proof}
In Appendix \ref{app c}.
\end{proof}

\begin{rema}
Theorem \ref{ epsilon lm sv} is a quantitative version of  a result due to Bonk-Kleiner: in \cite[Proposition 3.2]{BK}, the authors proved that any light continuous map from a compact metric space of topological dimension at least $n$ to $\R^{n}$ has stable values.  The proof we give is modeled on theirs. 
\end{rema}

\begin{cor} \label{net to sv}
For any integer $c \geq 1$, let $\{\mathscr{U}_{\alpha}\}_{\alpha \in \cA}$ 
be an open cover of $[0,1]^{c}$ s.t. $\mathrm{diam}(\mathscr{U}_{\alpha}) < \epsilon(c)$ for all $\alpha \in \cA$, where $\epsilon(c)$ is given by Theorem \ref{ epsilon lm sv}. Let  $f \colon [0,1]^{c}\to \R^{c}$ be a continuous map s.t. for any $x \in [0,1]^{c}$, there exists $\cI \subset \cA$ satisfying
\begin{enumerate}
	\item $\cap_{\alpha \in \cI} f(\partial \mathscr{U}_{\alpha}) = \emptyset$;
	\item $x \in \mathscr{U}_{\alpha}\text{ for all }\alpha \in \cI$.
\end{enumerate}
Then $f$ has a stable value.
\end{cor}
\begin{proof}
By Theorem \ref{ epsilon lm sv}, it suffices to check that $f$ is $\epsilon(c)$-light. Given any $x \in [0,1]^{c}$, take $\cI \subset \cA$ satisfying (1), (2). In particular, there exists $\alpha \in \cI$ such that $f(x) \notin f(\partial \mathscr{U}_{\alpha})$. We denote by $P_x$ the connected component of $f^{-1}(f(x))$ containing $x$. We claim that $P_x$ is contained in $\mathscr{U}_{\alpha}$. Indeed, by the continuity of $f$, $f^{-1}(f(x))$ has no accumulating point in $\partial \mathscr{U}_{\alpha}$. If  $P_x \cap (\mathscr{U}_{\alpha})^{c} \neq \emptyset$, then we can find two disjoint open sets $U,V$ s.t. $P_x \subset U \cup V$ and $P_x \cap U$, $P_x \cap V$ are both nonempty. This contradicts the connectedness of $P_x$, hence the claim is true. In particular, the diameter of $P_x$ is not larger than the diameter of $\mathscr{U}_{\alpha}$ which by hypothesis is strictly smaller than $\epsilon(c)$. Since $x$ is an arbitrary point in $[0,1]^{c}$, we deduce that $f$ is $\epsilon(c)$-light.
\end{proof}

\begin{figure}
	\begin{center}
		\includegraphics [width=12cm]{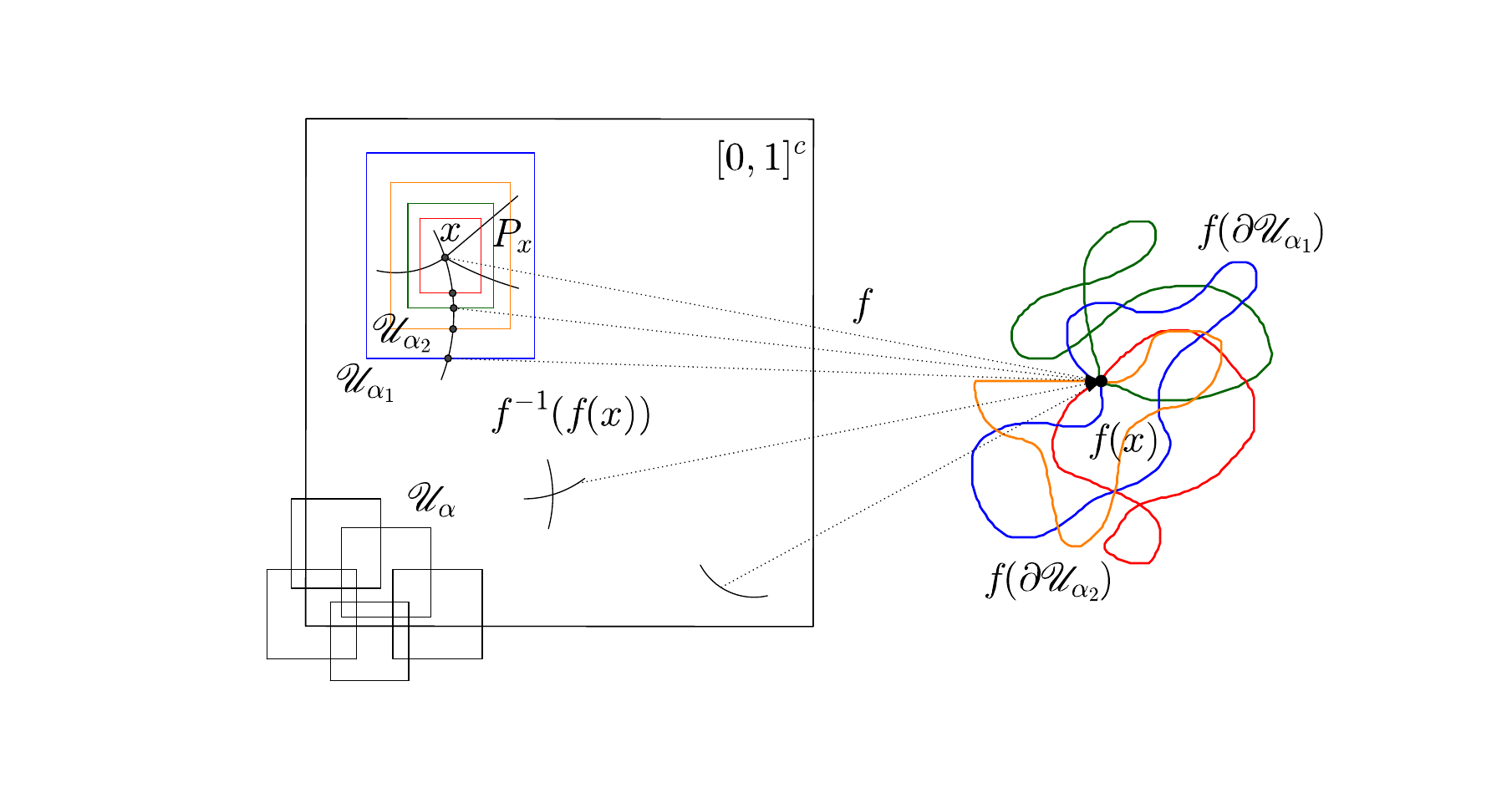}
	\end{center}
\caption{Controlling the size of connected components of point inverses.}
\end{figure}

\subsection{Choosing a cover by disjoint squares}\label{subs disjoint squares}

In this section, we define a cover of $[0,1]^c$ by open cubes, which will later be used when we apply Corollary \ref{net to sv} to show the existence of open accessibility classes.

Given an integer  $c \geq 1$, a positive constant $\theta \in (\frac{c-1}{c}, 1)$, we set
\begin{equation}\label{def K0 K1}
K_0(c, \theta):=\Big\lc \frac{3c + 1}{c-(c-1)\theta^{-1}} \Big\rc +1,\qquad K_1(c, \theta):= c K_0(c, \theta).
\end{equation}
In the following, we will fix $c, \theta$ and abbreviate $K_i(c,\theta)$ as $K_i$, $i=0,1$.

By a direct construction, we can fix a cover $\{\mathscr{U}_{\alpha}\}_{\alpha \in \cA}$ of $[0,1]^{c}$ by open sets in $\R^{c}$, which satisfies: 
\begin{enumerate}
\item $\cA$ is a finite set and for all $\alpha \in \cA$, there exist constants $\{ p_{\alpha, i}, q_{\alpha, i}\}_{i=1,\dots,c} \subset [-1,2]$ such that $\mathscr{U}_{\alpha} = (p_{\alpha,1}, q_{\alpha,1}) \times \dots \times (p_{\alpha,c}, q_{\alpha,c})$;

\item for any $\alpha \in \cA$, $\mathrm{diam}(\mathscr{U}_{\alpha}) < \epsilon(c)$, where $\epsilon(c)$ is given by Theorem \ref{ epsilon lm sv};

\item for each $x \in [0,1]^{c}$, there exists a subset $\cI \subset \cA$ with more than $K_1$ elements satisfying that $x \in \mathscr{U}_{\alpha}$ for all $\alpha \in \cI$, and $\{ \partial \mathscr{U}_{\alpha} \}_{\alpha \in \cI}$ are mutually disjoint;
\item for each $i \in \{1,\dots, c\}$, the points $\{p_{\alpha,i }, q_{\alpha, i}\}_{\alpha \in \cA}$ are mutually distinct. 
\end{enumerate}

For each integer $i \in \{1,\dots, c\}$, we let $\cB_i:= \{ p_{\alpha,i}, q_{\alpha,i} \}_{\alpha \in \cA}$, and for each $\alpha \in \cA$, we denote 
$\partial^i \mathscr{U}_\alpha:=[p_{\alpha,1}, q_{\alpha,1}] \times \dots \times [p_{\alpha,i-1}, q_{\alpha,i-1}]\times \{p_{\alpha,i},q_{\alpha,i}\} \times [p_{\alpha,i-1}, q_{\alpha,i-1}]\times\dots\times [p_{\alpha,c}, q_{\alpha,c}]$. Given any $s \in [-1,2]$, we introduce the normalized coordinate
\begin{equation}\label{varphi i s}
\varphi(i,s):= \frac{6i-2+s}{6c}\in \Big[\frac{i}{c}-\frac{1}{2c},\frac{i}{c}\Big] \subset(0,1].
\end{equation}
Note that for any $i<i'$ and any $s,s'\in [-1,2]$, $\varphi(i,s)<\varphi(i',s')$. We also set 
\begin{equation}\label{definition cmin}
0<C_{\mathrm{min}}:=100 \Big(\min_{\substack{1\leq i \leq c\\ t\neq t' \in \cB_i }} |\varphi(i,t)-\varphi(i,t')|\Big)^{-1}< +\infty.
\end{equation}  

\begin{defi}
We define the set 
\begin{gather}\label{defi Gamma}
\Gamma:= \big\{(i, \cB, \{s_t = (s_{t,1}, \dots, s_{t,c}) \}_{t \in \cB}) \ \vert\ 
  i \in \{1,\dots, c\},\  \cB \subset \cB_i, \ |\cB| = K_0, \\
 \mbox{ and } s_t \in [-1,2]^{i-1} \times \{t\} \times [-1,2]^{c-i},\ \forall\, t \in \cB)\big\}.\nonumber
\end{gather}
\end{defi}

\section{Holonomy maps associated to a family of loops}\label{holo maps fami}

In this section, \hyperref[(H1)]{(H1)} holds. 

\subsection{Continuous and regular family of loops}

\begin{defi}
	Given  $x \in X$, a one-parameter family $\{\gamma(s) = (x_1(s), x_2(s), x_3(s))\}_{s\in [0,1]}$ of $f$-loops at $x$ is said to be \textit{continuous} if for any $i=1,2,3$, the map $s \mapsto x_i(s)$ is continuous. We define $\ell(\gamma) := \sup_{s \in [0,1]} \ell(\gamma(s))$. 
\end{defi}
\begin{lemma}[Continuation of $f$-loops]\label{extension continuous family floops}
	There exist $\cU$, a $C^1$-open neighbourhood of $f$, as well as $ \varsigma_{f} > 0$ such that  the following is true. Let $\gamma$ be a continuous family of $f$-loops at $x\in X$ satisfying $\ell(\gamma)<\frac{\varsigma_f}{2}$. Then
	for any $g \in \cU$ and $y \in B(x, \varsigma_{f})$, we can define $\gamma_{g,y}$, a continuous family of $g$-loops at $y$, such that  $\gamma_{f,x}= \gamma$ and each coordinate of $\gamma_{g,y}(s)$ depends continuously on $(g,y,s)$.
\end{lemma}
\begin{proof}
	Let $\gamma=(x_1,x_2,x_3)$. If $(g,y)$ is chosen sufficiently close to $(f,x)$, then for any $s \in [0,1]$, the following leaves intersect at a unique point, and we define  
	\begin{itemize}
		\item $\{y_{g,1}(s)\}:=\cWu_{g}(y, h_f)\cap\cWcs_{g}(x_1(s), h_f)$; 
		\item $\{y_{g,2}(s)\}:=\cWs_{g}(y_{g,1}(s), h_f)\cap\cWcu_{g}(x_2(s), h_f)$; 
		\item $\{y_{g,3}(s)\}:=\cWu_{g}(y_{g,2}(s), h_f)\cap\cWcs_{g}(y, h_f)$.
	\end{itemize}
	
	Then we can verify our lemma for  $\gamma_{g,y}:=(y_{g,1},y_{g,2},y_{g,3})$.   
\end{proof}

\subsection{A criterion for stable accessibility}

Given $x \in X$, and $\gamma=\{\gamma(s)\}_{s \in [0,1]}$, a continuous family of $f$-loops at $x$, satisfying $\ell(\gamma) < \overline{\sigma}_f$ (defined in Notation \ref{notation 2}), we introduce
\begin{equation} \label{psifxgammadefi}
\psi = \psi(f,x,\gamma)\colon\left\{
\begin{array}{rcl}
[-1,2]^{c} &\to& \cWc_f(x),  \\
s = (s_1,\dots, s_c)&\mapsto& (\prod_{j=1}^{c}H_{f, \gamma(\varphi(j,s_j))})(x).
\end{array}
\right.
\end{equation}
By Notation \ref{notation 2}, $\psi$ is well-defined. Besides, it is clear that $\mathrm{Im}(\psi)\subset Acc_f(x)$.

 The following property combines a global property, based on the notion of \lq\lq accessibility modulo central disks\rq\rq\, which  appears in \cite{DW} (see also Section \ref{section spanning}), and a local one, based on the notion of $\epsilon$-light maps in Section \ref{section Bonk-Kleiner}, which together imply the accessibility property. 

\begin{defi}[Property $(\mathcal{P})$]\label{property cP}
	We say 
	that $f$ satisfies property $(\mathcal{P})$ if there exist $0 < \theta  <  \theta' < 1$, an integer $k\geq 1$, and $\cD$, a $(\theta,k)$-spanning $c$-family for $f$, such that for any $\cC\in \cD$, any $x \in \theta'\cC$, there exists a continuous family of $f$-loops at $x$, denoted by $\{\gamma_x(s)=(x_1(s),x_2(s),x_3(s))\}_{s}$, with $\ell(\gamma_x)<\overline{\sigma}_f$, such that the following is true: let  $\psi_x:=\psi(f,x,\gamma_x)$ be given by \eqref{psifxgammadefi}. Then for any $(i, \cB, \{s_t\}_{t \in \cB}) \in \Gamma$ (defined in \eqref{defi Gamma}), there exist $t,t' \in \cB$ such that $\psi_x(s_t) \neq \psi_x(s_{t'})$. 
\end{defi}
\begin{prop} \label{lemmappointlooptostablevalue}
	If $f$ satisfies 
	property $(\mathcal{P})$, then $f$ is $C^1$-stably accessible. 
\end{prop}
\begin{proof}  Assume that $f$ satisfies $(\mathcal{P})$ for $0 < \theta  <  \theta' < 1$,  $k\geq 1$, some $(\theta,k)$-spanning $c$-family for $f$, denoted by $\cD$, and a set of families of $f$-loops $\{\gamma_x\}_{x \in  \theta'\cC,\ \cC \in \cD}$. Take $\cC \in \cD$, $x \in \theta\cC$, and set $\psi_x:=\psi_x(f,x,\gamma_x)$. We claim that $Acc_f(x)$ is open. 
	
	To see this, take  an arbitrary $s \in [0,1]^c$. By property (3) of  the open cover $\{\mathscr{U}_\alpha\}_{\alpha \in \mathcal{A}}$   in Subsection \ref{subs disjoint squares}, there exists a subset $\cI \subset \cA$ with $|\cI| \geq K_1$, such that $s \in \mathscr{U}_{\alpha}$ for all $\alpha \in \cI$, and $\{ \partial \mathscr{U}_{\alpha} \}_{\alpha \in \cI}$ are mutually disjoint. Let us show that 
	$\cap_{\alpha \in \cI} \psi_x(\partial \mathscr{U}_{\alpha}) =\emptyset$. Assume it is not true. By \eqref{def K0 K1} and the pigeonhole principle, we may choose $i \in \{1,\dots,c\}$ and $\cI' \subset \cI$ with $|\cI'| = K_0$, such that $\cap_{\alpha \in \cI'} \psi_x(\partial^i \mathscr{U}_{\alpha}) \neq \emptyset$. Thus there exists $(i,\mathcal{B},\{s_t\}_t) \in \Gamma$ such that
	$$
	\psi_x(s_t)=\psi_x(s_{t'}),\quad \forall\, t,t' \in \mathcal{B},
	$$
	which contradicts $(\mathcal{P})$. Therefore, $\cap_{\alpha \in \cI} \psi_x(\partial \mathscr{U}_{\alpha}) =\emptyset$. Since $s$ can be taken arbitrary in $[0,1]^c$, Corollary \ref{net to sv} implies that $\psi_x$ has a stable value $y$, and thus, $\mathrm{Im}(\psi_x)$ contains an open neighbourhood of $\{y\}$ in $\cWc_f(x)$. But $\psi_x$ takes values in $\cWc_f(x)\cap Acc_f(x)$, hence the latter has non-empty interior. Saturating by local stable and unstable leaves, we deduce that $Acc_f(x)$ has non-empty interior. Then  the accessibility class 
	$Acc_f(x)$ is open, and the claim is proved. 

	Since $\cD$ is $(\theta,k)$-spanning, the previous claim implies that for any $x\in X$, $Acc_f(x)$ is open. This shows that $f$ is accessible, since $X$ is connected. 
	
	Now it suffices to show that  $(\cal P)$ is a $C^1$-open condition. Let $\varsigma_{f}>0$ be as in Lemma \ref{extension continuous family floops} and let $\sigma \in (0, \varsigma_{f})$ be a  small constant to be determined. Let $\theta''':=\frac{\theta+\theta'}{2}$.  By Lemma \ref{lemma stability spanning} and Remark \ref{remarque trois}, for any  $g$ sufficiently $C^1$-close to $f$, there exists $\cD_g$, a $(\theta''', k+2)$-spanning $c$-family for $g$, such that for each $\cC_g \in \cD_g$, there exists $\mathcal{C}\in \mathcal{D}$ so that $\theta' \cC_g \in (\theta'\cC, \sigma)$. For each $y \in \theta' \cC_g$, take $x \in \theta' \cC$ with $y \in B(x, \sigma) \subset B(x, \varsigma_{f})$. Applying Lemma \ref{extension continuous family floops} to $\gamma_x$, we obtain $\gamma_{g,y}$, a continuous family of $g$-loops at $y$ which is close to $\gamma_x$. By choosing $\sigma$  sufficiently small, we can ensure that for any $g$ sufficiently close to $f$ in $C^1$ topology, any $\cC_g \in \cD_g$, $y \in \theta' \cC_g$, any $(i, \cB, \{s_t\}_{t \in \cB}) \in \Gamma$,   there exists $t, t' \in \cB$ such that $\tilde\psi_y(s_t) \neq \tilde\psi_y(s_{t'})$, where $\tilde\psi_y := \psi(g, y, \gamma_{g,y})$. Thus $(\cal P)$ is a $C^1$-open condition. 
\end{proof}

\subsection{Parametrising an accessibility set using a family of loops}

To  optimize the pinching exponents in our theorems, we will mainly consider the class of continuous families of loops as follows.

\begin{defi}[Regular family] \label{defregularfamily}
	Given $x \in X$ and constants $\sigma \in \big(0,\frac{\overline{\sigma}_f}{2C_f}\big)$, $C>0$,  a continuous family $\{\gamma(s) = (x_1(s), x_2(s), x_3(s))\}_{s \in [0,1]}$ of $f$-loops at $x$ is said to be $(\sigma,C)$-\textit{regular} if it satisfies:
	\enmt
	\item $\ell(\gamma) < \sigma$ and the map $s \mapsto x_1(s)$ is   $C$-Lipschitz  (with respect to $d_{\cWu_f}$); 
	\item there exists $x' \in \cWcs_{f}(x, \frac{\sigma_f}{2C_f})$ such that $x_2(s) \in \cWcu_{f}(x', \sigma)$ for all $s \in [0,1]$.
	\eenmt
	In this case, we say that $\gamma$ is \textit{determined by} $x'$ and $(x_1(s))_{s\in[0,1]}$.
	Indeed, for any $s \in [0,1]$, $x_2(s)$ is the unique intersection of $\cWs_f(x_1(s), h_f)$ and $\cWcu_f(x', h_f)$, and $x_3(s)$ is the unique intersection of $\cWu_f(x_2(s), h_f)$ and $\cWcs_f(x, h_f)$.\footnote{Indeed, by $\ell(\gamma)<\sigma$ and  Notation \ref{notation 1}, we have  $d(x',x_1(s))<C_f(\sigma+\frac 12 \sigma_f)<\sigma_f$. Then again by  Notation \ref{notation 1},  $d_{\cWs}(x_2(s),x_1(s))<\sigma$ and $d_{\cWcu}(x_2(s),x')<C_f(\sigma+\frac 12 \sigma_f) <\sigma_f$, we see that $x_2(s)$ is uniquely determined. We argue similarly for $x_3(s)$. } 
	
\end{defi}

We now restrict our attention to maps in the region defined as follows. 

\begin{nota}\label{choosingneighbourhood}
	Assume that $f_0 \in \cPH^2(X)$ is dynamically coherent and center bunched. We consider the following cases:
	\enmt
	\item If $\dim E^c_{f_0}  = 1$ and $f$ is plaque expansive, then we let $\mathbb{U}(f_0)$ be a $C^1$-open neighbourhood of $f_0$ in which all maps are plaque expansive;
	\item If $\dim E^c_{f_0} \geq 2$ and satisfies  \ref{condition ae} or   \ref{condition be}, then we let $\mathbb{U}(f_0)$ be a $C^1$-open neighbourhood of $f_0$ such that $d_{C^1}(f_0, f)<\varepsilon_{f_0}$ for any $f \in \mathbb{U}(f_0)$ (see Notation \ref{notation 1}\eqref{item 5 definition 1}).
	\eenmt
	Moreover, we assume that $\mathbb{U}(f_0)$ is small enough so that any $g \in \mathbb{U}(f_0)$ is  $\theta'_{f_0}$-pinched and center bunched; and the constants  $h_{f_0}, \sigma_{f_0}, C_{f_0}$, $\overline{\Lambda}_{f_0}$ in Notation \ref{notation 1} work for any $g \in \mathbb{U}(f_0)$.
	By points \eqref{item 1 definition 1}, \eqref{item 2 definition 1}, \eqref{item 4 definition 1}, \eqref{item 5 definition 1} in Notation \ref{notation 1}, such $\mathbb{U}(f_0)$ exists. 
	We stress that we do \textit{not} require $\Lambda_f$ to be uniformly bounded for  $f \in \mathbb{U}(f_0)$.
\end{nota}

In the rest of this section, we fix a map $f_0$ as in Notation \ref{choosingneighbourhood}, and assume that $f \in \mathbb{U}(f_0)$. We also fix $x \in X$, $C > 0$, $\sigma \in (0, \frac{\overline{\sigma}_f}{2C_f})$, and take  $\{\gamma(s) = (x_1(s), x_2(s), x_3(s))\}_{s \in [0,1]}$, a $(\sigma, C)$-regular family of $f$-loops at $x$, determined by $x' \in \cWcs_f(x,\frac{1}{2}\sigma_f)$ and $(x_1(s))_{s \in [0,1]}$.  Let $\hat{f} \colon U \times X \to X$ be a $C^2$-deformation at $(a,f)$, and set $T  = T(\hat{f})$. We will always assume that $U$ is conveniently small so that for all $b \in U$, $\hat f(b,\cdot)\in \mathbb{U}(f_0)$ and the $C^2$-uniform constant  $\Lambda_f$ continues to work for $\hat f(b,\cdot)$. 
We now define a  continuation of $\gamma$ as follows. 
\begin{defi} \label{lift of loop for deformation}
	We define a lift of $\{\gamma(s)\}_s$ by
	\begin{equation} \label{term 2040}
	\hat{\gamma}(s) = ((a,x_1(s)), (a,x_2(s)), (a,x_3(s))), \quad \forall\,  s \in [0,1].
	\end{equation}
	Then by continuity, there exists a $C^2$-uniform constant $\delta_{a, T}=\delta_{a,T}( T )> 0$ 
	such that $B(a, \delta_{a, T}) \subset U$, and for any $(b,y) \in \cWc_T((a,x),\delta_{a,T})$,  any $s \in [0,1]$, each of the following intersections exists and is unique:
	\enmt
	\item $\{(b, \hat{x}_1(b,y, s))\}:=\cWu_T((b,y), h_f)\cap\cWcs_T((a,x_1(s)), h_f)$;
	\item $\{(b, \hat{x}_2(b,y, s))\}:=\cWs_T((b, \hat{x}_1(b,y,s)),h_f)\cap\cWcu_T((a,x'),h_f)$;
	\item $\{(b, \hat{x}_3(b,y, s))\}:=\cWu_T((b, \hat{x}_2(b,y,s)),h_f)\cap\cWcs_T((a,x),h_f)$.
	\eenmt
	We thus get a continuous family of $\hat{f}(b,\cdot)$-loops at $y$, denoted by $\{\gamma_{b,y}(s)\}_s$, where
	\begin{equation*}
	\gamma_{b,y}(s) = ( \hat{x}_1(b,y, s), \hat{x}_2(b,y, s), \hat{x}_3(b,y, s)), \quad \forall\, s\in [0,1].
	\end{equation*} 
	We further require that $\delta_{a,T}<\varepsilon_{f_0}/10$, and is sufficiently small so that for any $(b,y) \in \cWc_T((a,x),\delta_{a,T})$, we have $\ell(\gamma_{b,y}) < \overline{\sigma}_{\hat f(b,\cdot)}$ (see Notation \ref{notation 2}). 
	Note that in general, $\gamma_{b,y}$ is different from $\gamma_{\hat{f}(b,\cdot), y}$ as defined in Lemma \ref{extension continuous family floops}. 	
	
	Since $\ell(\gamma)<\sigma<\frac{\overline{\sigma}_f}{2C_f}$, there exists a $C^2$-uniform constant $\overline{\delta}_{a,T}=\overline{\delta}_{a,T}(T,\sigma)\in (0,\delta_{a,T}/2\Lambda_{f})$ such that for any $(b,y)\in \cW_T^c((a,x),\overline{\delta}_{a,T})$, we have $\ell(\gamma_{b,y})<2 \sigma$. 
	In the rest of the proof we will reduce the size of $\overline{\delta}_{a,T}$ finitely many times.
	By Notation \ref{notation 2}, the following map is well-defined:
	\begin{equation} \label{widehatpsidefi}
	\widehat{\psi} = \widehat{\psi}(T) \colon\left\{
	\begin{array}{rcl}
	\cWc_{T}((a,x), \overline{\delta}_{a,T}) \times [-1,2]^{c} &\to& \cWc_T(a,x),  \\
	(b,y,s)&\mapsto& (\prod_{j=1}^{c}H_{T, \hat \gamma(\varphi(j,s_j))})(b,y).
	\end{array}
	\right.
	\end{equation}
	Moreover, by Lemma \ref{lem41} and Notation \ref{notation 1}\eqref{item 1 definition 1} (recall Notation \ref{notation 2}), we have
	\begin{equation}\label{relationwidehatpsipsifxgamma}
	\pi_X\widehat{\psi}(b,y,\cdot) = \psi(\hat{f}(b,\cdot), y, \gamma_{b,y})(\cdot) \in \cWc_{\hat{f}(b,\cdot)}(y, 2K_f\sigma) \subset B(y, 2C_{f_0}K_f \sigma).
	\end{equation} 
\end{defi}

\begin{figure}
	\begin{center}
		\includegraphics [width=14cm]{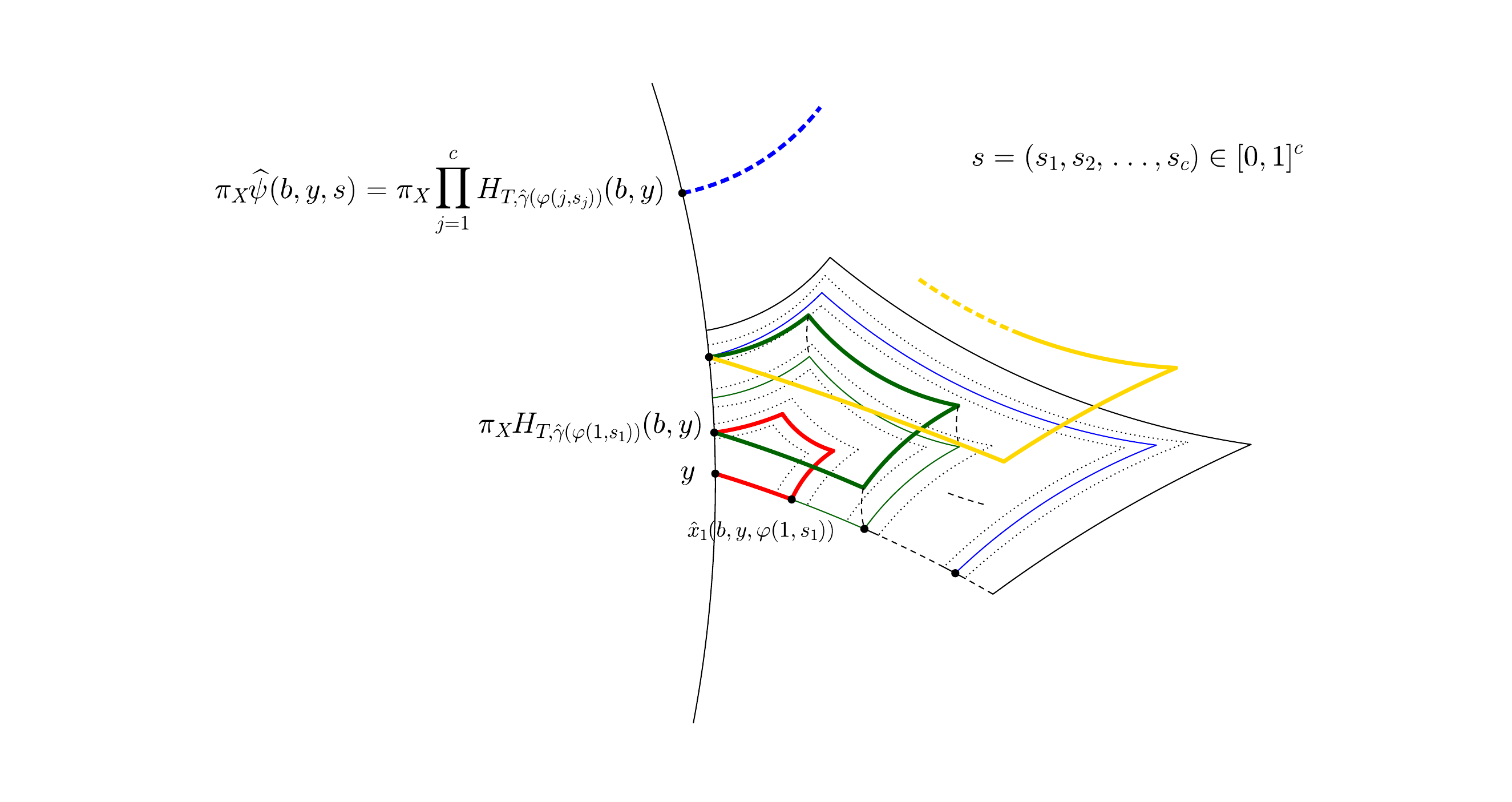}
	\end{center}
\caption{Parametrizing a subset of the accessibility class of $y$ for $\hat f(b,\cdot)$.}\label{fig 4}
\end{figure}

The following lemma is important. This is  the place where several technical conditions introduced earlier come into use.
\begin{lemma} \label{thetatheta}
	Let $f_0,f,\hat f,C,\sigma,x$ be given as above, and assume in addition  that $c \geq 2$  (in particular, $f_0$ satisfies \ref{condition ae} or \ref{condition be}) and $\sigma\in \big(0,\frac{1}{100}K_f^{-2}\overline{\Lambda}_{f_0}^{-2} \min(\sigma_{f_0}^2,\delta_{a,T}^2)\big)$. Then there exists a $C^2$-uniform constant $\widehat C_{2}=\widehat C_2(f) > 0$ such that for any $(b,y) \in \cWc_{T}((a,x), \overline{\delta}_{a,T})$, any $s,s'\in [-1,2]^{c}$, we have
	\begin{equation*}
	d(\widehat{\psi}(b,y, s) , \widehat{\psi}(b,y, s')) \leq \widehat C_2 C  |s-s'|^{\theta_0},
	\end{equation*}
	where $\theta_0$ is defined for \ref{condition ae} (resp. for \ref{condition be}),  as
	\begin{equation}\label{defi theta0}
	\theta_0 = \theta'_{f_0} (\theta''_{f_0})^3>\frac{c-1}{c}\quad \text{ \big(resp.}\ \theta_0 =  \theta'_{f_0} (\theta''_{f_0})^4>\frac{c-1}{c} \big).
	\end{equation}
\end{lemma}
\begin{proof}
	Let $s=(s_k),s'=(s_k') \in [-1,2]^c$, and for any $j\in\{1,\dots,c\}$, set $t_j:=\varphi(j,s_j)\in [0,1]$ and $t_j':=\varphi(j,s_j')\in [0,1]$. For each $0 \leq i \leq c$, we let 
	\begin{equation*}
	W_i :=  \prod_{j={1}}^{i} H_{T, \hat{\gamma}(t_j)}(b,y), \quad Z_{i} := \prod_{j=i+1}^{c}H_{T, \hat{\gamma}(t'_j)}(W_i).
	\end{equation*}
	Arguing as in Lemma \ref{lem41}, for each $0 \leq i \leq c$, we see that  $\pi_X(W_i) \in \cWc_{\hat{f}(b,\cdot)}(y, 2K_f\sigma)$, and hence $W_i \in \cW_T^c((a,x),\overline{\delta}_{a,T}+2K_f \sigma)\subset \cW_T^c((a,x),\delta_{a,T}
	)$. Similarly, for any $0\leq i\leq c-1$, $p\in [0,1]$, we have $H_{T,(a,x),(a,x_1(p))}(W_i) \in \cW_T^c((a,x_1(p)),\delta_{a,T})$.
	
	It is direct to see that $Z_0 = \widehat{\psi}(b,y,s')$, $Z_c = \widehat{\psi}(b,y,s)$. Thus it is enough to estimate $d(Z_0,Z_c)$.
	For any $0 \leq i \leq c-1$, we observe that
	\begin{equation*}
	Z_i = \prod_{j=i+2}^{c} H_{T, \widehat{\gamma}(t'_j)}(H_{T, \widehat{\gamma}(t'_{i+1})}(W_i)), \quad Z_{i+1} = \prod_{j=i+2}^{c} H_{T, \widehat{\gamma}(t'_j)}(H_{T, \widehat{\gamma}(t_{i+1})}(W_i)).
	\end{equation*}
	Since $f$, thus $T$, are $C^2$ and center bunched, and the maps $\{H_{T, \widehat{\gamma}(t'_j)}\}_{j=1,\dots,c}$ are obtained by composing holonomy maps, Notation \ref{notation 1}\eqref{item 3 definition 1} yields 
	\begin{equation*}
	d(Z_i, Z_{i+1}) \leq C_{f_0} \Lambda_{f}^{4c} d(H_{T, \widehat{\gamma}(t'_{i+1})}(W_i), H_{T, \widehat{\gamma}(t_{i+1})}(W_i)).
	\end{equation*}
	Therefore, it suffices to prove that for some $C^2$-uniform constant $c_1=c_1(f_0)>0$, for any $z \in \cW_T^c((a,x),\delta_{a,T})$, any $p,q \in [0,1]$, we have 
	\begin{equation*}
	d_{\cWc_{\hat{f}(b,\cdot)}}(\pi_X (H_{T, \hat{\gamma}(p)}(z)) , \pi_X (H_{T,\hat{\gamma}(q)}(z)))  \leq c_1  |p-q|^{\theta_0}. 
	\end{equation*}
	Given any $z=(b,\pi_X(z))\in \cW_T^c((a,x),\delta_{a,T})$, any $p,q\in [0,1]$, we set 
	\begin{equation*}
	\begin{array}{lcl}
	z_1 := H^u_{T, (a,x), (a, x_1(p))}(z), && z'_1 := H^u_{T, (a,x), (a, x_1(q))}(z), \\
	z_2 := H^s_{T, (a, x_1(p)), (a, x_2(p))}(z_1), &&  z'_2 := H^s_{T, (a, x_1(q)), (a, x_2(q))}(z'_1), \\
	z_3 := H^u_{T, (a, x_2(p)), (a, x_3(p))}(z_2), &&  z'_3 := H^u_{T, (a, x_2(q)), (a, x_3(q))}(z'_2).
	\end{array}
	\end{equation*}
	\begin{claim}\label{claime deux} 
		We have $d_{\cWu_{\hat{f}(b,\cdot)}}(\pi_X(z_1),\pi_X(z'_1)) \leq \overline{\Lambda}_{f_0}^2 C_{f_0} d_{\cWu_f}(x_1(p), x_1(q))^{\theta_1}$. Here $\theta_1= (\theta''_{f_0})^3$ if \ref{condition a} is satisfied, otherwise $\theta_1= (\theta''_{f_0})^4$, when \ref{condition b} is satisfied.
	\end{claim}
	\begin{proof}
		We abbreviate $x_1(p)$ (resp. $x_1(q)$) as $x_1$ (resp. $x'_1$).  It is clear that $z_1,z_1' \in \{b\}\times X$. 
		By Notation \ref{notation 1}\eqref{item 5 definition 1} and Notation \ref{choosingneighbourhood}, we see that there exists a leaf conjugacy between $\cWc_{f}$ and $\cWc_{\hat{f}(b,\cdot)}$, denoted by $\mathfrak{h}=\mathfrak{h}_{f,\hat f(b,\cdot)}$, such that $d(\mathfrak{h}(x_1), \mathfrak{h}(x'_1)) < \overline{\Lambda}_{f_0} d_{\cWu_f}(x_1, x'_1)^{(\theta''_{f_0})^2}\leq \overline{\Lambda}_{f_0}(2\sigma)^{1/2}<\frac{\sigma_{f_0}}{2}$. 
		
		By reducing the size of $\overline{\delta}_{a,T}$ if necessary, we may suppose that both $z_1$ and $(b,\mathfrak{h}(x_1))$ belong to $W^c_T((a,x_1), \sigma_{f_0}/(4C_{f_0})) \cap \{b\} \times X$.
		This implies that $\pi_X(z_1)\in \cW^c_{\hat f(b,\cdot)}(\mathfrak{h}(x_1),\sigma_{f_0}/2)$. Similarly, $\pi_X(z_1')\in \cW_{\hat f(b,\cdot)}(\mathfrak{h}(x_1'),\sigma_{f_0}/2)$.
		
		By definition, we have $\pi_X(z_1)\in \cW_{\hat f(b,\cdot)}^u(\pi_X(z_1'),h_{f_0})$. Then by Notation \ref{notation 1}\eqref{item 1 definition 1}, \eqref{item 5 definition 1}, we can see that 
		\begin{equation}\label{ew sept six}
		d_{\cWu_{\hat{f}(b,\cdot)}}(\pi_X(z_1),\pi_X(z'_1)) \leq \overline{\Lambda}_{f_0} C_{f_0}  d(\mathfrak{h}(x_1), \mathfrak{h}(x'_1))^{\theta_2},
		\end{equation}
		where $\theta_2 = \theta''_{f_0}$ if  \ref{condition a} is satisfied; $\theta_2 = (\theta''_{f_0})^2$ if \ref{condition b} is satisfied. Hence by Notation \ref{notation 1}\eqref{subitem 6.2}, the right hand side of \eqref{ew sept six} is at most $ \overline{\Lambda}_{f_0}^2 C_{f_0} d_{\cWu_f}(x_1,x'_1)^{(\theta_{f_0}'')^2 \theta_2}=\overline{\Lambda}_{f_0}^2 C_{f_0} d_{\cWu_f}(x_1,x'_1)^{\theta_1}$.
	\end{proof}

	By Notation \ref{notation 1}\eqref{item 4 definition 1}, Notation \ref{choosingneighbourhood},  Claim \eqref{claime deux} and $\sigma<\Big(\frac{\sigma_{f_0}}{10 \Lambda_f \overline{\Lambda}_{f_0}^2 C_{f_0}}\Big)^2$, we obtain 
	\begin{align}\label{eq sept point sept}
	d_{\cWcu_{\hat{f}(b,\cdot)}}(\pi_X(z_2),\pi_X(z'_2)) &\leq  \Lambda_{f} d_{\cWcu_{\hat{f}(b,\cdot)}}(\pi_X(z_1), \pi_X(z'_1))^{\theta'_{f_0}}\nonumber  \\ & \leq \Lambda_f \overline{\Lambda}_{f_0}^{2}C_{f_0}  d_{\cWu_f}(x_1(p), x_1(q))^{\theta'_{f_0} \theta_1}<\sigma_{f_0}.
	\end{align}  
	By Notation \ref{notation 1}\eqref{item 2 definition 1},\eqref{item 3 definition 1}, and since $f$ is $C^2$ and center bunched, we get
	\begin{align}\label{eq sept point huit} d_{\cWc_{\hat{f}(b,\cdot)}}(\pi_X(H_{T, \hat{\gamma}(p)}(z)) , \pi_X( H_{T,\hat{\gamma}(q)}(z))) 
	&\leq \Lambda_f d_{\cWc_{\hat{f}(b,\cdot)}}(\pi_X(z_3),\pi_X(z'_3))\nonumber \\ &\leq \Lambda_f^2 C_{f_0} d_{\cWcu_{\hat{f}(b,\cdot)}}(\pi_X(z_2), \pi_X(z'_2)).
	\end{align}
	Since $\gamma$ is $(\sigma, C)$-regular, we have $d_{\cWu_f}(x_1(p), x_1(q)) \leq C |p-q|$. We conclude the proof by \eqref{eq sept point sept}, \eqref{eq sept point huit} and by noting that $\theta_0 = \theta'_{f_0} \theta_1$.
\end{proof}

\section{Constructing charts and vector fields}\label{sec construct cha}

In order to construct  infinitesimal deformations with required properties, we will first introduce coordinates in a neighbourhood of each $c$-disk. In this section, we assume that \hyperref[(H1)]{(H1)} holds, and $r \geq 2$.

In the following, our goal is to define certain vector fields in order to perturb the dynamics and induce a displacement of the holonomies. More precisely, given a small center disk, we define a vector field localized close to the disk. These vector fields will be rich enough for us to apply Proposition \ref{determinant for smooth deformations}.   

\begin{construction}\label{given c disk get a chart}
There exist $C^2$-uniform constants $\overline{h}_{f}\in (0, h_f)$,  $\overline{C}_{f}  > 1$ such that the following is true. For any $c$-disk of $f$, denoted by $\cC = \cWc_f(x,h)$, with $x \in X$ and $h \in (0, \overline{h}_{f})$, there exists a $C^r$ volume preserving map $\phi = \phi(\cC) \colon (-h, h)^{d} \to X$ such that $\phi(0) = x$ and
\enmt
\item\label{notation 3 1} $\frac{1}{5}\cC \subset \phi((-h/ 4, h/ 4)^{c} \times \{0\}^{d_u + d_s} )\subset \phi((-2h/3, 2h/3)^{c} \times \{0\}^{d_u + d_s}) \subset \cC$;
\item\label{notation 3 2} $\norm{\phi}_{C^2} < \overline{C}_{f}$;
\item\label{notation 3 3} $D\phi(0,  \R^{c} \times \{0\}^{d_u + d_s})$, $D\phi(0, \{0\}^{c} \times \R^{d_u} \times \{0\}^{d_s})$, $D\phi(0, \{0\}^{c+d_u} \times \R^{d_s})$ are respectively equal to $E^{c}_f(x),E^{u}_f(x), E^{s}_f(x)$;
\item\label{notation 3 4} for any $y \in \phi((-h,h)^{d})$, $\Pi_cD_y(\phi^{-1})\colon E^c_{f}(y) \to \R^c$ has determinant in $\big(\overline{C}_f^{-1},\overline{C}_f\big)$, where $\Pi_c \colon \R^{d} \simeq \R^{c} \times \R^{d_u} \times \R^{d_s} \to \R^{c}$ is the canonical projection;
\item\label{notation 3 5}
for any $\zeta > 0$, there exists a $C^1$-uniform constant $\overline{h}_{f, \zeta} \in (0, \overline{h}_f)$ so that if $h \in (0, \overline{h}_{f, \zeta})$, then for any $y \in \phi((-h,h)^{d})$, $\Pi_cD_y(\phi^{-1})\colon E^{su}_{f}(y)  \to \R^c$ has norm  smaller than $\zeta$.
\eenmt

The above charts exist. Indeed, for some $C^2$-uniform constant $\overline{C}_f'>0$, for any $x \in X$, we can choose a $C^\infty$ volume preserving diffeomorphism $\phi'\colon (-2h,2h)^d \to X$ satisfying: $\phi'(0)=x$;  \eqref{notation 3 2},\eqref{notation 3 3} for $(\phi',\overline{C}_f')$ in place of $(\phi,\overline{C}_f)$; and that $D\phi'(0)$ is an isometry restricted to $\R^c \times \{0\}^{d_u+d_s}$, etc. Then, for sufficiently small $h$, $(\phi')^{-1}(\mathcal{C})$ is contained in the graph of a $C^r$ map $\psi\colon (-2h,2h)^c \to (-2h,2h)^{d_u+ds}$ with $\psi(0)=0$ and $\|\psi\|_{C^2} < \overline{C}_f''$ for some $C^2$-uniform constant $\overline{C}_f''>0$. For $x=(x_c,x_{us})\in (-h,h)^c \times (-h,h)^{d_u+d_s}$, we define $\phi(x)=\phi'(x_c,x_{us}+\psi(x_c))$. It is direct to verify \eqref{notation 3 1}-\eqref{notation 3 5} by taking $\overline{h}_f$ sufficiently small, and $\overline{C}_f$ sufficiently large.

For $*=u,s$, set $e_* := (1,0,\dots, 0) \in \R^{d_*}$. 
For any $0<\lambda<h<\overline{h}_f$, we define
\begin{align*}
\cWcs(\lambda) &:= \phi((- h, h)^{c} \times \{ \lambda e_u \} \times (- h, h)^{d_s}),\\
\cWcu(\lambda) &:= \phi((- h, h)^{c} \times (-h, h)^{d_u} \times \{\lambda e_s\}).
\end{align*}
\end{construction}

We construct regular families of loops as follows. 

\begin{lemma}\label{lem def regular family and chart}
There exist $C^2$-uniform constants $\widetilde{h}_f \in (0,\overline{h}_f)$,  $\widetilde{C}_{f} > 0$ such 
that the following is true. For any $\rho \in (0, \widetilde{h}_f)$, any $\sigma \in (0, \widetilde{C}_f^{-1}\rho)$, there exist constants $\hat \varepsilon_0 =\hat \varepsilon_0(f, \rho, \sigma)>0$, $\hat \sigma_0 = \hat \sigma_0(f, \rho,\sigma)\in (0,\sigma)$, such that for any $c$-disk  $\cC$  of $f$ with radius $h \in (\rho, \widetilde{h}_f)$; any $g \in \cPH^1(X)$ such that $d_{C^1}(f,g) < \hat \varepsilon_0$; any $x \in (\frac{1}{5}\cC, \hat  \sigma_0)$, there exists $\{\gamma(s) = (x_{1}(s), x_{2}(s), x_{3}(s))\}_{s \in [0,1]}$, a $(\sigma, \widetilde{C}_{f})$-regular family of $g$-loops at $x$ with the following properties: let $\phi=\phi(\mathcal{C})$, $\sigma':= \widetilde{C}_f^{-\frac{1}{2}}\sigma$; then we have
\begin{enumerate}[label=(\roman*)]
\item for any $s \in [0,1]$, any $i= 2,3$, $\cWc_{g}(x_i(s), K_f\sigma)$ is disjoint from the image $\phi((-h,h)^{c+d_u} \times (-\frac{\sigma'}{2}, \frac{\sigma'}{2})^{d_s})$;
\item take $C_{\mathrm{min}}$ as in \eqref{definition cmin}. For any $s \in [0,1]$, we have
\begin{equation*}
\cWc_{g}(x_1(s),K_f\sigma) \subset \phi\big((-h/2,h/2)^c\times (\sigma' s e_u+(-C^{-1}_{\mathrm{min}}\sigma',C^{-1}_{\mathrm{min}}\sigma')^{d_u})\times (-\frac{\sigma'}{5},\frac{\sigma'}{5})^{d_s}\big).
\end{equation*}
\end{enumerate}
\end{lemma}

\begin{proof}
Set $\{x'\} := \cWs_g(x, h_f) \cap \cWcu(\sigma')$. By Notation \ref{notation 1},  Construction \ref{given c disk get a chart}(2),(3), and by taking  $\widetilde{C}_f$ sufficiently large,  $\widetilde{h}_f, \hat{\varepsilon}_0, \hat{\sigma}_0$ sufficiently small, we have:
\enmt
\item For each $s \in [0,1]$, each of the following intersections exists and is unique:
\begin{enumerate}[label=(\alph*)]
	\item $\{ x_1(s)\} := \cWu_g(x, h_f) \cap \cWcs(s\sigma')$;
	\item $\{x_2(s)\} := \cWs_g(x_1(s), h_f) \cap \cWcu_g(x', h_f)$;
	\item $\{ x_3(s) \} := \cWu_g(x_2(s), h_f) \cap \cWcs_g(x, h_f)$.
\end{enumerate}   
\item For each $s \in [0,1]$, set $\gamma(s) := (x_1(s), x_2(s), x_3(s))$. Then $\gamma:=\{\gamma(s)\}_{s\in [0,1]}$ is a $(\sigma, \widetilde{C}_f)$-regular family for $g$.
\eenmt

Note that for any $s \in [0,1]$, $x_2(s),x_3(s)\in \cW_g^{cu}(x',\widetilde{C}_f\sigma')$ for sufficiently large $\widetilde{C}_f$. We get $(i)$ by the continuity of $E_f^{cu}$, and by letting $\widetilde{h}_f$, $\hat \varepsilon_0$, $\hat \sigma_0$ be sufficiently small. We get $(ii)$ by the continuity of $E_f^{cs}$ and $E_f^c$; by  $\widetilde{C}_f\sigma < h < \widetilde{h}_f$; by letting $\widetilde{C}_f$  be sufficiently large, and then letting $\widetilde{h}_f,\hat \varepsilon_0,\hat \sigma_0$  be sufficiently small.
\end{proof}

\begin{construction} \label{get a set of vector fields}
	
	For any $c$-disk $\cC$ such that $\varrho(\cC)=:h \in (0, \overline{h}_{f})$, with $\overline{h}_{f}$ as in Construction \ref{given c disk get a chart}, for any $\sigma \in (0,h)$, 
	we define a collection of vector fields as follows.
	
	$\bullet$ For each $1\leq j\leq c$, let $U_j \colon (-2/3, 2/3)^{c} \times (-1, 1)^{d_u} \times (-1/3, 1/3)^{d_s} \to \R^{d}$ be a compactly supported $C^{\infty}$ divergence-free vector field such that $U_j$ restricted to $(-1/2, 1/2)^{c + d_u} \times (-1/5, 1/5)^{d_s}$ is equal to the constant vector, denoted by $E_j$, that has $1$ at $j$-th coordinate and $0$ at the others.
	Such $U_j$ always exists since $d \geq 3$.
	Moreover, we can assume that $U_j$ satisfies $\|U_j\|_{C^1} < C_{*}$ for some  constant $C_{*} = C_{*}(d) > 0$.

	For any $x_c \in \R^c$, $x_u \in \R^{d_u}$, $x_s \in \R^{d_s}$, $a_c, a_u, a_s > 0$, we denote for every $z_c \in \R^c, z_u \in \R^{d_u}, z_s \in \R^{d_s}$:
	\begin{equation*}
	P_{x_c, a_c, x_u, a_u, x_s, a_s}(z_c, z_u, z_s) = (x_c + a_cz_c, x_u + a_u z_u, x_s + a_s z_s).
	\end{equation*}
	
	Now, for any $i,j \in \{1,\dots, c\}$, any $t \in \cB_i$, we let $U^{\sigma}_{\cC,i,t,j} \colon (-h, h)^d \to \R^d$ be the vector field 
	\begin{align*}
	U^{\sigma}_{\cC, i, t , j} = U_j (P_{\underline{0^c}, h, \varphi(i,t)\sigma e_u, 2C_{\mathrm{min}}^{-1}\sigma, \underline{0^{d_s}}, \sigma})^{-1}.
	\end{align*}
	The support of $U^{\sigma}_{\cC, i, t , j}$ is contained in
	\begin{equation*}
	(-2h/3, 2h/3)^c \times (\varphi(i,t) \sigma e_u + 2C_{\mathrm{min}}^{-1}(-\sigma, \sigma)^{d_u}) \times (-\sigma/3, \sigma/3)^{d_s}.
	\end{equation*}
	Moreover, for any $z_c \in (-h/2, h/2)^c$, $z_u \in \varphi(i,t) \sigma e_u + C_{\mathrm{min}}^{-1}(-\sigma, \sigma)^{d_u}$ and $z_s \in (-\sigma/5, \sigma/5)^{d_s}$, we have
	\begin{equation*}
	U^{\sigma}_{\cC, i, t , j}(z_c, z_u, z_s) = E_j.
	\end{equation*}
	We set
	\begin{equation*}
	V^{\sigma}_{\cC, i,t, j}  := D\phi(U^{\sigma}_{\cC, i,t, j}).
	\end{equation*}
\end{construction}

By  Construction \ref{given c disk get a chart}, and the $C^1$-bound on $U_j$ above, we see that the vector field $V_{\mathcal{C},i,t,j}^\sigma$ is divergence-free and satisfies:
\begin{equation}\label{eq Vitj adapt C1}
\sigma \norm{ \partial_x V^{\sigma}_{\cC, i, t ,j}}_{X} + \norm{ V^{\sigma}_{\cC, i, t ,j}}_{X} < \widehat{C}_f,
\end{equation}
for some $C^2$-uniform constant $\widehat{C}_f >0$, independent of $\mathcal{C},i,t,j$.

\begin{rema} \label{remaaboutthesupportofthevectorfields}
By construction, it is clear that
\begin{equation*}
\mathrm{supp}_X(V^{\sigma}_{\cC, i, t, j}) \subset \phi((-2h/3,2h/3)^{c} \times (-2\sigma, 2\sigma)^{d_u}\times (-\sigma/3,\sigma/3)^{d_s}).
\end{equation*}
Thus for any $\sigma_0 > 0$, there exists $\sigma > 0$ such that for any $\cC, i, j, t$ in  Construction \ref{get a set of vector fields},
\begin{equation*}
\mathrm{supp}_X(V^{\sigma}_{\cC, i, t, j}) \subset (\cC, \sigma_0).
\end{equation*}
\end{rema}



The following lemma describes the values taken by $V^{\sigma}_{\cC, i, t, j}$ at the corners of the loops that  we constructed in  Lemma \ref{lem def regular family and chart}. 

\begin{lemma} \label{lem vector field and loop}
There exists a $C^2$-uniform constant $\kappa'=\kappa'(f)>0$ such that for any $ \rho_1 \in ( 0, \frac{1}{2}\widetilde{h}_f)$, any $ \sigma \in ( 0, \widetilde{C}_f^{-1}\rho_1)$, 
there exists  
$\mathcal{U}'=\mathcal{U}'(f,\rho_1,\sigma)$, a $C^2$ open neighbourhood of $f$, such that for any $g \in \mathcal{U}'$, for any $\big(\frac{1}{20},6\big)$-spanning $c$-family for $f$ denoted by $\mathcal{D}$, satisfying $[\underline{r}(\mathcal{D}),\overline{r}(\mathcal{D})]\subset (\rho_1,\frac 12 \widetilde{h}_f)$, there exists $\mathcal{D}'$, a $\big(\frac{1}{10},8\big)$-spanning $c$-family for $g$ with $[\underline{r}(\mathcal{D}'),\overline{r}(\mathcal{D}')]\subset (\rho_1,  \widetilde{h}_f)$  satisfying the following  property. 

For any $\mathcal{C}'\in \mathcal{D}'$,   we have $\cC' \subset (\cC, \sigma)$   for some $\cC\in \cD$, and for each $x \in \frac{1}{5}\cC'$, there exists a $(\sigma,\widetilde C_f)$-regular continuous family of $g$-loops at $x$, denoted by $\{\gamma(s) = (x_1(s), x_2(s), x_3(s))\}_{s \in [0,1]}$, such that for any $i \in \{1,\dots, c\}$, $s \in [0,1]$, $t \in \cB_i$, any $y \in   \cWc_{ g}(x,K_{f}\sigma) \cup  \cWc_{g}(x_2(s), K_{f}\sigma)\cup  \cWc_{g}(x_3(s), K_{f}\sigma)$, and any $z\in \cWc_{g}(x_1(\varphi(i,t)), K_{f}\sigma)$, by  letting $\sigma' := \widetilde{C}_{f}^{-\frac{1}{2}}\sigma$, we have 
\begin{equation}\label{v sigma 2}
V^{\sigma'}_{\cC, i, t, j}(y) = 0,\ \forall\,1 \leq j \leq c,\qquad  |\det\big((\pi_cV^{\sigma'}_{\cC, i,t, j}(z))_{j=1,\dots,c}\big)| > \kappa'.
\end{equation}
\end{lemma}
\begin{proof}
	Take $\varepsilon_1:=\hat \varepsilon_0(f,\rho_1,\sigma)$ and   $\sigma_1:=\hat \sigma_0(f,\rho_1,\sigma)\in (0,\sigma)$ given by Lemma \ref{lem def regular family and chart}. 
	
By Lemma \ref{lemma stability spanning} and Remark \ref{remarque trois} applied to $f$, and  $(6,\frac{1}{20}, \frac{1}{10},\frac 15,\rho_1,  \widetilde h_f,\sigma_1)$ in place of $(k,\theta,\theta',\theta'',\rho_m,\rho_M,\sigma)$, for all $g$ sufficiently $C^1$-close to $f$,  there exists $\cD'$, a $(\frac{1}{10},  8)$-spanning $c$-family for $g$, with $[\underline{r}(\cD'), \overline{r}(\cD')] \subset (\rho_1, \widetilde{h}_f)$, such that  for each $\cC' \in \cD'$, we have $\cC' \subset (\cC, \sigma_1)\subset (\cC,\sigma)$ and $\frac{1}{5}\cC'  \subset (\frac{1}{5}\cC,\sigma_1)\subset (\frac{1}{5}\cC,\sigma)$ for some $\cC \in \cD$. 

We can apply Lemma \ref{lem def regular family and chart} to any $c$-disk $\cC$  of $f$, since it has radius in $(\rho_1, \frac{1}{2}\widetilde{h}_{f})$, and any $x \in \frac{1}{5}\cC'$, to construct $\{\gamma(s)\}_{s \in [0,1]}$, a $(\sigma,\widetilde C_f)$-regular continuous family of $g$-loops at $x$  such that if $i,j,s,t, \sigma'$ are as in the lemma, we have
\enmt
\item $\cWc_{g}(x_i(s), K_{f}\sigma)$, $i=2,3$, are disjoint from $\phi((-h,h)^{c+d_u} \times (-\frac{\sigma'}{2}, \frac{\sigma'}{2})^{d_s})$;
\item $\cWc_{g}(x, K_{f}\sigma) \subset \phi((-h/2,h/2)^c\times (-C^{-1}_{\mathrm{min}}\sigma',C^{-1}_{\mathrm{min}}\sigma')^{d_u}\times (-\frac{\sigma'}{5},\frac{\sigma'}{5})^{d_s})$;
\item $\cWc_{ g}(x_1(\varphi(i,t)), K_f\sigma) \subset \phi\big((-h/2,h/2)^c\times (\sigma'\varphi(i,t)  e_u+\\ +(-C^{-1}_{\mathrm{min}}\sigma',C^{-1}_{\mathrm{min}}\sigma')^{d_u})\times (-\frac{\sigma'}{5},\frac{\sigma'}{5})^{d_s}\big)$.
\eenmt
By  Construction \ref{given c disk get a chart}\eqref{notation 3 4} and  Construction \ref{get a set of vector fields},  we see that  \eqref{v sigma 2} holds for some $C^2$-uniform constant $\kappa'>0$  depending only on $f,\overline{C}_f$.
\end{proof}

\section{On the prevalence of the accessibility property}\label{section generic open acc}

In this section, we fix  $r  \in \N_{ \geq 2} \cup \{\infty\}$ and  an integer $J\geq 1$.  In order to avoid repetition, we consider only the volume preserving case in the following.  The more general case is handled by repeating exactly the same  proof after replacing $\diff^r(X,\mathrm{Vol})$ by $\diff^r(X)$, $\cPH^r(X,\mathrm{Vol})$ by $\cPH^r(X)$, etc. 

Let us first give an outline of the construction in this section with an illustration in Figure \ref{fig 5}. 
Given a good $C^r-J-$family ${\bf f} := \{f_\omega\}_{\omega \in [0,1]^J}$, we will find a family $\{ \hat{\bf f}^\theta \}_{\theta \in U_1 \subset \R^I}$ of $C^r-J-$families which are perturbations of ${\bf f}$, in which most  $\hat{\bf f}^\theta$  contain a large proportion of accessible maps. More precisely, we will construct a $C^r-{(I+J)}-$family ${\hat {\bf f}} := \{ f_{(\omega, \theta)} \}_{(\omega , \theta) \in  [0,1]^J \times U_1 }$, where $U_1$ is a neighbourhood of the origin $0 \in \R^I$, such that $\{ f_{(\omega, 0)} \}_{\omega \in [0,1]^J} = {\bf f}$. 
\begin{figure}[H]
	\begin{center}
		\includegraphics [width=13.5cm]{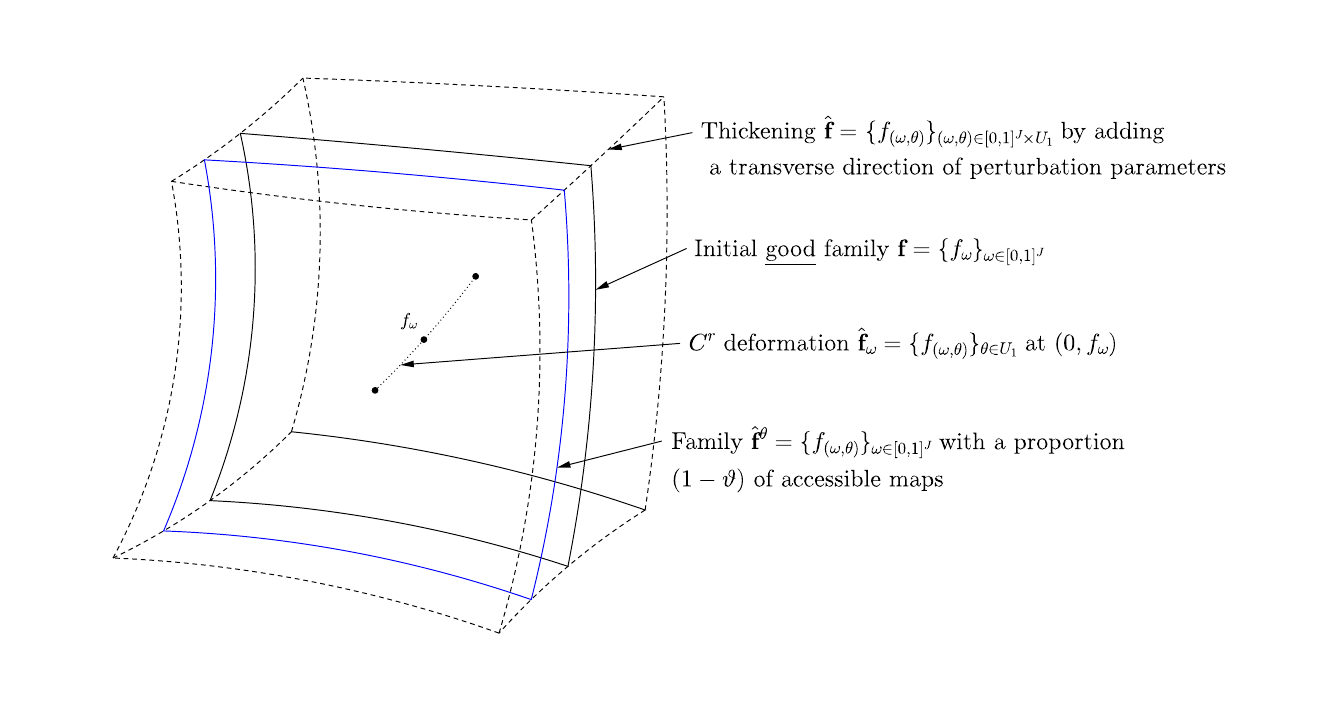}
	\end{center}
	\caption{Selection of a perturbed family $\hat{\bf f}^\theta$ with many accessible maps.} \label{fig 5}
\end{figure}
We construct $\hat{\bf f}$ in the following  way.
We apply Proposition \ref{prop slow recurrent} repeatedly to produce a well-distributed finite subset $A$ in $[0,1]^J$ such that for each parameter $a \in A$ we get a $(\frac{1}{20}, 6)$-spanning $c$-family $\cD_a$ for $f_a$, and produce by Lemma \ref{lem vector field and loop} a vector field $V^{\sigma}_{\cC, i, t, j}$ for each $\cC \in \cD_a$, $1 \leq i,j \leq c$, $t \in \cB_i$ and some small $\sigma > 0$.
We construct a $C^r$ map $V\colon [0,1]^{J} \times \R^I \times X \to TX$ by gluing together the above data in a careful way, and define $\hat{\bf f} = \{ f_{(\omega, \theta)} \}_{(\omega,\theta)\in [0,1]^J \times U_1}$ so that for each $\omega \in [0,1]^J$, $\hat{\bf f}_\omega :=\{f_{(\omega, \theta)}\}_{\theta \in U_1} $ is a $C^r$ deformation at $(0, f_\omega)$ with $I$-parameters generated by $V(\omega, \cdot, \cdot)$. By choosing $V$ carefully,  we may ensure that for a typical parameter $\omega \in [0,1]^J$, the $C^r$ deformation $\hat{\bf f}_\omega$ exhibits approximately independent perturbations for many different $su$-paths.  Together with a Fubini's argument, this will enable us to verify property $(\cP)$ for maps at all but extremely small amount of parameters within a typical family $\hat{\bf f}^\theta:= \{ f_{(\omega, \theta)} \}_{\omega  \in  [0,1]^J}$.
Notice that the maps in $\hat{\bf f}$ have uniformly bounded $C^r$-norms. Thus, by studying carefully the proofs of this section, we can see that throughout this paper we only need to use the fact that $h_f$, $\sigma_f$, $C_f$ and $\Lambda_f$ are $C^2$-uniform constants.

\subsection{Constructing perturbations for a family of diffeomorphisms}\label{Constructing perturbations for a family of diffeomorphisms}
In this subsection, we fix 
a $C^{r}-J-$family $\{f_\omega\}_{\omega \in [0,1]^{J}}$  in the space of dynamically coherent, center bunched $C^r$ partially hyperbolic diffeomorphisms on $X$.

Let $\Omega_0$ be an open set compactly supported in $(0,1)^{J}$, let $U_1$ be an open neighbourhood of the origin in $\R^{I}$ for some integer $I \geq 1$, and let $\hat{f} \colon \Omega_0 \times U_1 \times X \to X$ be a $C^{r}$ map such that  $\hat f(a,b,\cdot) \in \diff^r(X)$, for all $(a,b) \in \Omega_0 \times U_1$, and $\hat{f}(a,0,\cdot) = f_a$ for all $a \in \Omega_0$. In particular, for any $a \in \Omega_0$, the map $\hat{f}(a,\cdot) \colon U_1 \times X \to X$ is a $C^{r}$-deformation at $(0,f_a)$. We set $
T_a := T(\hat{f}(a,\cdot))$. Moreover, by applying Lemma \ref{lemma T deformation}
 to $\hat{f}(a,\cdot)$ in place of $\hat{f}$, after taking $U_1$ sufficiently small, for any $(b,x) \in U_1 \times X$, we will denote by $\nu^{a}_{b}(x,\cdot) \colon \R^I \to E^{su}_{\hat{f}(a,b,\cdot)}(x)$  the unique linear map such that
\begin{equation}\label{term 2050}
E^c_{T_a}(b,x) = Graph(\nu^{a}_{b}(x,\cdot))\oplus E^{c}_{\hat{f}(a,b,\cdot)}(x).
\end{equation}

Given an element of a $C^r-J-$family as above, the following notion  combines a global property (through spanning $c$-families) and a local one (existence of deformations which induce an infinitesimal displacement of the holonomies in many directions) which together will be useful to verify Property $(\mathcal{P})$ in Definition \ref{property cP}.

\begin{defi}[Removability] \label{def remov}
Let $\hat{f}$ be as above. Then for $\rho_m, \rho_{M}, \sigma, C, \kappa > 0$, $a \in \Omega_0$, we say that $\hat f$ is  \underline{$(\rho_m, \rho_{M},\sigma,C, \kappa)$-Removable at $a$} if the following is true.
There exists  $\cD$, a  $(\frac{1}{10},  8)$-spanning $c$-family for $f_{a}$ with $[\underline{r}(\cD), \overline{r}(\cD)] \subset ( \rho_m, \rho_{M})$, such that for each $\cC \in \cD$, for each $x \in \frac{1}{5}\cC$, there exists a $(\sigma, C)$-regular continuous family $\gamma$ of $f_{a}$-loops at $x$ with the following properties. Let $K_0,\Gamma$ be taken  as in Subsection \ref{subs disjoint squares}. For any $(i, \cB, \{s_{t}\}_{t \in \cB}) \in \Gamma$ and $(t,j) \in  \cB\times \{1,\dots,c\}$, we set $\gamma_{t,j}:= \gamma(\varphi(j, s_{t,j}))$, $z_t:= (\prod_{j=1}^cH_{f_{a}, \gamma_{t,j}})(x)$. Let $\hat{\gamma}_{t,j}$ be the lift of $\gamma_{t,}$ for $T_a$, and 
\begin{equation}\label{def Xi a x}
\Xi_{a,x} \colon \left\{
\begin{array}{rcl}
T_0U_1 &\to& \prod\limits_{t \in \cB} E^{c}_{f_{a}}(z_t) \simeq \R^{K_0 c},  \\
B &\mapsto& \left[ \pi_c\big(D\left(\prod_{j=1}^c H_{T_{a}, \hat{\gamma}_{t,j}}\right) \cdot (B + \nu^{a}_{0}(x, B))\big) \right]_{t \in \cB}.
\end{array}
\right.
\end{equation}
Then there exists a linear subspace $H \subset \R^I$ of dimension $K_0c$ such that
\begin{equation*}
|\det(\Xi_{a,x} |_{H})| > \kappa.
\end{equation*}

\end{defi}

The main goal of this subsection is the following. 

\begin{prop}\label{lower bound determinant 2}
Assume that  $\{f_\omega\}_{\omega \in [0,1]^J}$ is a good $C^r-J-$family. Then there exist constants $Q,C, \kappa_1 > 0$ such that  for any $\vartheta > 0$, 
 any sufficiently small $\widetilde{h}>0$, there exists  $\rho_1 \in (0, \frac{1}{2}\widetilde{h})$ such that for 
any sufficiently small $\sigma > 0$,  there exist 
 \begin{enumerate}
	\item an open set $\Omega_0$ compactly contained in $(0,1)^{J}$, with $\mathrm{Leb}([0,1]^{J} \backslash \Omega_0) < \vartheta$;
	\item an integer $I>0$, and an open neighbourhood $U_1$ of the origin in $\R^I$; 
	\item a $C^{r}$ map $\hat{f} \colon [0,1]^{J} \times U_1 \times X \to X$ such that
	\begin{enumerate}[label=(\roman*)]
		\item $\hat f(a,b,\cdot) \in \diff^r(X,\mathrm{Vol})$ and  $\hat{f}(a,0,\cdot) = f_a$, for all $(a,b) \in [0,1]^{J}\times U_1$; 
		\item $\norm{\hat{f}}_{C^1} < Q$;
		\item $\hat f$ is $(\rho_1,  \widetilde{h}, \sigma, C, \kappa_1)$-Removable at $a$,  for any $a \in \Omega_0$.
	\end{enumerate}
\end{enumerate}

\end{prop}

\begin{proof}
	By compactness, we can choose $\widehat{C}_1,C>0$ so that for all $a \in [0,1]^{J}$, $\widehat{C}_1> \widehat{C}_{f_a}$, $C> \widetilde{C}_{f_a}$, where  $\widehat{C}_{f_a}$ is given by  Construction \ref{get a set of vector fields};  $\widetilde{C}_{f_a}$ is given by Lemma \ref{lem def regular family and chart}. 
	
	We assume that $\widetilde{h}>0$ is sufficiently small so that for all $a \in [0,1]^{J}$, $\widetilde{h}< \widetilde{h}_{f_a}< \overline{h}_{f_a}$, where $\widetilde{h}_{f_a}$ is given by Lemma \ref{lem def regular family and chart}  and $\overline{h}_{f_a}$ is given by Construction \ref{given c disk get a chart}. Fix $\rho_1 \in (0,\frac 12 \widetilde{h})$. For any sufficiently small $\sigma>0$, we choose $\varepsilon_1,\sigma_1>0$ so that $\varepsilon_1< \hat \varepsilon_0(f_a,\rho_1,\sigma)$, $\sigma< \hat \sigma_0(f_a,\rho_1,\sigma)$ (see Lemma \ref{lem def regular family and chart}) for all $a \in [0,1]^{J}$. 
	
	Given any $C^r$ map $V\colon [0,1]^{J} \times \R^I \times X \to TX$ such that $V(a,\cdot)$ is an infinistesimal $C^r$ deformation with $I$-parameters for any $a \in [0,1]^{J}$, we associate with $V$ a $C^r$ map $\hat f\colon [0,1]^{J} \times U_1 \times X \to X$ by 
	 \begin{equation}\label{neuf trois}
	 \hat f \colon (a,b,x)\mapsto \mathscr{F}_{V(a,b,\cdot)}(1,f_a(x)),\qquad \forall\, (a,b,x) \in [0,1]^{J} \times U_1 \times X,	 \end{equation}
	 where $U_1$ is a sufficiently small neighbourhood of the origin in $\R^I$, and for any $(a,b)\in [0,1]^{J} \times U_1$, $\mathscr{F}_{V(a,b,\cdot)}\colon \R \times X \to X$ is the flow generated by $V(a,b,\cdot)$.  

To prepare for the proof of Proposition \ref{lower bound determinant 2}, we first show the following lemma.
\begin{lemma}\label{lem construct for one c family}
Set $\widehat{K}:=\sup_{a \in [0,1]^{J}} K_{f_a}+4$. Then there exist constants $R_1,\kappa_1>0$ such that the following is true. For any sufficiently small $\widetilde{h}>0$, for any $\rho_1 \in (0,\frac 12 \widetilde h)$, any sufficiently small $\sigma > 0$, there exists a constant $\lambda_1= \lambda_1(\rho_1, \sigma) > 0$ such that for any $a \in [0,1]^{J}$,  if $\cD$ is a $3\widehat{K}\sigma$-sparse $(\frac{1}{20},  6)$-spanning $c$-family for $f_a$ with
$[\underline{r}(\cD), \overline{r}(\cD)] \subset (\rho_1, \frac{1}{2}\widetilde{h})$, and if a $C^r$ map $V\colon [0,1]^{J} \times \R^I \times X \to TX$ as above satisfies
\enmt
\item $\sigma\norm{\partial_b\partial_xV}_{X} + \norm{\partial_bV}_{X} < \widehat{C}_1$ and $R_{\pm}(f_a, (\cD, 3\widehat{K}\sigma), \mathrm{supp}_X(V)) > R_1$;
\item  for $B = (B_{\cC, i, t, j})_{\cC \in \cD, i, j \in \{1,\dots, c\}, t \in \cB_i} \in \R^{2n(\cD)c^2|\cA|}$ and $\sigma_a':= \widetilde{C}_{f_a}^{-\frac{1}{2}}\sigma$,
\begin{equation*}
V(a,\cdot)|_{(\cD, 2\widehat{K}\sigma)} = \sum_{\cC \in \cD, i, j \in \{1,\dots, c\}, t \in \cB_i} B_{\cC, i, t, j} V^{\sigma_a'}_{\cC, i, t, j},
\end{equation*}
\eenmt
then for any $a' \in B(a,\lambda_1)\cap [0,1]^{J}$, the $C^r$ map $\hat f \colon [0,1]^{J} \times U_1 \times X \to X$ given by \eqref{neuf trois} is $(\rho_1,\widetilde{h},\sigma,C, \kappa_1)$-Removable at $a'$. 
\end{lemma}
\begin{proof}
	Take $K_0$ as in \eqref{def K0 K1}. 
By compactness, we can choose   $\kappa> 0$  to be sufficiently small, depending only on $\{f_\omega\}_{\omega}$,  so that for  all $a \in [0,1]^{J}$,   $\kappa< \kappa'(f_{a})$ (see Lemma \ref{lem vector field and loop}); and choose $R_1>0$, $\kappa_1>0$, depending only on $\{f_\omega\}_{\omega}$, so that for any $a \in [0,1]^{J}$, $R_1>R_0(f_a,K_0,c,\widehat{C}_1,\frac{\kappa}{2})$,  $\kappa_1<\kappa_0(f_a,K_0,c,\widehat{C}_1,\frac {\kappa}{2})$ as in  Proposition \ref{determinant for smooth deformations}.

By hypothesis (1) and  Lemma \ref{lem vector field and loop}, we can choose $\lambda_1>0$ to be sufficiently small, depending only on $\{f_\omega\}_{\omega}$, $\rho_1,\sigma,R_1$, such that for any $a, \cD$ given in  the lemma, for any $a' \in B(a,\lambda_1) \cap [0,1]^{J}$, we have $f_{a'}\in \mathcal{U}'(f_a,\rho_1,\sigma)$ (see Lemma \ref{lem vector field and loop}) and $R_\pm(f_{a'},(\mathcal{D},2\widehat K\sigma),\mathrm{supp}_X(V))>R_1$. Then we apply  Lemma \ref{lem vector field and loop} to   $\cD$, and $(f_{a'},f_a)$ in place of $(g,f)$,  to obtain $\cD'$, a $(\frac{1}{10},8)$-spanning $c$-family for $f_{a'}$ with $[\underline{r}(\cD'), \overline{r}(\cD')] \subset (\rho_1, \widetilde{h})$ such that  the conclusion of Lemma \ref{lem vector field and loop} holds.  

For any $\cC' \in \cD'$, any $x \in \frac{1}{5}\cC'$, let $\gamma$ be a $(\sigma, \widetilde C_{f_a})$-regular (hence $(\sigma,C)$-regular) continuous family of $f_{a'}$-loops at $x$ satisfying the conclusion of Lemma \ref{lem vector field and loop}.
We claim that for any $s \in [0,1]$, the vector field $V(a,\cdot)$ in the lemma is adapted to $(\gamma(s), \sigma, \widehat{C}_1, R_1)$.
Indeed, by $\widehat{K} \geq K_{f_{a'}} + 4$ and $\cC' \subset (\cD, \sigma)$, we have $\cWc_{f_{a'}}(z, K_{f_{a'}}\sigma) \subset (\cD, 2\widehat{K}\sigma)$ for any $z \in \{x, x_1(s), x_2(s), x_3(s)\}$. Then by the choice of $\lambda_1$ and \eqref{v sigma 2}, we verify (2), (3) in Definition \ref{defiadaptedtosth} for $(\gamma(s),\sigma,\widehat{C}_1,R_1)$  in place of $(\gamma,\sigma,C,R_0)$. We verify (1) in Definition \ref{defiadaptedtosth} by \eqref{eq Vitj adapt C1} and the hypothesis on $V$. This verifies the claim.

For any $(i, \cB, \{s_{t}\}_{t \in \cB}) \in \Gamma$, any $(t,j) \in  \cB\times \{1,\dots,c\}$, we define $\gamma_{t,j}$ as in Definition \ref{def remov}.
Note that \eqref{determinant 1}, \eqref{determinant 2} are satisfied by  \eqref{v sigma 2}, thus for all $\sigma > 0$ sufficiently small (depending only on $\{f_\omega\}_\omega$), we can apply Proposition \ref{determinant for smooth deformations} to $(f_{a}, K_0, \widehat{C}_1, \frac{\kappa}{2}, V(a,\cdot), \{\gamma_{t,j}\}_{\substack{t \in \cB\\ 1 \leq j \leq c}})$ in place of $(f, L, C, \kappa, V, \{\gamma_{i,j}\}_{\substack{1\leq i \leq L\\ 1 \leq j \leq c}})$, which concludes.
\end{proof}

We now continue with the proof of Proposition \ref{lower bound determinant 2}.

Let $R_1, \kappa_1 > 0$ be given by Lemma \ref{lem construct for one c family}. Take any $\vartheta>0$. 
We set $K:=(20J)^J$. 
Then by applying Proposition \ref{prop slow recurrent} to $r,J,\{f_\omega\}_\omega, K, \vartheta$ and to $(R_1, \frac{1}{2}\widetilde{h})$ in place of $(R_0, h_0)$, we  obtain a set $\Omega_1$ compactly contained in $(0,1)^{J}$ and constants $N_0,\rho_0 \in (0, \frac{1}{2}\widetilde{h}),\rho_1 \in (0, \rho_0),\sigma_0,\lambda_0>0$ satisfying the conclusion of Proposition \ref{prop slow recurrent}. 
For any sufficiently small $\sigma > 0$ in  Proposition \ref{lower bound determinant 2}, we let $\lambda_1 = \lambda_1(\rho_1, \sigma)$ be taken as in Lemma \ref{lem construct for one c family}.

Let $T > 0$ be some large integer such that $\lambda:= \frac{100 J}{T} < \min(\lambda_0,  \lambda_1)$, and set $W_0:= (-\frac{1}{2T}, \frac{1}{2T})^{J}$. 
We choose points $\{a_1, \dots, a_{M_0}\} \subset \Omega_1$, $M_0\leq T^J$, such that    $W_i:=a_i + 2W_0$ is compactly contained in $[0,1]^{J}$,   the collection $\{W_i\}_{1 \leq i \leq M_0}$ forms an open cover of $\Omega_1$, and the cover multiplicity of $\{a_i+10 J W_0\}_{1 \leq i \leq M_0}$ is bounded by $K=(20J)^J$. 

Let $\Theta \colon \R^J \to [0,1]$ be a smooth function such that $\Theta |_{2 W_0} \equiv 1$ and $\mathrm{supp}(\Theta)\subset 3 W_0$. Let $\Theta_{i}:= \Theta(\cdot - a_i)$ for all $1 \leq i \leq M_0$, so that $\mathrm{supp} (\Theta_i) \subset a_i + 3W_0$. 

For each $1 \leq i \leq M_0$, we will inductively define a $(\frac{1}{20},  6)$-spanning $c$-family for $f_{a_i}$, denoted by $\cD_i$, in the following way.
Assume that for some $k \in \{1,\dots, M_0\}$, and for all $1 \leq i \leq k-1$, we have defined $\cD_i$ satisfying:
\begin{enumerate}
\item $\cD_i$ is a $\sigma_0$-sparse $(\frac{1}{20},6)$-spanning $c$-family for $f_{a_i}$;
\item $[\underline{r}(\cD_i), \overline{r}(\cD_i)] \subset (\rho_1, \rho_0)$;
\item $n(\cD_i) < N_0$.
\end{enumerate}
This assumption  is always true for $k=1$.

Let $\{ i_1, \dots, i_l \}$ be the set of all indices $p \in \{1,\dots, k-1\}$ such that $W_p \subset B(W_k,\frac{5c}{T})$. By the choice of $\{a_i\}$, we have  
$l < K$. Then we can apply Proposition \ref{prop slow recurrent} to obtain a spanning $c$-family for $f_{a_k}$, denoted by $\cD_k$, such that (1), (2), (3) above are true for $i=k$. Moreover, for any $1 \leq j \leq l$, $(\cD_{i_j}, \sigma_0)$ is disjoint from $(\cD_k, \sigma_0)$, and for all $a \in W_k \subset B(a_k, \lambda_0)$,  we have $R(f_a, (\cD_k, \sigma_0), (\{\cD_{i_j}\}_{j=1}^{l}, \sigma_0)) >R_1$ and $R_{\pm}(f_a, (\cD_k, \sigma_0)) > R_1$. 

Having constructed $\{\cD_i\}_{1 \leq i \leq M_0}$,  for each $1 \leq k \leq M_0$, we set $I_k:=n(\cD_k)c \sum_{i=1}^c|\cB_i|=2n(\mathcal{D}_k)c^2|\mathcal{A}|$, set  $\sigma_{a_k}':=\widetilde{C}_{f_{a_k}}^{-\frac{1}{2}}\sigma$, and let $V^{(k)} \colon \R^{I_k} \times X \to TX$ be the infinitesimal $C^r$ deformation defined as follows:
\begin{equation}\label{term 2030} 
V^{(k)}(B,\cdot):= \sum_{\cC \in \cD_k, i,j \in \{1,\dots, c\}, t \in \cB_i} B_{\cC, i, t, j} V^{\sigma_{a_k}'}_{\cC, i, t, j},\ \forall\, B = (B_{\cC,i,t,j}) \in \R^{I_k}.
\end{equation}
By Remark \ref{remaaboutthesupportofthevectorfields}, for all sufficiently small $\sigma > 0$, we have $\mathrm{supp}_X(V^{(k)}) \subset (\cD_k, \sigma_0)$.
Let 
 $I:= \sum_{k=1}^{M_0} I_k$.
For any $B = (B_k)_{k=1}^{M_0} \in \R^{I}$, where  $B_k \in \R^{I_k}$ for each $1 \leq k \leq M_0$, we define a $C^r$ map $V \colon [0,1]^{J} \times \R^{I} \times X \to TX$ as follows:
\begin{equation*}\label{def champ de vecteur a B x}
V(a,B,\cdot):= \sum_{k=1}^{M_0} \Theta_k(a) V^{(k)}(B_k,\cdot).
\end{equation*}
By definition, the map $V$ is linear in $B$. For each $a \in [0,1]^{J}$, let $\{i_1, \dots, i_{l}\}$ be the set of indices $p$ such that $\Theta_p(a) \neq 0$. Note that $l \leq K$. Moreover, by construction, we see that the sets  $(\cD_{i_j}, \sigma_0)$ are mutually disjoint for $j \in \{1,\dots,l \}$, and 
\begin{equation}\label{term supp v}
\mathrm{supp}_{X}(V(a,\cdot)) \subset \bigsqcup_{j=1}^l (\cD_{i_j}, \sigma_0).
\end{equation}
By the choice of $\widehat{C}_1$ above,  \eqref{eq Vitj adapt C1}, \eqref{term supp v}, and since $\mathcal{D}_k$ is $\sigma_0$-sparse for all $1 \leq k \leq M_0$, 
\begin{equation} \label{term lip v}
\sigma \norm{\partial_b\partial_xV}_{X} + \norm{\partial_bV}_{X} < \widehat{C}_1.
\end{equation}
Again, by the above construction, we see that
\begin{equation}\label{term rec supp}
R_{\pm}(f_a, (\{\cD_{i_j}\}_{j=1}^l,\sigma_0)) > R_1.
\end{equation}
Take any $k \in \{1,\dots,M_0\}$. By construction, $\mathcal{D}_k$ is a $\sigma_0$-sparse $\big(\frac{1}{20},6\big)$-spanning $c$-family for $f_{a_k}$, and for any $a \in W_k$, for each $B=(B_l)_{1\leq l\leq M_0} \in \R^I$, we see that
\begin{equation}\label{term v d}
V(a,B,\cdot)|_{(\cD_{k}, \sigma_0)}=  V^{(k)}(B_{k},\cdot)|_{(\cD_{k}, \sigma_0)}.
\end{equation}

We define $\hat f$ by \eqref{neuf trois} for $V$ given as above. It is clear that $\hat f$ is $C^r$, and for each $a \in [0,1]^{J}$, $\hat f(a,\cdot)\colon U_1 \times X \to X$ is the $C^r$ deformation at $(0,f_a)$ generated by $V(a,\cdot)$. By \eqref{neuf trois}, 
\eqref{term lip v},  and Lemma \ref{lem compare T V}(1), we obtain $\norm{\hat{f}}_{C^1} < Q$ for some $Q > 0$ depending only on $\{f_a\}_{a}$ and $\widehat{C}_1$, after possibly reducing the size of $U_1$.

By \eqref{term 2030}-\eqref{term v d}, for each $1 \leq k \leq M_0$, for any $a' \in W_k \subset B(a_k,\lambda_1)$,  
the assumptions of Lemma \ref{lem construct for one c family} are satisfied for all sufficiently small $\sigma>0$, hence   $\hat{f}$ is $(\rho_1, \widetilde{h}, \sigma,C,\kappa_1)$-Removable at $a'$. 
The set $\Omega_0:= \bigcup_{i=1}^{M_0} W_i$ is compactly contained in $[0,1]^J$. By our choices of $\{a_i\}$, we have  $\Omega_1 \subset \Omega_0 \subset [0,1]^{J}$, thus it is clear that $\mathrm{Leb}([0,1]^{J} \backslash \Omega_0)<\vartheta$. Then $\hat f$ is $(\rho_1, \widetilde{h}, \sigma,C, \kappa_1)$-Removable at $a$ for all $a \in \Omega_0$.   This concludes the proof.
\end{proof}

\subsection{Getting accessibility by perturbation}\label{Getting accessibility by perturbation}

In this subsection, we  fix a map $f_0 \in \cPH^r(X, \mathrm{Vol})$ 
which is dynamically coherent and center bunched.

Let $f\in \mathbb{U}(f_0) \cap \diff^{r}(X, \mathrm{Vol})$ ($\mathbb{U}(f_0)$ is defined in Notation \ref{choosingneighbourhood}). Let $\cC$ be a $c$-disk of $f$ with radius $h$ in $(0, \overline{h}_f)$, and set $\phi=\phi(\cC)$ (see Construction  \ref{given c disk get a chart}). Let $\hat{f} \colon U \times X \to X$ be a $C^r$ deformation at $(a,f)$, and set $T = T(\hat{f})$. Then for $C > 0$, $\sigma \in ( 0, \frac{\overline{\sigma}_f}{2})$, $x \in \frac 15 \cC$, and $\gamma$, a $(\sigma, C)$-regular continuous family of $f$-loops at $x$, let $\hat{\gamma}$ be given by \eqref{term 2040}, and let $\widehat{\psi}$ be given by \eqref{widehatpsidefi}. We let $\sigma > 0$ be sufficiently small, and by \eqref{relationwidehatpsipsifxgamma}, 
we have $\pi_X \widehat{\psi}(b, y, s) \in \phi((-h,h)^{d})$ for all $(b,y,s) \in \cWc_T((a,x),\overline{\delta}_{a,T}) \times [-1,2]^{c}$. Let $\Pi_c$ be as in Construction \ref{given c disk get a chart}(4); we define  
\begin{equation}\label{def hat Phi}
\Phi \colon \left\{
\begin{array}{rcl}
 \cWc_T((a,x),\overline{\delta}_{a,T}) \times [-1,2]^{c} &\to&  \R^{c}, \\
(b,y,s)&\mapsto& \Pi_c\phi^{-1}\pi_X\widehat{\psi}(b,y,s),
\end{array}
\right.
\end{equation}
where $\Pi_c \colon \R^{d} \simeq \R^{c}  \times \R^{d_u} \times \R^{d_s} \to \R^{c}$ denotes the canonical projection.

\begin{lemma}\label{remark theta holder}
Let $f, \hat{f}, \cC, \gamma, \hat{\psi}$ be as above. Then there is a $C^2$-uniform constant $\widehat C_0 = \widehat C_0(f)> 0$ such that, after possibly reducing the size of $U$, the following is true:
\enmt 
\item  for any $s \in [-1,2]^c$, the map $(b,y) \mapsto \Phi(b,y,s)$ is $C^1$, and $D\Phi(b,y,s)$ is uniformly continuous, uniformly bounded by $\widehat C_0$;
\item if  $c \geq 2$, there is a $C^2$-uniform constant $\theta_0 = \theta_0(f_0) \in \big(\frac{c-1}{c}, 1\big)$ such that for any $(b,y) \in \cWc_{T}((0,x), \overline{\delta}_{a,T})$, the map
 $s \mapsto \Phi(b,y,s)$ has $\theta_0$-H\"older norm less than $\widehat C_0$.
 \eenmt
\end{lemma}
We remark that Lemma \ref{remark theta holder}(2) is needed  to \lq\lq discretize\rq\rq \  Property $(\mathcal{P})$ when $c \geq 2$. 

\begin{proof}
Point (1) follows from the fact that $f$ is $C^2$, center bunched, and Lemma \ref{lemma T deformation}. Point (2) follows from Lemma \ref{thetatheta}.
\end{proof}

The main technical result of this section is the following. It provides estimates on the volume of ``bad'' parameters under some removability condition.

\begin{prop} \label{construct trans diffeo 2} 
Let $\{f_\omega\}_{\omega \in [0,1]^{J}}$ be a good $C^{r}-J-$family in $\mathbb{U}(f_0) \cap \diff^{r}(X, \mathrm{Vol})$.
For any $Q,C,  \kappa_1 > 0$, all sufficiently small $h  > 0$, for any $\rho_1 \in (0, h)$,  and for all sufficiently small $\sigma > 0$, the following is true. Assume that there exist  an open set $\Omega_0$ compactly supported in $(0,1)^{J}$;   and integer $I > 0$;  an open neighbourhood $U_1$ of the origin in $\R^{I}$; and a $C^{r}$ map $\hat{f} \colon [0,1]^{J} \times U_1 \times X \to X$, such that
	\begin{enumerate}[label=(\roman*)]
		\item $\hat f(a,b,\cdot)\in \diff^r(X,\mathrm{Vol})$ and  $\hat{f}(a,0,\cdot)=f_a$, for all $(a,b) \in [0,1]^{J}\times U_1$;
		\item $\norm{\hat{f}}_{C^1} < Q$; 
		\item $\hat f$ is $(\rho, h,\sigma, C,\kappa_1)$-Removable at $a$, for all $a \in \overline{\Omega_0}$. 
	\end{enumerate}


Then for any sufficiently small $\epsilon > 0$, any  $\delta > 0$, there exists a subset $\mathcal{E} = \mathcal{E}(\epsilon, \delta) \subset  \Omega_0 \times U_1$ of the parameter space such that
 $\mathrm{Leb}(\mathcal{E}) < \delta$, and  for any $(a,b) \in (\Omega_0 \times (U_1 \cap B(0,\epsilon))) \backslash \mathcal{E}$, $\hat{f}(a,b,\cdot)$ is $C^1$-stably accessible.

\end{prop}

\begin{proof}
	
We only detail the case where $c = 2$. We will sketch the adaptation needed for $c = 1$ at the end of the proof.

Let us assume for now that $c = 2$. Consequently, either \ref{condition ae} or \ref{condition be} holds.
In the following, we take $\theta = \theta_0(f_0)$ as in Lemma \ref{remark theta holder}, and set $K_0:=K_0(c,\theta)$ as in \eqref{def K0 K1}. By Lemma \ref{remark theta holder},  $\theta>\frac{c-1}{c}$ and thus $K_0\geq 2$. 

Let $\hat{f}, Q,C, \kappa_1, h, \rho_1, \sigma$ be as in the proposition. Let $T \colon [0,1]^{J} \times U_1 \times X \to [0,1]^{J} \times U_1 \times X$ be the $C^{r}$ map
$T\colon(a,b,x)\mapsto (a,b,\hat{f}(a,b,x))$.

For any $a \in \overline{\Omega_0}$, $\hat{f}$ can be regarded as a $C^r$ deformation at $((a,0), f_{a})$ with $J+I$ parameters. Let $\nu_{a,0}(x,\cdot) \colon \R^{J+I}  \to E^{su}_{f_a}(x)$ be the (unique) linear map given by Lemma \ref{lemma T deformation}.  
Set $T_a := T(\hat{f}(a,\cdot))$ and let $\nu^{a}_{0}(x,\cdot)\colon \R^{I} \to E^{su}_{f_a}(x)$ be the unique linear map satisfying \eqref{term 2050} for $b=0$.
It is direct to see that $\nu^{a}_0(x, B) = \nu_{a,0}(x, \{0\}^{J} \times B)$ for all $x \in X, B \in \R^{I}$. In the following we tacitly use the inclusion $\R^{I} \subset \{0\}^{J} \times \R^{I}$ and for any $B \in \R^{I}$, we abbreviate $\nu_{a,0}(x, \{0\}^{J} \times B)$ as $\nu_{a,0}(x, B)$. 

Fix $a \in \overline{\Omega_0}$. By the hypothesis of $(\rho_1, h, \sigma,C,\kappa_1)$-Removability, we can choose $\cD = \cD(a)$, a $(\frac{1}{10},8)$-spanning $c$-family for $f_a$, and for any $\cC \in \cD$, any $x \in \frac{1}{5}\cC$, we let $\gamma = \gamma(a,x)$ be a $(\sigma, C)$-regular continuous family of $f_a$-loops at $x$. 
By compactness, we can take a small constant $\overline{\delta} \in (0, \min(\overline{\delta}_{(a',0),T}, \overline{\delta}_{0, T_{a'}}))$  where $\overline{\delta}_{(a',0),T}, \overline{\delta}_{0,T_{a'}}$ are given by Definition \ref{lift of loop for deformation}. Let $\hat \gamma$ be the lift of $\gamma$ for $T$. Take $\sigma>0$ sufficiently small, so that the map  associated to $ \gamma$ and $T$ as in \eqref{def hat Phi}, denoted by 
 $\Phi\colon \cW^c_T((a,0,x), \overline{\delta}) \times [-1,2]^{c} \to \R^c$ is well-defined. For each $(i, \cB, \{s_t\}_{t \in \cB}) \in \Gamma$, set
\begin{equation}\label{def Psi}
\Psi  = \Psi_{a, \cC, x, i, \cB, \{s_t\}}\colon\left\{
\begin{array}{rcl}
\cW^c_T((a,0,x), \overline{\delta}) &\to& \R^{K_0c}, \\
(a', b',y)&\mapsto& (\Phi(a', b', y, s_t))_{t \in \cB}.
\end{array}
\right.
\end{equation}
By Lemma \ref{remark theta holder}(1), we can
differentiate $\Psi$ at $(a,0,x)$, 
and obtain for each $B \in \R^I$:
\begin{align*}
D\Psi((a, 0, x),B+\nu_{a,0}(x, B)) &= \big(D \Phi((a, 0, x, s_t),B + \nu_{a,0}(x,B))\big)_{t \in \cB} \\ 
&= \big(\Pi_cD\phi^{-1}\pi_XD \big(\prod_{j=1}^{c} H_{T, \hat{\gamma}_{t,j}}\big)(B + \nu_{a,0}(x, B))\big)_{t \in \cB},
\end{align*}
where $\gamma_{t,j}=\gamma(\varphi(j,s_{t,j}))$, and $\hat \gamma_{t,j}$ is the lift of $\gamma_{t,j}$ for $T$. Let $\widetilde{\gamma}_{t,j}$ be the lift of $\gamma_{t,j}$ for $T_a$. Then by definition and by Lemma \ref{prop holon map and nu}, we obtain 
\begin{equation*}\label{term 2060}
\pi_XD \big(\prod_{j=1}^{c} H_{T, \hat{\gamma}_{t,j}}\big)(B + \nu_{a,0}(x, B)) 
= \pi_cD \big(\prod_{j=1}^{c} H_{T_a, \widetilde{\gamma}_{t,j}}\big)(B + \nu^{a}_{0}(x, B)) + \nu^{a}_{0}(z, B), 
\end{equation*}
where we have set $z := \prod_{j=1}^{c} H_{f_a, \gamma_{t,j}}(x)$.

Let $\zeta > 0$ be a small constant to be determined. 
Let $h > 0$ be sufficiently small such that for any $a' \in [0,1]^{J}$, we have $h < \overline{h}_{f_{a'}, \zeta}$ (as in Construction \ref{given c disk get a chart}(5)). Then by Lemma \ref{lem property of nu}\eqref{lem a priori bound for nu},  Construction \ref{given c disk get a chart} and hypothesis $(ii)$, there exists a constant $D_1>0$ depending only on $\{f_{\omega}\}_{\omega\in [0,1]^{J}}$ such that  for any $B \in \R^{I}$,
\begin{equation}\label{term 2080}
\norm{\Pi_cD\phi^{-1}\nu^{a}_{0}(z, B)} \leq D_1\zeta Q \norm{B}.
\end{equation}

By $(\rho_1, h, \sigma, C, \kappa_1)$-Removability at $a$, there exists a linear subspace $H \subset \R^{I}$ of dimension $K_0c$ such that 
\begin{equation}
\Big|\det\big(H \ni B \mapsto \pi_cD\big(\prod_{j=1}^c H_{T_{a}, \widetilde{\gamma}_{t,j}}\big) (B + \nu^{a}_{0}(x, B))\big)_{t \in \cB}\Big| > \kappa_1.
\end{equation}
Then by \eqref{term 2080}, we can choose  $\zeta>0$  sufficient small, depending only on  $(Q,  \kappa_1, \{f_{\omega}\}_{{\omega} \in [0,1]^{J}})$, such that for some constant $D_2>0$ depending only on $\{f_{\omega}\}_{{\omega}\in [0,1]^{J}}$, it holds
\begin{equation*}
\big|\det(H \ni B \mapsto D\Psi((a, 0,x) ,B + \nu_{a,0}(x, B)) \in \R^{K_0c})\big| > D_2^{-1}\kappa_1.
\end{equation*}

Now, by Lemma \ref{remark theta holder} and the pre-compactness of $\Omega_0$,  $D\Psi$ is uniformly continuous, with norms uniformly bounded for all choices of  
$a, \cC, x, i, \cB, \{s_t\}$.  Then, by possibly reducing the size of $\overline{\delta}$ independently of the choices of $a, \cC, x, i, \cB, \{s_t\}$, we can assume that for any $(a', b',y) \in \cW^c_T((a,0,x), \overline{\delta})$,  there exists a subspace $H' \subset T_{a', b'}(\Omega_0 \times U_1)$ of dimension $K_0c$ such that 
\begin{equation}\label{lower bound for nearby param}
|\det(D\Psi(a,' b', y)|_{H'}) | > \frac{1}{2}D_2^{-1} \kappa_1.
\end{equation}

By compactness, for any $\cC \in \cD$, we can choose a finite set $\cal A(a, \cC) \subset \frac{1}{5}\cC$  s.t.
\begin{equation}\label{cV a cC}
\cV(a, \cC):= \bigcup_{x \in \cal A(a, \cC)} \cWc_T((a, 0, x), \overline{\delta})
\end{equation}
is an open neighbourhood of $\{a\}\times \{0\} \times \frac 15\cC$ in $\cWc_T$.

 By compactness, Lemma \ref{lemma stability spanning} and Remark \ref{remarque trois} (for $f=f_a,\mathcal{D},(k,\theta,\theta',\theta'',\rho_m,\rho_M)=(8,\frac{1}{10},\frac{1}{9},\frac{1}{8},\rho_1,h)$ and $\sigma < \frac 12 d(\mathcal{V}(a,\mathcal{C})^c,\{a\}\times \{0\}\times \frac 15 \mathcal{C})$), there exists a constant $\delta_0 > 0$ (independent of the choices of $a, \cC, x, i, \cB, \{s_t\}$) such that for any $(a',b') \in B(a, \delta_0) \times (U_1 \cap B(0,\delta_0))$, there exists  $\cD'$, a $(\frac{1}{9},10)$-spanning $c$-family for $\hat{f}(a',b',\cdot)$, such that  for any $\cC' \in \cD'$,  $\{a'\} \times \{b'\} \times \frac{1}{8}\cC' \subset \cal V(a, \cC)$ for some $\cC \in \cD$.
Without loss of generality, we assume that $U_1 \subset B(0,\delta_0)$ and  set 
$\cU(a):= B(a,\delta_0) \times U_1$.

Now,  by $\Omega_0\subset [0,1]^{J}$, we can find a finite set $\cal K \subset \Omega_0$ such that
\begin{equation}\label{U1 aK Ua}
\Omega_0 \times U_1 \subset \cup_{a \in \cal K} \, \cU(a).
\end{equation}

By \eqref{def K0 K1} we have  $\frac{\theta  (cK_0 -2c - 1)}{(c-1)K_0} > 1$. Let $\beta \in \big(0, \min(\frac{\theta  (cK_0 -2c -1)}{(c-1)K_0} - 1, 1) \big)$, 
so that $-(c-1)K_0\frac{1+ \beta}{\theta} + cK_0 - 2c > 1$. Then, choose $\eta > 0$ small enough such that
\begin{equation}\label{defeta}
\upsilon:=-K_0\frac{1+ \beta}{\theta}(c-1) + K_0 c - 2c - (K_0+1)\eta  -1>0.
\end{equation}
For any sufficiently small $\delta > 0$, for each $i \in \{1,\dots, c\}$, for any $t \in \cB_i$, let $\cN_{t,i}$ be a $\delta^{\frac{1 + \beta}{\theta}}$-net in $[-1,2]^{i-1} \times \{t\} \times [-1,2]^{c-i}$ such that $|\cN_{t,i}| < \delta^{-\frac{1+\beta}{\theta}(c-1) - \eta}$.

We denote by $\Sigma$ the diagonal of $(\R^{c})^{K_0}\simeq \R^{K_0c}$, that is,
\begin{equation}\label{definition diagonal}
\Sigma:= \{ (y,\dots,y) \in \R^{K_0c}\ \vert\ y \in \R^c \},
\end{equation}
and for any $\delta > 0$, we let $\Sigma_\delta$ be the $\delta$-neighbourhood of $\Sigma$ defined by
\begin{align*}
\Sigma_{\delta}:= \{ (y_i)_{1\leq i \leq K_0} \in (\R^{c})^{K_0}\ \vert\ \exists\, y \in \R^{c}, |y_i - y| < \delta,\ \forall\,  1\leq i\leq K_0 \}.
\end{align*}

For  any $a \in \cal K$, let $\cD=\cD(a)$ be the $(\frac{1}{10},8)$-spanning $c$-family for $f_a$ given above. For any $\cC \in \cD$, $x \in \cal A(a,\cC)$ and $(i, \cB, \{s_t\}_{t \in \cB}) \in \Gamma$,   set $\Psi:=\Psi_{a, \cC, x, i, \cB, \{s_t\}}$. By \eqref{lower bound for nearby param}, 
the map $D\Psi$ 
is a submersion from $\cWc_T((a, 0, x), \overline{\delta})$ to its image, and  is uniformly transverse to $\Sigma$, i.e., whenever $w=(a',b',y)\in \cWc_T((a, 0, x), \overline{\delta})\cap \Psi^{-1}(\Sigma)$, 
$$
T_{\Psi(w)} \Sigma+ D\Psi(T_w \cWc_T((a, 0, x), \overline{\delta}))\simeq \R^{K_0 c}.
$$
Therefore, $\Psi_{a, \cC, x, i, \cB, \{s_t\}}^{-1}(\Sigma)$ is a submanifold of $\cW^c_T((a,0,x), \overline{\delta})$ of dimension\footnote{See   \cite[Theorem 3.3]{Hirsch}; by transversality, $\Psi^{-1}(\Sigma)\subset \cWc_T((a, 0, x), \overline{\delta})$ has same codimension as $\Sigma$ in $\R^{K_0c}$. Here, transversality is w.r.t. the parameter $b'$, and by \eqref{lower bound for nearby param}, it is uniform in variables $a,x$, which gives a uniform bound in \eqref{estimate uniform transv npsi} on the number of balls needed to cover 
$\Psi^{-1}(\Sigma_\delta)$.} $J + I + 2c - K_0c$. Besides, 
by uniform transversality, there exists $\delta_1>0$ independent of the choice of $a, \cC, x, i, \cB, \{s_t\}$ such that for any $0<\delta<\delta_1$, we have
\begin{align}\label{estimate uniform transv npsi}
\mathscr{N}(\Psi_{a, \cC, x, i, \cB, \{s_t\}}^{-1}(\Sigma_{\delta}), \delta) < \delta^{K_0c-2c - I - J - \eta},
\end{align}
where for any set $\mathcal{S}$, 
$\mathscr{N}(\mathcal{S},\delta)$ is the minimal number of $\delta$-balls required to cover $\mathcal{S}$.
For any $0<\delta<\delta_1$, let $\mathcal{E}=\mathcal{E}(\delta)$ be the subset of ``bad'' parameters in 
$\Omega_0 \times U_1$:
\begin{equation}\label{def Omega}
\mathcal{E}:= \bigcup_{\substack{ a \in \cal K,\ \cC \in \cD, \\ x \in \cal A(a,\cC), \\ (i, \cal B, \{s_t\}_{t \in \cal B}) \in \Gamma\ s.t.\ \forall\, t \in \cB,\  s_t \in \cal N_{t,i} }} \pi_{[0,1]^{J} \times U_1}(\Psi_{a, \cC, x, i, \cB, \{s_t\}}^{-1}(\Sigma_{\delta})).
\end{equation}
Since in the above collection, only the last item, $\mathcal{N}_{t,i}$, depends on $\delta$, there exists a constant $D_3>0$ 
such that for any $0< \delta<\delta_1$,
\begin{equation} \label{eq NEdelta}
\mathscr{N}(\mathcal{E}, \delta) < D_3   \delta^{-K_0\frac{1+\beta}{\theta}(c-1) - K_0 \eta} \delta^{K_0 c - 2c - I  - J - \eta }=(D_3 \delta^\upsilon) \delta^{-I-J+1}.
\end{equation}
By  \eqref{defeta},  there exists $0< \delta_2 < \delta_1$ such that $D_3 \delta_2^\upsilon<1$. We deduce that $\mathrm{Leb}(\mathcal{E}) \leq (2\delta)^{J+I} \mathscr{N}(\mathcal{E}, \delta) <   \delta$ for  all $0<\delta<\delta_2$.

We claim that for all sufficiently small $\epsilon > 0$, for all sufficiently small $\delta > 0$, and any $(a,b) \in (\Omega_0 \times (U_1 \cap B(0,\epsilon))) \backslash \mathcal{E}$,  $\hat{f}(a, b, \cdot)$ is $C^1$-stably accessible. This would finish the proof for the case $c\geq2$. 

Indeed, by \eqref{U1 aK Ua}, for any $(a,b) \in (\Omega_0 \times (U_1 \cap B(0,\epsilon))) \backslash \mathcal{E}$, there exists $a_0 \in \cal K$ such that $(a,b) \in \cU(a_0)$. Let $\cD_0=\cD(a_0)$.  
Then by the definition of $\cU(a_0)$, there exists a $(\frac{1}{9},10)$-spanning $c$-family for $\hat{f}(a,b,\cdot)$, denoted by $\cD$, such that for each $\cC \in \cD$, there exists $\cC_0 \in \cD_0$ such that $\{a\} \times \{b\} \times \frac{1}{8}\cC \subset \cV(a_0, \cC_0)$.
Then for each $x \in \frac{1}{8}\cC$, by \eqref{cV a cC}, there exists $x_0 \in \cal A(a_0, \cC_0)$, such that $(a,b,x) \in \cWc_T((a_0, 0, x_0), \overline{\delta})$.

\begin{claim} 
	For any sufficiently small $\delta>0$, the following is true. 
For any $(i, \cB, \{s_t\}_{t \in \cB}) \in \Gamma$, take $\Psi  := \Psi_{a_0, \cC_0, x_0, i, \cB, \{s_t\}}$ as in \eqref{def Psi}. Then, $\Psi(a,b,x) \notin \Sigma$. 
\end{claim}
\begin{proof}
Indeed, for any $(i, \cB, \{s_t\}_{t \in \cB}) \in \Gamma$, there exists $\{w_t\}_{t \in \cB}$ such that for all $t \in \cB$, $w_t \in \cal N_{t,i}$ and $|s_t   - w_t | < \delta^{\frac{1 + \beta}{\theta}}$.
Since $(a,b) \notin \cal E$, and by \eqref{def Psi} and \eqref{def Omega}, there exist $t,t' \in \cB$ such that $| \Phi(a, b, x, w_{t}) - \Phi(a, b, x, w_{t'})| > \delta$, where $\Phi$ is defined as in \eqref{def hat Phi} for $((a_0, 0), x_0)$ in place of $(a,x)$. 
By Lemma \ref{remark theta holder}, we get 
$|\Phi(a,b,x, w_{t}) - \Phi(a,b,x, s_{t})|$, 
$|\Phi(a,b,x, w_{t'}) - \Phi(a,b,x, s_{t'})|< D_4 \delta^{1+\beta}$
for some constant $D_4>0$ independent of $\delta$.
The claim follows, since for sufficiently small $\delta>0$, we then have 
$
| \Phi(a, b, x, s_{t}) - \Phi(a, b, x, s_{t'})| > \delta - 2D_4\delta^{1+\beta} > 0$.
\end{proof}

Let $\gamma_0=\gamma(a_0,x_0)$ be the $(\sigma, C)$-regular continuous family of $f_{a_0}$-loops at $x_0$ associated to $(a_0,x_0)$. Since $\overline{\delta} < \overline{\delta}_{(a,0), T} < \delta_{(a,0), T}$ (see Definition \ref{lift of loop for deformation}) and $(a,b,x) \in \cWc_{T}((a_0,0,x_0),\overline{\delta})$,  let  $\gamma_{a,b,x}$ be the continuous family of $\hat{f}(a,b,\cdot)$-loops at $x$ associated to $\gamma_0$ according to Definition \ref{lift of loop for deformation}.   
By 
\eqref{widehatpsidefi}, \eqref{relationwidehatpsipsifxgamma}, \eqref{def hat Phi} and \eqref{def Psi}, we see that $\hat{f}(a,b,\cdot)$ satisfies $(\mathcal{P})$ (for $(\theta,\theta',k)=(\frac 19,\frac 18,10)$, $\mathcal{D}$,   $\{\gamma_{a,b,x}\}_{x \in \frac{1}{8}\cC,\ \cC \in \cD}$ and $\psi = \psi(\hat{f}(a,b,\cdot), x, \gamma_{a,b,x})$ as in \eqref{psifxgammadefi}. Thus,
our claim follows from Proposition \ref{lemmappointlooptostablevalue}. 

Now we consider the case where $c=1$.  We can just choose $K_0=K_1=5$ in \eqref{def K0 K1}.  It suffices to notice that for each $i \in \{1,\dots, c\}$, for each $t \in \cB_i$, the set $[-1,2]^{i-1} \times \{t\} \times [-1,2]^{c-i}$ is reduced to the singleton $\{t \}$. Thus we can choose in the above proof that for any $\delta > 0$,  $\cN_{t,i} = \{t\}$. It is straightforward to verify that the above proof for $c \geq 2$  (below \eqref{defeta})  carries over to the case $c=1$. For instance, \eqref{eq NEdelta} becomes
\aryst
\mathscr{N}(\mathcal{E}, \delta) < D_3    \delta^{K_0 c - 2c - I  - J - \eta }=(D_3 \delta^{\upsilon'}) \delta^{-I-J+1}
\earyst
where 
\aryst
\upsilon'=  K_0 c - 2c - \eta  -1>0.
\earyst
This finishes the proof.
\end{proof}

\section{The proof of Theorem \ref{main thm 2}}
Combining Propositions \ref{lower bound determinant 2} and \ref{construct trans diffeo 2}, we are ready to prove Theorem \ref{main thm 2}.

\begin{proof}[Proof of Theorem \ref{main thm 2}]
We only detail the volume preserving case, for the proof of the other one follows the same line  after replacing $\diff^{r}(X,\mathrm{Vol})$ by $\diff^{r}(X)$.

By Notation \ref{notation 1}, for any map $f$ in Theorem \ref{main thm 2}, $f$ is dynamically coherent, center bunched, 
and satisfies   \ref{condition ae} or \ref{condition be}. Set $\cU:= \mathbb{U}(f) \cap \diff^{r}(X,\mathrm{Vol})$ (see Notation \ref{choosingneighbourhood}), and
let $\{f_a\}_{a \in [0,1]^{J}}$ be a good $C^{r}-J-$family of diffeomorphisms in $\cU$.  By Proposition \ref{lower bound determinant 2} and Proposition \ref{construct trans diffeo 2}, for any $\vartheta > 0$, there exist an open set $\Omega_0 \subset [0,1]^{J}$ with $\mathrm{Leb}([0,1]^{J} \backslash \Omega_0) < \vartheta$, an open neighbourhood $U_1$  of the origin in $\R^{I}$, and a $C^r$ map $\hat{f} \colon [0,1]^{J} \times U_1 \times X \to X$ with $f_a=\hat f(a,0,\cdot)$ for all $a \in [0,1]^{J}$, 
such that 
for all sufficiently small $\epsilon > 0$, and any  $\delta > 0$, there exists $\mathcal{E}  \subset \Omega_0 \times U_1$ such that  $\mathrm{Leb}(\mathcal{E}) < \delta$ and $\hat{f}(a,b,\cdot)$ is  $C^1$-stably accessible for all $(a,b) \in (\Omega_0 \times (U_1 \cap B(0,\epsilon))) \backslash \mathcal{E}$. Now, given any sufficiently small $\epsilon > 0$, let $\delta \in (0, \epsilon^{I}\vartheta)$ and $\mathcal{E}=\mathcal{E}(\epsilon,\delta)$ be as above, and for any $b \in U_1$, 
set 
$\mathcal{E}^b:=\mathcal{E}\cap (\Omega_0 \times \{b\})$. Then, by Fubini's theorem, 
there exists $b \in U_1 \cap B(0,\epsilon)$ such that
\begin{equation}\label{bad parameters estimate size}
\mathrm{Leb}\big((([0,1]^{J}\backslash \Omega_0) \times \{b\}) \cup\mathcal{E}^b\big) \lesssim \mathrm{Leb}([0,1]^{J} \backslash \Omega_0) + \epsilon^{-I}\mathrm{Leb}(\mathcal{E}) \lesssim \vartheta.
\end{equation}
For any  integer $n \geq 1$, we consider the following collection of $C^{r}-J-$families in $\mathcal{U}$: 
\begin{gather*}
\cF_{n}:= \{ \{f_a\}_{a \in [0,1]^{J}} \in C^{r}([0,1]^{J}, \mathcal{U})\ \vert\ \mbox{the set of $a \in [0,1]^{J}$ such that  }\\  \mbox{ $f_a$ is not $C^1$-stably accessible has measure less than $\frac{1}{n}$}\}.
\end{gather*}
It follows from $J \geq J_0$,  Proposition \ref{propapproximationbyregularfamily} and \eqref{bad parameters estimate size} above that for any $n \geq 1$, the set $ \cF_{n}$ is $C^1$-open and $C^{r}$-dense in the space $C^{r}([0,1]^{J}, \cU)$. In particular,  $\cG :=  \bigcap_{n \geq 1}  \cF_{n}$
is residual. By definition, for any $\{f_a\}_{a \in [0,1]^{J}} \in \cG$, the set of $a\in [0,1]^{J}$ such that $f_a$ is not $C^1$-stably accessible has  measure zero. This concludes the proof.
\end{proof}

\appendix

\section{} \label{app a}

\begin{proof}[Proof of Lemma \ref{prop holon map and nu}]
	
	We detail the case where $*=u$. The other one is handled similarly. 
	
	Proof of \eqref{lemma first coordinate}: for any $b \in U$, $x \in X$, we have $\cWu_{T}(b,x) \subset \{b\} \times X$. Hence the image of $H^{u}_{T, (0,x), (0,y)}$ is contained in $\{b\} \times X$.
	Then for any $B \in T_0U$, we have $DH^{u}_{T, (0,x),(0,y)}(B + \nu_0 (x,B)) \in B + T_zX$, while $\pi_b(B + T_zX) = B + \nu_0(z,B)$. 
	
	To show \eqref{lemma cent coordinate}, we need the following lemma.
	\begin{lemma}[\textit{A priori} estimates]\label{prop a priori est}
		There exists a $C^1$-uniform constant $\tilde \sigma_f>0$, and a $C^2$-uniform constant  $C_\star=C_\star(f) > 0$  
		such that the following is true. Take any $x \in X$, $w\in \cWcu_f(x,\tilde \sigma_f/2)$. Then $z:= H^{u}_{f, x, w}(x)$ is well-defined, and for any $B \in T_0U$, we have
		\begin{align} 
		&\norm{\pi_cDH^{u}_{T, (0,x), (0,w)}(B + \nu_0(x,B))}  \nonumber \\
		\leq \ & C_\star (\norm{\nu_0(x,\cdot)} + \norm{\nu_0(z,\cdot)}  +  \norm{D^2T} d_{\cWu_{f}}(x,z)) \|B\|. \label{lab c T}
		\end{align}
		We have an analogous statement for any $x,w$ in a local center stable leaf.
	\end{lemma}
	
	\begin{proof}
	
	We follow the construction in \cite[Proof of Theorem 4.1]{PSW} and refer to \cite{PSW} for many details. 
	
	Let $\widetilde{E}^u$ (resp. $\widetilde{E}^{cs}$) be a smooth bundle of $TX$ that closely approximates $E^{u}_f$ (resp. $E^{cs}_f$)  and let $\delta = \delta(f)>0$ be a small $C^1$-uniform constant, to be chosen in due course. Following the proof of \cite[Theorem 4.1, page 530]{PSW}, we embed  $\widetilde{E}^u$, $\widetilde{E}^{cs}$ via $C^\infty$ maps  $\imath_1\colon\widetilde{E}^{u}\to   X\times \R^{m_1}$ and $\imath_2\colon\widetilde{E}^{cs}\to  X\times \R^{m_2}$, where for $i=1,2$, $m_i \in \mathbb{N}$, and $\R^{m_i}$ is equipped with a metric $\|\cdot\|_i$ such 
	that the Lipschitz constant of $\imath_i$ is uniformly bounded by some $C^1$-uniform constant $c_0=c_0(f)>0$. 
	As in \cite[Proof of Theorem 4.1]{PSW}, 
	we can choose $\widetilde{E}^u,\widetilde{E}^{cs},\imath_1,\imath_2,\|\cdot\|_1,\|\cdot\|_2,U,\delta$ so that
	there is a $C^1$ bundle contraction $F_\#$ satisfying
	\begin{center}
		\begin{tikzcd}
		U \times X \times Y \arrow{r}{F_\#} \arrow[swap]{d}{} & U \times X \times Y \arrow{d}{} \\
		U \times X \arrow{r}{T}& U \times X
		\end{tikzcd}
	\end{center}
	Here $Y$ is defined as
	\begin{equation*}
	Y := \{ g \in C^0(\R^{m_1}(2c_0\delta),\R^{m_2}) \mid g(0) = 0,\, \mathrm{Lip}(g) \leq 1 \}
	\end{equation*}
	equipped with the norm
	\begin{equation*}
	\norm{g} :=  \sup_{x}\frac{\norm{g(x)}_2}{\norm{x}_1}, \quad \forall\, g \in Y.
	\end{equation*}
	Recall that by \cite[(11)]{PSW}, we have
	\begin{equation*}
	\norm{(F_\#)_pg - (F_\#)_ph} \leq C e^{-\bar\chi^u(p) + \hat \chi^c(p)} \norm{g - h}, \quad \forall p \in X.
	\end{equation*}
	Moreover, we have:
	\enmt
	\item
	the unique invariant section of $F_\#$  is  a family of Lipschitz functions $\{\gamma_p \colon \mathbb{R}^{m_1}(2c_0 \delta) \to \mathbb{R}^{m_2} \}_{p \in U\times X}$ parametrizing local unstable manifolds: 
	for any $p = (b,x) \in U \times X$, we have
	\begin{gather*}
	\gamma_p(\imath_1(\widetilde{E}^u(x, 2 c_0 \delta))) \subset \imath_2(\widetilde{E}^{cs}(x)), \\ \text{and}\
	\cWu_{\hat f(b,\cdot)}(x,\delta/c')\subset  \exp_x(Graph(\imath_2^{-1}\gamma_p\imath_1|_{\widetilde{E}^{u}(x,\delta)})) \subset \cWu_{\hat f(b,\cdot)}(x,c'\delta),
	\end{gather*}
	for some $C^1$-uniform constant $c'=c'(f)>0$;
	\item $F_\#$ preserves the sub-bundle $Y_0$ where
	\begin{equation*}
	Y_0(p) := \{ g \in Y \mid 
	g(\imath_1(\widetilde{E}^u(x, 2 c_0 \delta))) \subset \imath_2(\widetilde{E}^{cs}(x)) \}, \quad \forall\, p = (b,x) \in U \times X,
	\end{equation*}
	and satisfies that 	for some  $C^2$-uniform constant $c_*=c_*(f) > 0$,  for any sufficiently close  points $p,q$ in the same center-unstable leaf of $T$, for any $g \in Y$, for any $z \in \mathbb{R}^{m_1}(2c_0\delta)$,\footnote{See \cite[Page 533]{PSW}.}
	\begin{equation}\label{eq contraction jet} 
	\|(F_{\#})_p(g)(z) - (F_{\#})_q(g)(z)\|_2 \leq c_* \norm{D^2T} d(p,q)\|z\|_1.
	\end{equation}
	\eenmt
	
	\begin{lemma}
		For any sufficiently close $p,q \in U \times X$  in a local center-unstable manifold of $T$,
		\begin{equation} \label{gamma p 1 z}
		\|\gamma_{p}(z) - \gamma_{q}(z)\|_2 \leq c_* \norm{D^2T} d(p,q)\|z\|_1,\quad \forall\, z \in B(0,2c_0\delta)\subset\mathbb{R}^{m_1}.
		\end{equation}
	\end{lemma}
	\begin{proof}
		The proof follows \cite[Proof of Theorem 3.2]{HPS}.
		For each $p \in U \times X$, we let $\widehat \cJ_p$ be the set of $C^0$ sections $g\colon V_{p,g} \to Y$ -- where $V_{p,g}$ is an open neighbourhood of $p$ -- such that
		\begin{equation*}
		g(p) = \gamma_p, \quad
		\limsup_{\cWcu_T(p) \ni  q \to p} \frac{ \norm{g(q) - g(p)} }{d(q,p)} < \infty.
		\end{equation*}
		Define a semi-norm on $\widehat \cJ_p$ by
		\begin{equation*}
		d(g, h) := 
		\limsup_{\cWcu_T(p) \ni  q \to p} \frac{ \norm{g(q) - h(q)} }{d(q,p)},\qquad \forall\, g,h \in \widehat \cJ_p.
		\end{equation*}
		We declare that $g, h \in \widehat \cJ_p$ are equivalent if $d(g,h) = 0$.
		Then the semi-norm $d$ descends to a distance on the equivalence classes of $\widehat \cJ_p$, denoted by $\cJ_p$. Notice that $\cJ_p$ is the space of {\it Lipschitz jets}, defined in \cite[Chapter 3]{HPS},  at $p$.
		
		We define a Banach bundle $\cJ$ over $U \times X$ (equipped with the discrete topology) by setting 
		\begin{equation*}
		\cJ := \bigcup_{p \in U \times X}\cJ_p.
		\end{equation*}
		The map $F_\#$ gives rise to a bundle map $J$ of $\cJ$. More precisely, for any $g\colon V_{T^{-1}(p),g} \to Y$ in $\widehat \cJ_{T^{-1}(p)}$, we define $Jg \in \widehat \cJ_p$ by
		\begin{equation*}
		Jg(q) := (F_\#)_{T^{-1}(q)}(g(T^{-1}(q))), \quad \forall\, q \in U \times X \cap T(V_{T^{-1}(p), g}). 
		\end{equation*} 
		It is clear that $J$ descends to a map from  $\cJ_{T^{-1}(p)}$ to $\cJ_p$. Consequently, $J$ is a bundle map of $\cJ$ over $T$.
		For any $g, h \in \widehat \cJ_{T^{-1}(p)}$, we have
		\begin{align*}
		{d}(Jg,Jh) &= \limsup_{\cWcu_T(p) \ni q \to p} \frac{\norm{Jg(q) - Jh(q)}}{d(p,q)}  \\
		&=\limsup_{\cWcu_T(p) \ni q \to p} \frac{\norm{ (F_\#)_{T^{-1}(q)}(g(T^{-1}(q))) -  (F_\#)_{T^{-1}(q)}(h(T^{-1}(q)))}}{d(p,q)} \\
		&=\limsup_{\cWcu_T(T^{-1}(p)) \ni  u \to T^{-1}(p)} \frac{\norm{ (F_\#)_{u}(g(u)) -  (F_\#)_{u}(h(u))}}{\norm{g(u)-h(u)}}   \\
		& \quad \cdot \limsup_{\cWcu_T(p) \ni  q \to p} \frac{d(T^{-1}(p), T^{-1}(q))}{d(p,q)} \cdot \limsup_{\cWcu_T(p) \ni u \to T^{-1}(p)} \frac{\norm{g(u)-h(u)}}{d(T^{-1}(p),u)}  \\
		&\leq \hat c e^{-\bar\chi^u(T^{-1}(p)) + \hat \chi^c(T^{-1}(p)) -\bar\chi^c(T^{-1}(p))} {d}(g,h),
		\end{align*}
		for some constant $\hat c>0$, as the points $p,q$ above are in the same center-unstable manifold.  Moreover, by center bunching, $-\bar\chi^u + \hat \chi^c -\bar\chi^c<0$, thus $\mathcal{J}$ has a unique $J$-invariant section $h_0$. On the other hand, as $\gamma$ is $F_\#$-invariant, it is clear that $\gamma$ must be a representative of $h_0$.
		
		We say that $\bar g \in \cJ_p$ is the constant section if $\bar g$ admits a representative  $g \in \widehat \cJ_p$ such that
		\begin{equation*}
		g(q) = \gamma_p, \quad \forall\, q \in V_{p, g}.
		\end{equation*} 
		We denote by $h_*\colon U \times X \to \cJ$ the section where $h(p)$ is the constant section in $\cJ_p$.
		By \eqref{eq contraction jet}, the space of sections $h\colon U \times X \to \cJ$ that is within distance $c_* \norm{D^2T}$ to $h_*$ is invariant under $J$ if $c_*$ is sufficiently large. Consequently, $d(h_0, h_*) \leq c_* \norm{D^2T}$.
	\end{proof}
	
	Let  $x \in X$ and $p = (0,x)$.  In a small neighbourhood of $x$, we can define a $C^\infty$  coordinate chart $\tau_p\colon (-1,1)^{ d} \to X$ with $\tau_p(0)=x$ and the following properties:
	\enmt
	
	\item $D\tau_p(0,\cdot)$ maps $\R^{d_u} \times \{0\}$ (resp. $\{0\} \times \R^{c + d_s }$) to $E^{u}_{f}(x)$ (resp. $E^{cs}_{f}(x)$);
	
	\item $\widetilde{E}^{u}$ is close to the tangent space of $\tau_p((-1,1)^{d_u} \times \{z_{cs}\})$, $\forall\, z_{cs} \in (-1,1)^{c+d_s}$;
	
	\item $\widetilde{E}^{cs}$ is close to the tangent space of $\tau_p(\{z_u\} \times (-1,1)^{c + d_s})$, $\forall\, z_u \in (-1,1)^{d_u}$.
	
	\eenmt 
	The required closeness  in (2),(3) will be made evident from the proof. 
	
	Such chart $\tau_p$ is obtained by first choosing a $C^\infty$ chart $\tilde \tau_p$ satisfying $\tilde \tau_p(0)=x$ and (1), and then considering the restriction of $\tilde \tau_p$ to a sufficiently small neighbourhood of $0$. We can also choose $\tau_p$ to depend continuously on $p$, with $\|\tau_p\|_{C^1}$, $\|\tau_p^{-1}\|_{C^1}$ bounded by a $C^1$-uniform constant $c''=c''(f)>0$. 
	
	
	In the following, we fix $p=(0,x)$ and denote   
	$\tau=\tau_p$.
	We will not distinguish a point $z\in \tau((-1,1)^d)$ and its coordinate under $\tau^{-1}$ e.g. we denote $p=(0,0)$. Besides, we identify a tangent vector  $v \in T_z X$ with its  preimage $D\tau^{-1}(v)$, whenever it is defined. Without loss of generality, we assume that $\delta$ is small compared to the size of 
	$\tau((-1,1)^d)$.

	Let $p' = (0,x') \in \cWu_{T}(p)$ be a point sufficiently close to $p$ such that there exists $w_0 \in \imath_1 (\widetilde{E}^{u}(p, \delta/2)) \subset \R^{m_1}$ satisfying
	\begin{equation} \label{def x'}
	x' = \exp_{x}(\imath_1^{-1}(w_0) +  \imath_2^{-1}\gamma_p(w_0)).
	\end{equation}
	
	Fix any $B \in T_0 U$. Recall that $B+\nu_0(x,B)\in E_T^c(p)$. Then let $t > 0$ be any sufficiently small constant, and define $q=q(t)\in \cW_T^c(p,\delta)$ by 
	\begin{equation} \label{y = t nu b}
	q:=(tB,y)=p+t(B,\nu_0(x,B))+o(t),\qquad y:=t \nu_0(x,B)+o(t),
	\end{equation}
	where $o(t)$ denotes a vector of modulus sublinear in $t$. Here \eqref{def x'} is interpreted as follows: $D\tau^{-1}(t(B,\nu_0(x,B)))$ is a vector in $T_0(-1,1)^d$, and adds up to $\tau^{-1}(p)$ using the isomorphism $T_0(-1,1)^d\simeq \R^d$. Several other expressions in this proof shall be understood in the same way. Since $\widetilde{E}^u,\widetilde{E}^{cs}$ are $C^1$ embedded into $X \times \R^{m_1}$, $X \times \R^{m_2}$ respectively, then for sufficiently small $t$, there exists $w_2 \in \imath_1(\widetilde{E}^u(y,\delta))$ such that
	$$
	\|w_2-w_0\|_2 < c_1(\|y\|+t\|B\|)\|w_0\|_1
	$$
	for some $C^2$-uniform constant $c_1=c_1(f)>0$.
	
	We define a point in $\cW_{\hat f(t b,\cdot)}^u(y,c'\delta)$ by
	$$
	y'':=\exp_y (\imath_1^{-1}(w_2)+\imath_2^{-1}\gamma_q (w_2)).
	$$
	Recall that $TX=\widetilde{E}^{u}\oplus  \widetilde{E}^{cs}$ is  $C^{\infty}$  embedded into $X \times \R^{m_1+m_2}$, and $\exp\colon TX \to X$ is a $C^\infty$ map. Then, by \eqref{gamma p 1 z}, \eqref{y = t nu b}, and since $\mathrm{Lip}(\imath_1),\mathrm{Lip}(\imath_2)\leq c_0$, we deduce that
	\begin{align}\label{dist x' y''}
	\|x' - y''\| &<  c_2 (\|y\| +\|w_0-w_2\|_1+ \|\gamma_p(w_0) - \gamma_q(w_0) \|_2)\nonumber \\
	&< c_3 (t\|\nu_0(x,B)\| + t \norm{D^2T} (\|B\|+\|\nu_0(x,B)\|)  \|w_0\|_1),
	\end{align}
	for  $C^2$-uniform constants $c_2=c_2(f),c_3=c_3(f)>0$.
	
	Let $q':= H^{u}_{T, p, p'}(q)$. By definition, $\{q'\}= \cWcs_T(p') \cap \cWu_T(q)$. Since $\cWu_T(q)=\{tB\}\times \cW_{\hat f(tB,\cdot)}^{u}(y)$, we have $q' = (tB, y')$ for some $y'\in \cW_{\hat f(tB,\cdot)}^{u}(y,2c'\delta)$. 
	On $U \times \tau((-1,1)^d)$, $\cWcs_{T}(p')$ is closely approximated by $E^{cs}_{T}(p') =  Graph(\nu_0(x',\cdot)) \oplus E^{cs}_{f}(x')$. 
	Thus $y'=x'+ t\nu_0(x',B)+ v^{cs}(t)+o (t)$ for some $v^{cs}(t) \in E^{cs}_{f}(x')$. Hence
	\begin{equation}\label{equation y' y''}
	(y'-y'')-v^{cs}(t)= t\nu_0(x',B)+x'-y''+o (t). 
	\end{equation}
	Since $y',y''\in \cW_{\hat f(tB,\cdot)}^{u}(y)$, we also know that $y'-y''$ is close to $E^{u}_{\hat{f}(t B,\cdot)}(y'')$. By conditions (2),(3) in the choice of $\tau_p$,  the angle between $E^{u}_{\hat{f}(tb,\cdot)}(y'')$ and $E^{cs}_{f}(x')$ is uniformly bounded from below. 
	By \eqref{equation y' y''},  
	for some $C^2$-uniform constant $c_4=c_4(f)>0$, it holds  
	\begin{equation}\label{equation y' y'' 2}
	\|v^{cs}(t)\|,\|y'-y''\|\leq c_4 ( t\|\nu_0(x',B)\| +\|x'-y''\|).
	\end{equation}
	
	Combining estimates \eqref{dist x' y''} and \eqref{equation y' y'' 2}, and since $\|\pi_X p' - \pi_X q'\|\leq \|x'-y''\|+\|y'-y''\|$, we deduce that for some $C^2$-uniform constant 
	$c_5=c_5(f) > 0$, 
	\begin{equation*}
	\|\pi_X p' - \pi_X q'\| 
	\leq c_5 (\|\nu_0(x,\cdot)\|+\|\nu_0(x',\cdot)\| +\norm{D^2T} d_{\cW_{\hat f(b,\cdot)}^u}(x,x'))  t\|B\|.
	\end{equation*}
	We then conclude our proof by noting that 
	\begin{equation*}
	\|\pi_cDH^{u}_{T, p, p'}(B + \nu_0(x,B))\| = \lim\limits_{t \to 0} \frac{\|\pi_X p'-\pi_Xq'\|}{t}.
	\end{equation*}
	\end{proof}
	We now continue with the proof of \eqref{lemma cent coordinate}.  
	Without loss of generality, we may assume that $\sigma_f \in (0, \tilde\sigma_f)$.
	By Lemma \ref{prop a priori est} and Lemma \ref{lem property of nu}\eqref{lem apriori-refined est for nu}, there exist $C^2$, resp. $C^1$-uniform constants $c_2=c_2(f) > 0$, resp. $c_1=c_1(f) > 0$, such that  (recall \eqref{first line})
	\begin{gather*}
	\norm{\pi_cDH^u_{T,(0,x),(0,y)}(B + \nu_0(x,B))}\leq c_2 (\norm{\nu_0(x,\cdot)} + \norm{\nu_0(z,\cdot)} + \norm{D^2T} d_{\cWu_f}(x,z))\norm{B} \\
	\leq c_2c_1 (\max(e^{- \kappa(\hat f,x)}, e^{- \kappa(\hat f,z)}) \norm{T}_{C^1} + \norm{D^2T}d_{\cWu_f}(x,z))\norm{B}.
	\end{gather*}
\end{proof}

\section{} \label{app b}
\begin{proof}[Proof of Lemma \ref{lem compare T V}]
Let $V\colon \mathbb{R}^I \times X \to TX$ be a $C^r$ vector field as in Definition \ref{def infini deform}, and let $U \subset \mathbb{R}^I$ be a small neighbourhood of the origin. We let $\mathscr{F}\colon  \R \times U \times X \to X$ be the associate flow; it is defined by the following equation:
\begin{equation} \label{app eq 1}
\partial_t\mathscr{F}(t,b,x) = V(b, \mathscr{F}(t,b,x)),
\end{equation}
with initial condition $\mathscr{F}(0, b,x) = f(x)$.
For any $(s,b,x) \in \R \times U \times X$, we have
\begin{equation*}
\begin{array}{ll}
\bullet\ \partial_b \mathscr{F}(0,b,x) =0, &\bullet\ \partial_b^2 \mathscr{F}(0,b,x) = \partial_{b}\partial_x \mathscr{F}(0,b,x) = 0, \\
\bullet\ \partial_x \mathscr{F}(0,b,x) = \partial_x f(x),  &\bullet\ \partial_{x}^2\mathscr{F}(0,b,x) = \partial_x^2f(x), \\
\bullet\ \hat{f}(b,x) = \mathscr{F}(1,b,x),  &\bullet\ f(x) = \mathscr{F}(s,0,x).
\end{array}
\end{equation*}
By differentiating \eqref{app eq 1}, we obtain the following equations:
\begin{align*}
\partial_t \partial_x \mathscr{F} &= \partial_xV \partial_x\mathscr{F}, \quad
\partial_t \partial_b \mathscr{F} = \partial_bV + \partial_xV \partial_b \mathscr{F}, \\
\partial_t \partial_x^2 \mathscr{F} &= \partial_x^2V(\partial_x\mathscr{F}, \partial_x\mathscr{F}) + \partial_xV \partial_x^2\mathscr{F}, \\
\partial_t \partial_b\partial_x \mathscr{F} &= \partial_b\partial_xV (\partial_b,\partial_x\mathscr{F}) + \partial_x^2V (\partial_x \mathscr{F}, \partial_b\mathscr{F})+ \partial_xV \partial_b\partial_x\mathscr{F}, \\
\partial_t \partial_b^2 \mathscr{F} &= \partial_b^2V + 2 \partial_b\partial_xV(\partial_b, \partial_b\mathscr{F}) +  \partial_x^2V (\partial_b\mathscr{F}, \partial_b\mathscr{F}) + \partial_xV \partial_b^2\mathscr{F}.
\end{align*}

In particular, for all $t\in (0,1)$, $x \in X$ and $B \in T_0 U$, we have 
$$
\partial_t \partial_b \mathscr{F}((t,0,x),B)=\partial_b V((0,\mathscr{F}(t,0,x)),B)+\partial_x V((0,\mathscr{F}(t,0,x)),\partial_b \mathscr{F}((t,0,x),B)).
$$
In the above equality, the first term on the RHS equals $V(B,\mathscr{F}(t,0,x))=V(B,f(x))$; and the second term on the RHS equals $0$. Thus
$$
\pi_X DT((0,x),B)=\partial_b \hat f((0,x),B)=\partial_b \mathscr{F}((1,0,x),B)=V(B,f(x)).
$$
This concludes the proof of (2). 

By a slight abuse of notations, we use $\norm{\cdot}$ to denote the uniform norm for:

(a) derivatives of $f$, $\partial_bV$ and $\partial_b\partial_xV$ as functions on $X$; (b) derivatives of $\hat{f}$ and $V$ as functions on $U \times X$; (c) derivatives of $\mathscr{F}$ as functions on $[0,1] \times U \times X$.

To prove (1), we need to bound  the norms of $\|D\hat f\|$ abd $\| D^2 f\|$.
Since $B \mapsto V(B,\cdot)$ is linear,  it is clear that by reducing the size of $U$, we can assume that 
$\norm{\partial_xV} < \frac{1}{10}$.
Then by Gr\"onwall's inequality and possibly reducing the size of $U$, there exists an absolute constant $c_0 > 0$ such that
\begin{equation}\label{eq 10.3}
\norm{\partial_x\mathscr{F}}< c_0 \max(1, \norm{\partial_xf}), \quad
\norm{\partial_b\mathscr{F}}< c_0  \norm{\partial_bV}. 
\end{equation}
Since   $\partial_b^2V \equiv 0$,   by  Gr\"onwall's inequality and \eqref{eq 10.3}, there exists an absolute constant $c_1 > 0$ such that
\begin{align*}
\norm{\partial_x^2\mathscr{F}} &\leq \norm{\partial_{x}^2f} + c_0 \norm{\partial_x^2V} \norm{\partial_x\mathscr{F}}^2 \leq  \norm{\partial_x^2f} + c_1 \norm{\partial_x^2V} \max(1, \norm{\partial_xf}^2), \\
\norm{\partial_b\partial_x\mathscr{F}} &\leq c_0 (\norm{\partial_b\partial_xV} \norm{\partial_x\mathscr{F}}  + \norm{\partial_x^2V} \norm{\partial_x \mathscr{F}} \norm{\partial_b\mathscr{F}})\\
&\leq  c_1(\norm{\partial_b\partial_xV}  + \norm{\partial_x^2V} \norm{\partial_bV}) \max(1, \norm{\partial_xf}),\\
\norm{\partial_b^2\mathscr{F}}&\leq 2 c_0 (\norm{\partial_b\partial_xV} \norm{\partial_b\mathscr{F}} + \norm{\partial_x^2V} \norm{\partial_b\mathscr{F}}^2)\leq c_1 (\norm{\partial_b\partial_xV} \norm{\partial_bV} + \norm{\partial_x^2V} \norm{\partial_bV}^2).
 \end{align*}
Note that by possibly reducing the size of $U$, we can ensure that 
$$
\norm{\partial_x^2V} < \min(\max(1, \norm{\partial_xf}^2)^{-1},  \norm{\partial_bV}^{-2}, \norm{\partial_bV}^{-1}).
$$
Thus there exists an absolute constant $c_2 > 0$ such that
\begin{equation*}
\norm{\hat{f}}_{C^2} \leq \norm{\mathscr{F}(1,\cdot)}_{C^2} < c_2 (1 + \norm{\partial_b\partial_xV})(1 + \norm{f}_{C^2} + \norm{\partial_bV}).
\end{equation*}
We conclude the proof of (1) by noticing that $\norm{D^i T} \lesssim \norm{D^i\hat{f}}$ for $i=1,2$.
\end{proof}

\section{} \label{app c}
\begin{proof}[Proof of Theorem \ref{ epsilon lm sv}]

We repeat the proof of \cite[Proposition 3.2]{BK} for $\epsilon$-light maps, instead of light maps, for some $\epsilon=\epsilon(n)$. 

Recall that the order of a cover $\mathcal{O}=\{O_k\}_{k\in K}$ is the supremum of all numbers $\#K'$ such that $\cap_{k\in K'} O_k \neq \emptyset$.
Let then $\mathcal{V}_0$ be a cover of $X:=[0,1]^n$ such that $\mathcal{V}_0$ does not admit an open refinement of order less than or equal to $n$. Take $\delta>0$ a Lebesgue number of the cover $\mathcal{V}_0$ and define $\epsilon:=\delta/2$.

Assume that $f\colon X \to Y$ is $\epsilon$-light. Let $T$ be some triangulation of $Y$ and denote by $\mathcal{U}=\{U_i\}_{i \in I}$ its open star cover. For every $i \in I$, $f^{-1}(U_i)$ can be written as a disjoint union of connected open sets; they form an open cover of $X$, denoted by $\mathcal{V}=\{V_j\}_{j \in J}$. For each $j \in J$, we let $\alpha(j)\in I$ be such that $V_j \subset f^{-1}(U_{\alpha(j)})$. By assumption, $f$ is $\epsilon$-light, hence we can choose $T$ fine enough such that each $V_j$ has diameter smaller than $2 \epsilon=\delta$. Therefore $\mathcal{V}$ is an open refinement of $\mathcal{V}_0$; in particular, any open refinement of $\mathcal{V}$ has order at least $n+1$.

We define $\mathrm{Ner}(\mathcal{U})$ as the collection of subsets $I' \subset I$ such that $\cap_{i \in I'} U_i \neq \emptyset$. We define $\mathrm{Ner}(\mathcal{V})$ in a similar way. Both $\mathrm{Ner}(\mathcal{U})$ and $\mathrm{Ner}(\mathcal{V})$ are simplicial $n$-complexes, and we identify them with their geometric realization. 

Given a partition of unity $\{\rho_i\}$ subordinate to $\mathcal{U}$, we get a homeomorphism $\rho \colon Y\to \mathrm{Ner}(\mathcal{U})$,   
while $\alpha$ induces a local homeomorphism $\phi\colon \mathrm{Ner}(\mathcal{V})\to \mathrm{Ner}(\mathcal{U})$. 
Then, the functions $\nu_{j}:=\chi_{V_{j}} \cdot (\rho_{\alpha(j)} \circ \phi)$, $j \in J$, define a partition of unity subordinate to $\mathcal{V}$. We let $\nu\colon X \to \mathrm{Ner}(\mathcal{V})$ be the associate map, where $\nu(x)$ has barycentric coordinates $(\nu_{j}(x))_j$.  
 By construction, the following diagram commutes:

$$
\xymatrix{
X \ar[r]^f \ar[d]_\nu & Y \ar[d]^\rho\\
 \mathrm{Ner}(\mathcal{V}) \ar[r]_\phi &  \mathrm{Ner}(\mathcal{U})
}
$$
Let us see that some $n$-simplex $\sigma$ of $ \mathrm{Ner}(\mathcal{V})$ has an interior point $\xi$ which is a stable value of $\nu\colon X \to  \mathrm{Ner}(\mathcal{V})$. 
Assume it is not the case; then we may form a set $\mathcal S$ by choosing one interior point from each $n$-simplex of $ \mathrm{Ner}(\mathcal{V})$ and perturb $\nu$ on a small neighbourhood of $\nu^{-1}(\mathcal S)$ to get $\nu'\colon X \to  \mathrm{Ner}(\mathcal{V})\backslash \mathcal S$. Denote by $p\colon\mathrm{Ner}(\mathcal{V})\backslash \mathcal S\to[ \mathrm{Ner}(\mathcal{V})]_{n-1}$ the radial projection to the $(n-1)$-skeleton of $\mathrm{Ner}(\mathcal{V})$, s.t. the barycentric coordinates of $\nu'':=p \circ \nu'\colon X \to [ \mathrm{Ner}(\mathcal{V})]_{n-1}$ are subordinate to $\mathcal{V}$. Pulling back the open star cover of $\mathrm{Ner}(\mathcal{V})$ by $\nu''$, we get a refinement of $\mathcal{V}$ of order at most $n$, a contradiction. Thus, $\rho^{-1}(\phi (\xi))$ is a stable value of $f$, which concludes.
\end{proof}

\end{document}